\documentclass{amsart}

\usepackage{graphicx}
\usepackage{amsmath,amssymb}
\usepackage{enumerate}
\usepackage{bbm}
\usepackage{tikz}
\usepackage{hyperref} 

\newtheorem{theorem}{Theorem}[section]
\newtheorem{lemma}[theorem]{Lemma}
\newtheorem{prop}[theorem]{Proposition}

\theoremstyle{definition}

\theoremstyle{remark}
\newtheorem{remark}[theorem]{Remark}

\numberwithin{equation}{section}

\newcommand{\abs}[1]{\lvert#1\rvert}

\begin{document}

\title[Fourth order cubic nonlinear Schr\"odinger equation]
{Well-Posedness and Ill-Posedness for the Fourth order cubic nonlinear Schr\"odinger equation in negative Sobolev spaces}

\author{Kihoon seong}

\begin{abstract}
We consider the Cauchy problem for the fourth order cubic nonlinear Schr\"odinger equation \eqref{eqn:fourth order NLS}. 
The main goal of this paper is to prove low regularity well-posedness and mild ill-posedness for \eqref{eqn:fourth order NLS}. We prove three results. First, we show that \eqref{eqn:fourth order NLS} is locally well-posed in $H^s\left(\mathbb{R}\right), s\geq -\frac{1}{2}$ using the Fourier restriction norm method. Second, we show that \eqref{eqn:fourth order NLS} is globally well-posed in $H^s\left(\mathbb{R}\right),s\geq -\frac{1}{2}$. To prove this, we use the $I$-method with the correction term strategy presented in Colliander-Keel-Staffilani-Takaoka-Tao \cite{CKSTT2003}. Finally, we prove that \eqref{eqn:fourth order NLS} is mildly ill-posed in the sense that the flow map fails to be locally uniformly continuous in $H^s(\mathbb{R}), s<-\frac{1}{2}$. Therefore, these results show that $s=-\frac{1}{2}$ is the sharp regularity threshold for which the well-posedness problem can be dealt with an iteration argument.
\end{abstract}

\maketitle


 \section{Introduction}\label{sec:introduction}
 \subsection{Fourth order cubic nonlinear Schr\"odinger equation}
 In this paper, we consider the following one-dimensional fourth order cubic nonlinear Schr\"odinger equation (4NLS): 
 \begin{align}\label{eqn:fourth order NLS}\tag{4NLS}
 \begin{cases}
 i\partial_tu=\partial_x^4u\pm \vert u \vert^2u, \quad(t,x)\in  \mathbb{R}\times \mathbb{R}\\
 u(x,0)=u_0(x)\in H^s\left(\mathbb{R}\right),
 \end{cases} 
 \end{align} 
 wherer $u$ is a complex-valued function. 
 The equation $\eqref{eqn:fourth order NLS}$ is also known as the biharmonic NLS and was studied in the context of stability of solitons in magnetic materials (for more physical background, see \cite{Karpman1996, KS2000}). The $\eqref{eqn:fourth order NLS}$ has been extensively studied in recent years. For instance, see \cite{BKS2000,GW2002,Segata2006,Pausader2009-2,OT2016,OST2018,OW2018,Kwak2018}. In the following, we make no distinction between the defocusing or focusing nature of \eqref{eqn:fourth order NLS} and hence we assume that it is defocusing, that
 is, with the + sign in \eqref{eqn:fourth order NLS}.

 It is well known that $\eqref{eqn:fourth order NLS}$ enjoys the scaling symmetry. More precisely, if $u(t,x)$ is a solution to $\eqref{eqn:fourth order NLS}$ with an initial condition $u_0$, then
 \begin{align}\label{eqn: scaling symmetry}
 u_{\lambda}(t,x):=\lambda^2u(\lambda^4t,\lambda x), \quad \lambda>0
 \end{align}
 is also a solution to $\eqref{eqn:fourth order NLS}$ with the $\lambda$-scaled initial condition $u_{0,\lambda}(x):=\lambda^2u_0(\lambda x)$. Associated with this scaling symmetry, there is the so-called scaling critical regularity $s_c:=-\frac{3}{2}$ such that the homogeneous $\dot{H}^{s_c}$-norm is invariant under the scaling symmetry \eqref{eqn: scaling symmetry}. In general, we have
 \begin{align}
 \Vert u_{0,\lambda} \Vert_{\dot{H}^s\left(\mathbb{R}\right) }=\lambda^{s+\frac{3}{2}}\Vert u_0 \Vert_{\dot{H}^s\left(\mathbb{R}\right) }.
 \end{align}
 As in the case of the classical NLS, \eqref{eqn:fourth order NLS} is a Hamiltonian PDE with the following Hamiltonian 
 \begin{align}\label{eqn:Hamiltonian}
 H(u(t))=\frac{1}{2}\int_{\mathbb{R}}\vert \partial_x^2 u(t) \vert^2 \,dx + \frac{1}{4} \int_{\mathbb{R}}\vert u(t) \vert^4\, dx.
 \end{align} 
 We also define the mass $M(u(t))$
 \begin{align}\label{eqn:mass}
 M(u(t))=\int_{\mathbb{R}}\vert u(t) \vert^2\,dx.
 \end{align}
 Under the flow of $\eqref{eqn:fourth order NLS}$, the Hamiltonian $\eqref{eqn:Hamiltonian}$ and the mass $\eqref{eqn:mass}$ are conserved. 
 
 The main goal of this paper is to solve low regularity well-posedness problem for $\eqref{eqn:fourth order NLS}$.
 A small modification of \cite{T1987} with the mass conservation provides global well-posedness of \eqref{eqn:fourth order NLS} in $L^2(\mathbb{R})$.
 Therefore, it is natural to ask whether the well-posedness also holds in negative Sobolev spaces between scaling critical space $H^{-\frac{3}{2}}\left(\mathbb{R}\right) $ and $L^2\left(\mathbb{R}\right)$.  
 We want to investigate which regularity is optimal for the local and global well-posedness for $\eqref{eqn:fourth order NLS}$. In this paper, we prove that $\eqref{eqn:fourth order NLS}$ is locally and globally well-posed in $H^{s}\left(\mathbb{R}\right), s\geq -1/2$.
 Also, we show that $\eqref{eqn:fourth order NLS}$ is midly ill-posed in the sense that the solution map fails to be locally uniformly continuous on $H^s(\mathbb{R}), s<-\frac{1}{2}$. This means that $s=-\frac{1}{2}$ is the sharp regularity threshold for which the well-posedness can be handled by an iteration argument.

 \begin{remark}
 	The one-dimensional cubic NLS is given by
 	\begin{align}\label{eqn: cubic NLS,focusing defocusing}\tag{NLS}
 	\begin{cases}
 	i\partial_tu=\partial_x^2u\pm \vert u \vert^2u, \quad(t,x)\in  \mathbb{R}\times \mathbb{R}\\
 	u(x,0)=u_0(x)\in H^s\left(\mathbb{R}\right).
 	\end{cases} 
 	\end{align}
 	In \cite{T1987}, Tsutsumi proved that  \eqref{eqn: cubic NLS,focusing defocusing} is globally well-posed in $L^2(\mathbb{R})$. We note that
 	the equation \eqref{eqn: cubic NLS,focusing defocusing} admits the Galilean invariance: if $u$ is a solution of $\eqref{eqn: cubic NLS,focusing defocusing}$ with initial data $u_0$, then
 	\begin{align} 
 	u_v(t,x)=e^{ix\cdot v}e^{-it\vert v \vert^2}u(t,x-2vt)
 	\end{align}   
 	is also a solution to the same equation $\eqref{eqn: cubic NLS,focusing defocusing}$ with initial data $e^{ix\cdot v }u_0(x)$. As a consequence of the Galilean invariance, the flow map cannot be locally uniformly continuous in $H^s, s<0$ (i.e. mild ill-posedness). We refer to \cite{KPV2001, CCT2003} for more details.
 	In view of the failure of local uniform continuity, one can observe that in negative Sobolev spaces it is impossible to prove well-posedness of \eqref{eqn: cubic NLS,focusing defocusing} via a contraction argument. As for $\eqref{eqn:fourth order NLS}$, thanks to the lack of the Galilean invariance, there is a hope to prove local well-posedness of \eqref{eqn:fourth order NLS} by a contraction argument in negative Sobolev spaces.
 	
 \end{remark}

 \begin{remark}
 	The fourth order cubic nonlinear Schr\"odinger equation on $\mathbb{T}$ is given by	
 	\begin{align}\label{eqn:on Torus 4NLS}
 	\begin{cases}
 	i\partial_tu=\partial_x^4u\pm \vert u \vert^2u, \quad(t,x)\in  \mathbb{R}\times \mathbb{T}\\
 	u(x,0)=u_0(x)\in H^s\left(\mathbb{T}\right).
 	\end{cases} 
 	\end{align}
 	In \cite[Appendix A]{OT2016}, Oh and Tzvetkov proved that \eqref{eqn:on Torus 4NLS} is globally well-posed in $L^2(\mathbb{T})$. However, below $L^2(\mathbb{T})$ Oh and Wang \cite{OW2018} proved a nonexistence result for \eqref{eqn:on Torus 4NLS}. Despite this ill-posedness, by considering the renormalized cubic 4NLS

 	\begin{align}\label{eqn:on Torus wick ordered 4NLS}
 	\begin{cases}
 	i\partial_tu=\partial_x^4u\pm \left( \vert u \vert^2 -\frac{1}{\pi} \int_\mathbb{T} \vert u \vert^2 \,dx \right)u, \quad(t,x)\in  \mathbb{R}\times \mathbb{T}\\
 	u(x,0)=u_0(x)\in H^s\left(\mathbb{T}\right), 
 	\end{cases}  
 	\end{align}
 	Oh and Wang \cite{OW2018} and Kwak \cite{Kwak2018} showed the well-posedness for $\eqref{eqn:on Torus wick ordered 4NLS}$ under negative Sobolev spaces. More precisely, by using short-time Fourier restriction norm method, 
 	Oh and Wang \cite{OW2018} showed the existence of a global solution $u \in C\left(\mathbb{R}; H^s\left(\mathbb{T}\right)\right)$ to the renormalized cubic 4NLS $\eqref{eqn:on Torus wick ordered 4NLS}$ for any initial data $u_0\in H^s\left(\mathbb{T}\right), s\in \left(-\frac{9}{20},0 \right)$.
 	Moreover, by exploiting an infinite iteration of normal form reductions, they showed that the renormalized cubic 4NLS $\eqref{eqn:on Torus wick ordered 4NLS}$ is globally well-posed in $H^s\left(\mathbb{T}\right)$ for $s\in \left(-\frac{1}{3},0\right)$ with enhanced uniqueness. Later, by an adaptation of Takaoka and Tsutsumi’s argument \cite{TT2004}, Kwak \cite{Kwak2018} proved that the renormalized cubic 4NLS $\eqref{eqn:on Torus wick ordered 4NLS}$ is locally well-posed in $H^s\left( \mathbb{T} \right)$ for $-\frac{1}{3}\leq s <0$. This result extends
 	local well-posedness of \eqref{eqn:on Torus wick ordered 4NLS} to the endpoint regularity $s=-\frac{1}{3}$.

 \end{remark}

 \subsection{Local well-posedness}
 In this subsection, we present our first main result which is local-wellposedness in $H^s(\mathbb{R}), s\geq-\frac{1}{2}$. Before we state Theorem \ref{thm: local well-posedness}, we briefly look into the local well-posedness of \eqref{eqn:fourth order NLS} in $L^2(\mathbb{R})$.
 In Proposition \ref{St}, the Strichartz estimates associated with the (linear) biharmonic Schr\"odinger equation are given by
 \begin{align}\label{st}
 \left\| D^{\frac{\alpha}{2}\left(1-\frac{2}{r}\right)} e^{it\partial_x^4}u_0  \right\|_{L_t^qL_x^r\left(\mathbb{R}\times\mathbb{R}\right)}\lesssim_{q,r} & \left\| u_0 \right\|_{L_x^2\left(\mathbb{R}\right)},
 \end{align}
 for $0\leq \alpha \leq 1 $, $ r \geq 2, q\geq\frac{8}{(1+\alpha)}$ and $\frac{4}{q}+\frac{1+\alpha}{r}=\frac{1+\alpha}{2}$. In contrast to the case of linear Schr\"odinger equation, one can see that derivative gains occur in the estimates \eqref{st}, thanks to the stronger dispersive effect in the high frequency mode. See Section \ref{sec: linear estimate} for more details.
 Using the Strichartz estimates \eqref{st}, one can easily show the local well-posedness of \eqref{eqn:fourth order NLS} for the regular initial data $u_0 \in H^s\left(\mathbb{R}\right)$, $s\geq 0 $ as in \cite{T1987}. Moreover, global well-posedness of \eqref{eqn:fourth order NLS} in $H^s\left(\mathbb{R}\right), s\geq 0$ follows from the conservation laws.

 The first main result of this paper is the following local well-posedness result in $H^{s}(\mathbb{R}), s\geq -\frac{1}{2}$.
 \begin{theorem}\label{thm: local well-posedness}
 	Let $s\geq -\frac{1}{2}$. Then, \eqref{eqn:fourth order NLS} is locally well-posed in $H^s(\mathbb{R})$. More precisely, for any $u_0 \in H^{s}\left(\mathbb{R}\right)$, there exists $T=T\left(\Vert u _0 \Vert_{H^s\left(\mathbb{R}\right)}\right)>0$ and 
 	a solution $u\in C([0,T];H^s(\mathbb{R}) )$ to \eqref{eqn:fourth order NLS}. This solution is unique in $X^{s,\frac{1}{2}+}$-space depending on the choice of a time cutoff function. Moreover, the flow map from data to solutions is locally Lipschitz continuous.	  
 \end{theorem} 
 \begin{remark} 
 	In our formulation, the uniqueness depends on the choice of the cutoff function in the Duhamel formulation \eqref{Duha}. In particular, given two solutions to the Duhamel formulation \eqref{Duha} with different cutoff $\eta$ and $\widetilde{\eta}$, both belonging to $X^{s,\frac{1}{2}+}$, they agree on $[0,\delta]$ where $\delta>0$ denotes the shorter one of the local existence times of these two solutions.

 \end{remark}
 To prove Theorem \ref{thm: local well-posedness}, we use the contraction mapping argument. 
 A natural choice of the iteration space is the $X^{s,b}$-space. The $X^{s,b}$-space were simultaneously introduced by Klainerman and Machedon in the context of wave equations and Bourgain \cite{Bourgain1, Bourgain2} in the context of Schr\"odinger equations and the KdV equation.
 In the contraction mapping argument, the main part is to obtain a suitable trilinear estimate (Proposition \ref{prop:t1}). In Proposition \ref{prop:t1}, we prove that the trilinear estimate holds for $s\geq -\frac{1}{2}$.  
 We point out that below $L^2(\mathbb{R})$, nonlinear interactions generate a loss of derivatives in the trilinear estimate  and hence the Strichartz estimates \eqref{st} are not enough to get around this loss of derivatives. To deal with this loss, we strongly use the nature of $X^{s,b}$-space which captures dispersive smoothing effects. Observing the non-resonant case in the nonlinear interactions, one can detect a dispersive smoothing effect, which is crucial to overcome the loss of derivatives. By exploiting this dispersive smoothing effects, one can prove the trilinear estimate in the negative regularity regime (Proposition \ref{prop:t1}).
 However, in Remark \ref{rem: trilienar estimate fail remark.}, we present a counterexample for the trilinear estimate when $s<-\frac{1}{2}$ (we point out that the failure of the trilinear estimate is due to the resonant interaction of high-high-high to high). Therefore, $s=-1/2$ is the optimal regularity for the trilinear estimate to hold. 
 More details are presented in Remark \ref{rem: trilienar estimate fail remark.}.


 \subsection{Global well-posedness}
 In this subsection, we present our second main result which is global-wellposedness in $H^s(\mathbb{R}), s\geq-\frac{1}{2}$.
 
 \begin{theorem}\label{thm: global well-posedness with correction term}
 	Let $s\geq-\frac{1}{2}$. The \eqref{eqn:fourth order NLS} is globally well-posed in $H^s(\mathbb{R})$. 
 \end{theorem}
 
 We observe that $L^2$ solutions of  $\eqref{eqn:fourth order NLS}$ satisfy the mass conservation law $\eqref{eqn:mass}$. This conservation law allows us to extend the local solution  to the global one for $L^2$ data.
 However, it is non-trivial to obtain global well-posedness below $L^2$, due to the absence of a conservation law. We make use of the $I$-method to obtain global well-posedness in $H^s,s<0$. The $I$-method was introduced by Colliander-Keel-Staffilani-Takaoka-Tao \cite{CKSTT2001,CKSTT2002,CKSTT2003}. We briefly describe their approach.
 We introduce a radial $C^{\infty}$, monotone multiplier $m$, taking values in $[0,1]$, and
 \begin{align}
 m(\xi):= 
 \begin{cases}
 1,   & |\xi|<N \\ 
 \left(\frac{|\xi|}{N}\right)^s,  & |\xi|>2N,
 \end{cases}
 \end{align}
 Here, $N$ is a large parameter to be determined later. We define an operator $I$ by 
 \begin{align*}
 \widehat{Iu}\left(\xi\right):=m\left(\xi\right) \widehat{u}\left( \xi \right).    
 \end{align*}
 Note that we have the estimate
 \begin{align*}
 \| u\|_{H^s}\lesssim \| Iu \|_{L^2}\lesssim N^{-s} \| u\|_{H^s}.   
 \end{align*}
 Thus, the operator $I$ acts as the identity for low frequencies and as a smoothing operator of order $\abs{s}$ on high frequencies. That is, it maps $H^s$ solutions to $L^2$. 
 Observe that $Iu\to u$ as $N \to \infty$. Therefore, it is intuitively plausible that when $N$ is large enough, $\Vert Iu(t) \Vert_{L_x^2}$ almost follows mass conservation law i.e. by regularizing a solution $u$ to $Iu$, $Iu$ approximately satisfies mass conservation law. 
 According to this idea, we will prove Lemma \ref{lem:almost conservation law} (the almost conservation law).
 Indeed, Lemma \ref{lem:almost conservation law} shows that there is a tiny increment in $\left\| Iu(t) \right\|_{L^2}$ as $t$ evolves from 0 to $\delta$ ($\delta \lesssim 1$) if $N$ is very large.

 The basic structure of
 our argument to prove Lemma \ref{lem:almost conservation law} (the almost conservation law) follows the argument introduced in \cite{CKSTT2003}.
 As in \cite{CKSTT2003}, we carry out our energy estimate on a modified energy \eqref{eqn : modified energy E_I^4}. This introduction of a modified energy is essential to exhibit a hidden dispersive smoothing effect.

 \subsection{Mild ill-posedness below $H^{-\frac{1}{2}   }(\mathbb{R}) $}
 In this subsection, we discuss our third main result which is the mild ill-posedness. We show that $\eqref{eqn:fourth order NLS}$ is mildly ill-posed in the sense that the data-to-solution map fails to be locally uniformly continuous on $H^s(\mathbb{R}), s<-\frac{1}{2}$. This implies that $s=-\frac{1}{2}$ is the optimal regularity threshold for which the well-posedness can be dealt with an iteration argument. Our method is inspired by Christ-Colliander-Tao \cite{CCT2003}
 \begin{theorem}\label{thm: ill-posedness}
 	Let $-\frac{15}{14}<s<-\frac{1}{2}$. Then the solution map of the fourth order NLS equation $\eqref{eqn:fourth order NLS}$ fails to be locally uniformly continuous in $H^s(\mathbb{R})$. More precisely, there exists $\varepsilon_0>0$ such that for any $\delta>0,T>0$ and $0<\varepsilon<\varepsilon_0$, there are two solutions $u,v$ to $\eqref{eqn:fourth order NLS}$ such that
 	\begin{align}
 	\Vert u(0) \Vert_{H^s_x}, \Vert v(0) \Vert_{H^s_x}\lesssim \varepsilon,\\
 	\Vert u(0)-v(0) \Vert_{H^s_x}\lesssim \delta,\\
 	\sup\limits_{0\leq t \leq T} \Vert u(t)-v(t) \Vert_{H^s_x}\gtrsim \varepsilon.
 	\end{align}		 
 \end{theorem}
 We present a heuristic idea of the proof of Theorem \ref{thm: ill-posedness}.
 The key idea is to construct an approximate solution to $\eqref{eqn:fourth order NLS}$ by using the solution of \eqref{eqn: cubic NLS,focusing defocusing} and 
 use this approximate solution to transfer the mild ill-posedness result (Theorem \ref{thm: ill-posedness defocsuing NLS}) of \eqref{eqn: cubic NLS,focusing defocusing} to \eqref{eqn:fourth order NLS}.
 
 We point out that time-localized solutions to \eqref{eqn: cubic NLS,focusing defocusing} have spacetime Fourier transform near the parabola $\tau=\xi^2$, but time-localized solutions to \eqref{eqn:fourth order NLS} have spacetime Fourier transform near the quartic $\tau=\xi^4$.
 Choose an $N\gg 1 $. Let $u(t,x)$ be a linear solution to $\left(i\partial_t-\partial_x^4\right)u=0$ with $u(0)=u_0$. We use the following change of variables
 \begin{align*}
 \xi:=N+\frac{\xi'}{\sqrt{6}N}.
 \end{align*}
 Then, $\tau=\xi^4$ leads to $\tau=N^4+\frac{4}{\sqrt{6}}N^2\xi'+\tau'$ where
 \begin{align*}
 \tau'=(\xi')^2+\frac{2}{3\sqrt{6}N^2}\xi'^3+\frac{1}{36N^4}\xi'^4.
 \end{align*}
 By using these change of variables, we have
 \begin{equation} 
 \begin{split}
 u(t,x)=&\int_{\mathbb{R}\times \mathbb{R}}e^{it\tau}e^{ix\xi}\widehat{u}_0(\xi)\,d\xi d\tau\\
 =&\int_{\mathbb{R}\times\mathbb{R}}e^{it\left(N^4+\frac{4}{\sqrt{6}}N^2\xi'+\tau'  \right)+ix\left( N+\frac{\xi'}{\sqrt{6}N} \right) }\widehat{u}_0(\xi')\,d\xi' d\tau'\\
 =&e^{iN^4t+ixN}\int_{\mathbb{R}\times\mathbb{R}}e^{i\tau't+i\xi'\left(\frac{x}{\sqrt{6}N}+\frac{4}{\sqrt{6}}N^2t  \right) }\widehat{u}_0(N+\frac{\xi'}{\sqrt{6}N})\,d\tau'd\xi'.
 \end{split}
 \end{equation} 
 Thus, for $\abs{\xi'}\ll N $ , this change of variables converts $\tau=\xi^4 $  to an approximate $\tau'\approx \xi'^2$. Therefore, we obtain an approximate solution 
 \begin{align*}
 u(t,x)\approx e^{iN^4t+ixN}v\left(t, \frac{x}{\sqrt{6}N}+\frac{4}{\sqrt{6}}N^2 t \right).
 \end{align*} 
 to the fourth order linear Schr\"odinger equation where $v(t,x)$ solves the linear Schr\"odinger equation.
 Indeed, if $v$ solves \eqref{eqn: cubic NLS,focusing defocusing}, then the function
 \begin{align}\label{eqn:approximate fourth order NLS}
 u(t,x):= e^{iN^4t+ixN}v\left(t, \frac{x}{\sqrt{6}N}+\frac{4}{\sqrt{6}}N^2 t \right)
 \end{align}
 is an approximate solution to $\eqref{eqn:fourth order NLS}$ . We present the details in Section \ref{sec:ill-posedness}.

 \textbf{Organization of paper.} The rest of the paper is organized as follows: In Section \ref{sec: linear estimate}, we collect the estimates that capture the linear dispersive effects.  
 In Sections \ref{sec: local well-posedness} and \ref{sec: Global well-posedness and correction term strategy}, we prove that $\eqref{eqn:fourth order NLS}$ is locally and globally well-posed in $H^{s}\left(\mathbb{R}\right), s\geq -1/2$ respectively.
 In Section \ref{sec:ill-posedness} we show that $\eqref{eqn:fourth order NLS}$ is mildly ill-posed in the sense that the solution map fails to be locally uniformly continuous in $H^s(\mathbb{R}), s<-\frac{1}{2}$.

 \textbf{Notation.} We use $A\lesssim B $ if $A\leq CB$ for some $C>0$ and $A=O(B)$ if $A\lesssim B$. We use $X\sim Y$ when $X\lesssim Y$ and $Y\lesssim X$. Moreover, we use $A\ll B$ if $A\leq \frac{1}{C}B$, where $C$ is a sufficiently large constant. We also write $A^{\pm}$ to mean $A^{\pm \varepsilon}$ for any $\varepsilon>0$.

 Given $ p\geq 1 $, we let $p'$ be the H\"older conjugate of $p$ such that $ \frac{1}{p}+\frac{1}{p'}=1$. We denote $L^p=L^p\left(\mathbb{R}^d\right)$ be the usual Lebesgue space. We also define the Lebesgue space $L^q\left(I,L^r\right)$ be the space of measurable functions from an interval $I\subset \mathbb{R}$ to $L^r$ whose $L^q\left(I,L^r\right)$ norm is finite, where
 \begin{align*}
 \Vert u\Vert_{L^q\left(I,L^r \right)}=\left(    \int_I \Vert u(t) \Vert_{L^r}^q       \right)^{\frac{1}{q}}.
 \end{align*}
 We may write $L_t^qL_x^r\left(I\times \mathbb{R}\right)$ instead of  $L^q\left(I,L^r\right)$.

 We denote the space time Fourier transform of $u(t,x)$ by $\widetilde{u}(\tau,\xi)$ or $\mathcal{F}u$
 \begin{align*}
 \widetilde{u}(\tau,\xi)=\mathcal{F}u(\tau,\xi)=\int e^{-it\tau-ix\xi}u(t,x)\,dtdx.
 \end{align*}
 On the other hand, the space Fourier transform  of $u(t,x)$ is denoted by
 \begin{align*}
 \widehat{u}(t,\xi)=\mathcal{F}_xu(t,\xi)=\int e^{-ix\xi}u(t,x)\,dx.
 \end{align*}
 The fractional differential operators are defined by
 \begin{align*}
 \widehat{D^{\alpha}u}(\xi)&=\vert \xi \vert^{\alpha}\widehat{u}(\xi),\\
 \widehat{\langle D \rangle^\alpha u}(\xi)&=\langle \xi \rangle^\alpha \widehat{u}(\xi), \quad \alpha \in\mathbb{R},  
 \end{align*}
 and the biharmonic Schr\"odinger semigroup is defined by
 \begin{align}
 e^{-it\partial_x^4}g=\mathcal{F}_x^{-1}(e^{-it\vert \xi \vert^4}\mathcal{F}_xg)
 \end{align}
 for any tempered distribution $g$.
 
 Lastly, for each dyadic number $N\in 2^{\mathbb{Z}}$, the Littlewood-Paley projection $P_N,P_{\leq N}, P_{>N}$ are smoothed out projections to the regions $\vert \xi \vert \sim N, \vert \xi \vert\leq 2N, \vert \xi \vert >N $ respectively. More precisely, let $\varphi(\xi)$ be a bump function supported on the set $\left\{\vert \xi \vert \leq 2 \right\}$ which equals 1 on the unit ball $\left\{\vert \xi \vert \leq 1 \right\}$.
 For any dyadic number $N=2^k,k\in \mathbb{Z}$, we define the following Littlewood-Paley projections:
 \begin{align*}
 \widehat{P_{\leq N}u}(\xi)=&\varphi\left(\xi/N\right)\widehat{u}(\xi),\\
 \widehat{P_{> N}u}(\xi)=&\left(1-\varphi\left(\xi/N\right)\right)\widehat{u}(\xi),\\
 \widehat{P_ Nu}(\xi)=&\left(\varphi\left(\xi/N\right)-\varphi\left(2\xi/N\right)\right) \widehat{u}(\xi).
 \end{align*}
 They commute with derivative operators $D^\alpha, \langle D \rangle^\alpha$ and the semigroup $e^{it\partial_x^4}$. We also use the notation $u_N=P_Nu$ if there is no confusion. 
 Furthermore,
 they obey the following easily verified 
 Bernstein inequalities for $1\leq p \leq q \leq \infty$:
 \begin{align}
 \Vert P_{\leq N} f \Vert_{L_x^q(\mathbb{R})}&\lesssim_{q,p} N^{\frac{1}{p}-\frac{1}{q} } \Vert P_{\leq N} f \Vert_{L_x^p}\\
 \Vert P_{ N} f \Vert_{L_x^q(\mathbb{R})}&\lesssim_{q,p} N^{\frac{1}{p}-\frac{1}{q} } \Vert P_{ N} f \Vert_{L_x^p}\label{Bs}
 \end{align}
 We also use $a+$ (and $a-$) to denote $a+\eta$ (and $a-\eta,$ respectively) for arbitrarily small $\eta\ll 1$.

 \textbf{Acknowledgements.} The author would like to appreciate his advisor Soonsik Kwon for helpful discussion and encouragement. The author is also grateful to Chulkwang Kwak for  pointing out an unclear portion in the proof of Lemma \ref{lem:almost conservation law} and helpful discussion. The author is also grateful to Justin forlano for helpful discussion related to writing introduction part. 
 The author is also grateful to the anonymous referee for their helpful comments that have improved the presentation of this paper. The author is partially supported by NRF-2018R1D1A1A09083345 (Korea).

 \section{Linear estimates}\label{sec: linear estimate}
 
 In this section, we collect the estimates that capture the linear dispersive effects. As in NLS or KdV, the proofs are fairly standard, but 
 there are no places written for our purpose. Therefore, the proofs are self-contained. We follow the argument in \cite{T2016}

 \subsection{Dispersive estimate}
 In this subsection, we prove the following dispersive estimates.
 \begin{lemma}\label{prop:decay}
 	For any $\alpha \in [0,1]$ and $t\in \mathbb{R}$, we have
 	\begin{align}\label{eqn: dispersive estimate L^1 L^infty}
 	\left\| D^{\alpha}e^{it\partial_x^4}u_0\right\|_{L_x^{\infty}}\lesssim C_\alpha \left|t \right|^{-\frac{\alpha+1}{4}}\left\| u_0 \right\|_{L_x^1}.
 	\end{align}
 \end{lemma}   
 \begin{remark}
 	We observe that obtaining the $L^1 \to L^{\infty}$ estimate is more involved since there is no explicit formula for linear biharmonic Schr\"odinger operator.
 	It is important to notice that the dispersion is stronger for high frequencies and weaker for low frequencies compared to the usual linear Schr\"odinger evolution. 
 \end{remark}
 Before we prove Lemma \eqref{prop:decay}, we first prove the following oscillatory integral estimate.
 \begin{lemma}
 	Let $\phi$ be a smooth cutoff function supported in $\left\{ \frac{1}{2}\leq |\xi|\leq 2 \right\}$. Then for each $x,t \in \mathbb{R}$, we have
 	\begin{equation}
 	\label{os1}
 	\left|\int_{\mathbb{R}}e^{it\xi^4+ix\xi}\phi\left(\xi\right)  \,d\xi\right| \lesssim  \left\langle t \right\rangle^{-\frac{1}{2}},
 	\end{equation}
 	and if $|x|\gg |t|$ or $\vert  x \vert \ll \vert t \vert $, then
 	\begin{equation}
 	\label{os2}
 	\left|\int_{\mathbb{R}}e^{it\xi^4+ix\xi}\phi\left(\xi\right)\,d\xi  \right| \lesssim \frac{1}{\max\left(|x|^2,|t|^2\right)}.  
 	\end{equation}
 \end{lemma}
 
 \begin{proof}
 	Observe that the second derivative of the phase function on $\left\{ \frac{1}{2}\leq |\xi|\leq 2\right\}$ is given by
 	\begin{align*}
 	\left| \partial_{\xi}^2\left(\xi^4+\frac{\xi x}{t} \right)\right|=12|\xi|^2   \gtrsim 1.
 	\end{align*}
 	Hence, by the Vander Corput lemma, we have	
 	\begin{align*}
 	\left|\int_{\mathbb{R}}e^{it\xi^4+ix\xi}\phi\left(\xi\right)  \,d\xi\right| \lesssim  \left\langle t \right\rangle^{-\frac{1}{2}}.    
 	\end{align*}
 	For $|x| \gg |t|$ or $|x| \ll |t|$, the first derivative of the phase function is given by 
 	\begin{equation}
 	\label{os}
 	\left| \partial_{\xi}\left(\xi^4t+\xi x  \right)\right|=\left| 4\xi^3t+x \right|   \gtrsim \text{max}\left(|x|,|t|\right).    
 	\end{equation}
 	By two integration by parts and the bound (\ref{os}), we have
 	\begin{align*}
 	\left|\int_{\mathbb{R}}e^{it\xi^4+ix\xi}\phi\left(\xi\right)\,d\xi  \right|&=\left\vert\int_{\mathbb{R}}\frac{\phi\left(\xi\right)   }{4it\xi^3+ix}\partial_{\xi}e^{it\xi^4+ix\xi}\,d\xi \right\vert\\
 	&=\left\vert\int_{\mathbb{R}}\partial_{\xi}\left(\frac{\phi\left(\xi\right)   }{4it\xi^3+ix}\right)e^{it\xi^4+ix\xi}\,d\xi\right\vert\\
 	&=\left\vert \int_{\mathbb{R}}\partial_{\xi}\left(\frac{1}{4it\xi^3+ix}\partial_{\xi} \left(\frac{\phi\left(\xi\right)    }{4it\xi^3+ix}\right)\right) e^{it\xi^4+ix\xi}\,d\xi\right\vert\\
 	&\lesssim \sup\limits_{\frac{1}{2}\leq|\xi|\leq 2}\left| \partial_{\xi}\left(\frac{1}{4it\xi^3+ix}\partial_{\xi} \left(\frac{\phi\left(\xi\right)    }{4it\xi^3+ix}\right)\right)    \right| \\
 	&\lesssim  \left(   \frac{1}{\left(4\xi^3t+x \right)^2}+\frac{|t|}{\left| 4\xi^3t+x \right|^3} +\frac{t^2}{\left(4\xi^3t+x \right)^{4}}      \right)\\
 	&\lesssim  \frac{1}{\text{max}\left( |x|^2,|t|^2\right)}.
 	\end{align*}	
 	
 	\noindent
 	Hence, we obtain the desired result.
 \end{proof}

 \begin{proof}[Proof of Lemma \ref{prop:decay}]
 	Fix $\alpha \in [0,1]$. 
 	We define the kernel $K_t\left(x\right)$ of the operator $D^{\alpha}e^{it\partial_x^4}$ as follows:
 	\begin{align*}
 	K_t\left(x\right)=\mathcal{F}^{-1}\left(\vert \xi \vert^{\alpha} e^{it \vert \xi \vert^4} \right)\left(x\right).
 	\end{align*}
 	The functions $K_t\left(x\right)$ are defined a priori as distributions. Later, we will show that they are in fact functions. For any tempered distribution $T \in \mathcal{S}^{*}$ and invertible linear transformation $L$, we have $\mathcal{F}\left(T\circ L\right)=\vert \det L \vert^{-1}\left(\mathcal{F}T\right)\circ \left(L^{*}\right)^{-1}$. Therefore we obtain
 	\begin{align*}
 	K_t\left(x\right) =t^{-\frac{\alpha+1}{4}}K_1\left(\frac{x}{t^{\frac{1}{4}}} \right).    
 	\end{align*}  	
 	Note that	   
 	\begin{align*}
 	D^{\alpha}e^{it\partial_x^4}u_0=\mathcal{F}^{-1}\left(|\xi|^{\alpha}e^{it|\xi|^4}\right)*u_0=K_t*u_0.    
 	\end{align*}
 	Hence, by Young's inequality, it is enough to show that $ K_1(x)$ is a bounded function.	
 	For any Schwartz function $f\in \mathcal{S}\left(\mathbb{R}\right)$, we have $ P_{\leq 1}f+\sum\limits_{k=1}^n P_kf \to f$ in $\mathcal{S}\left(\mathbb{R}\right)$. Here $P_kf$ is the standard Littlewood-Paley projectors to the regions $\left\{ 2^{k-1}\leq \vert \xi \vert \leq 2^{k+1} \right\}$. Therefore, we have
 	\begin{align*}
 	\left\langle K_1, f \right \rangle =& \left \langle \mathcal{F}^{-1}\left(\vert \xi \vert^{\alpha}e^{i\xi^4} \right) , f \right \rangle\\
 	=&\left \langle \mathcal{F}^{-1}\left(\vert \xi \vert^{\alpha}e^{i\xi^4} \right) , P_{\leq 1}f \right \rangle+\lim\limits_{n\to \infty} \sum\limits_{k=1}^n \left \langle \mathcal{F}^{-1}\left(\vert \xi \vert^{\alpha}e^{i\xi^4} \right) , P_kf \right \rangle\\
 	=& \left \langle \mathcal{F}^{-1}\left(\vert \xi \vert^{\alpha}e^{i\xi^4} \varphi\left(\xi\right) \right) , f \right \rangle+\lim\limits_{n\to \infty} \sum\limits_{k=1}^n \left \langle \mathcal{F}^{-1}\left(\vert \xi \vert^{\alpha}e^{i\xi^4}\psi\left(\frac{\xi}{2^k}\right) \right) , f \right \rangle.
 	\end{align*}
 	Here $\varphi$ is a bump function supported on $\left\{ \vert \xi \vert \leq 2 \right\}$ which equals 1 on $\left\{\vert \xi \vert \leq 1 \right\}$ and $\psi$ is the function given by $\psi\left(\xi\right):=\varphi\left(\xi\right)-\varphi\left(2\xi\right)$.
 	Therefore, we can identify $K_1\left(x\right)$ as the distributional limits: 
 	\begin{align*} 
 	K_1\left(x\right)=\mathcal{F}^{-1}\left(\vert \xi \vert^{\alpha}e^{i\xi^4}\varphi\left(\xi\right) \right)\left(x\right)+\lim\limits_{n \to \infty}\sum\limits_{k=1}^n \mathcal{F}^{-1}\left(\vert \xi \vert^{\alpha} e^{i\xi^4}\psi\left(\frac{\xi}{2^k}\right)   \right)\left(x\right).
 	\end{align*}
 	Obviously, the first summand is a bounded function. The second term is
 	\begin{align*}
 	\sum\limits_{k=1}^{\infty}\int_{\mathbb{R}}e^{i\xi^4+ix\xi}|\xi|^{\alpha}\psi\left( \frac{\xi}{2^{k}}\right) \,d\xi=\sum\limits_{k=1}^{\infty} 2^{k\left(1+\alpha\right)}\int_{\mathbb{R}}e^{i2^{4k}\xi^4+i 2^kx \xi}|\xi|^{\alpha}\psi\left(\xi\right) \,d\xi. 
 	\end{align*}
 	By using the oscillatory integral estimate \eqref{os1} and \eqref{os2}, we have
 	\begin{align*}
 	&\sum\limits_{k=1}^{\infty} 2^{k\left(1+\alpha\right)}\int_{\mathbb{R}}e^{i2^{4k}\xi^4+i 2^kx \xi}|\xi|^{\alpha}\psi\left(\xi\right) \,d\xi\\
 	&\lesssim \sum\limits_{\substack{2^{3k}\gg |x| \\ k\geq1}}2^{2k}\cdot 2^{-8k}+\sum\limits_{\substack{2^{3k}\ll |x| \\ k\geq 1}}2^{2k}\cdot 2^{-2k}\left|x \right|^{-2}+\sum\limits_{\substack{2^{3k} \sim |x| \\ k\geq 1}} 2^{2k}\cdot 2^{-2k}\\
 	&\lesssim  \sum\limits_{k=1}^{\infty}2^{-6k}+\sum\limits_{\substack{2^{3k} \sim |x| \\ k\geq 1}}1\\
 	&\lesssim  1.  
 	\end{align*}  
 	
 	\noindent
 	This completes the proof of Lemma \ref{prop:decay}.
 \end{proof}
 
 \begin{lemma}[Dispersive decay estimate]\label{prop: dispersive estimate L^r } For any $r\geq 2$ and $\alpha \in [0,1]$, we have
 	\begin{align}
 	\left\|D^{\alpha\left(1-\frac{2}{r}\right)} e^{it\partial_x^4} u_0  \right\|_{L_x^{r}\left(\mathbb{R}\right)}\lesssim \frac{1}{\left|t \right|^{\frac{1+\alpha}{4}\left(1-\frac{2}{r}\right)}}\left\| u_0 \right\|_{L_x^{r'}\left(\mathbb{R}\right)}.    
 	\end{align}
 \end{lemma}
 
 \begin{proof}
 	Note that 
 	\begin{align*}
 	\left\| D^{\alpha}e^{it\partial_x^4 }u_0 \right\|_{L_x^{\infty}}&\lesssim \frac{1}{\left| t \right|^{\frac{1+\alpha}{4}}} \left\|u_0 \right\|_{L_x^1},\\
 	\left\|e^{it\partial_x^4}u_0 \right\|_{L_x^2}&\lesssim \left\| u_0 \right\|_{L_x^2}.
 	\end{align*}
 	Consider the analytic family of operators 
 	\begin{align*}
 	T_z=\left|t \right|^{z\frac{\alpha+1}{4}}D^{\alpha z}e^{it\partial_x^4},    
 	\end{align*}
 	where $z=x+iy, x\in \mathbb{R}, y\in [0,1]$.
 	
 	Therefore, complex interpolation between the lines $\text{Re}(z)=0$ and $\text{Re}(z)=1$ yields the theorem.
 \end{proof}

 \subsection{Strichartz estimates}
 In this subsection, we present the following Strichartz estimates.
 \begin{lemma}[Strichartz estimates]\label{St} For any $\alpha \in [0,1]$ , we call a triplet $\left(q,r,\alpha \right)$ admissible exponents if $ r \geq 2, q\geq\frac{8}{(1+\alpha)}$ and $\frac{4}{q}+\frac{1+\alpha}{r}=\frac{1+\alpha}{2}$. Then, for any admissible exponents $\left(q,r,\alpha\right)$ and $\left(\tilde{q},\tilde{r},\alpha \right)$, we have
 	
 	\begin{align}
 	\left\| D^{\frac{\alpha}{2}\left(1-\frac{2}{r}\right)} e^{it\partial_x^4}u_0  \right\|_{L_t^qL_x^r\left(\mathbb{R}\times\mathbb{R}\right)}\lesssim_{q,r} & \left\| u_0 \right\|_{L_x^2\left(\mathbb{R}\right)}, \label{eqn: LtLx Strichartz}
 	\\ \left\| \int_{\mathbb{R}}  D^{\frac{\alpha}{2}\left(1-\frac{2}{r}\right)}e^{-it'\partial_x^4}F\left(t'\right)  \,dt' \right\|_{L_x^2\left(\mathbb{R}\right)}\lesssim_{\tilde{q},\tilde{r}} & \left\| F \right\|_{L_t^{\tilde{q}'}L_x^{\tilde{r}'}\left(\mathbb{R}\times \mathbb{R}\right)}, \label{eqn: LtLx homogeneou Strichartz}
 	\\
 	\left\| \int_{\mathbb{R}}  D^{\alpha\left(1-\frac{2}{r}\right)}e^{i(t-t')\partial_x^4}F\left(t'\right) \,dt'  \right\|_{L_t^qL_x^r\left(\mathbb{R}\times\mathbb{R}\right)}\lesssim_{q,r,\tilde{q},\tilde{r}} & \left\| F \right\|_{L_t^{\tilde{q}'}L_x^{\tilde{r}'}\left(\mathbb{R}\times\mathbb{R}\right)}. \label{eqn: LtLx inhomogeneous Strichartz}
 	\end{align}
 \end{lemma}
 
 \begin{proof}
 	From the standard $TT^*$ argument, it suffices to show the last estimate. Note that 
 	\begin{align*}
 	T:=&D^{\frac{\alpha}{2}\left(1-\frac{2}{r}\right)}e^{it\partial_x^4},\\
 	T^*:=& \int_{\mathbb{R}}D^{\frac{\alpha}{2}\left(1-\frac{2}{r}\right)}e^{it\partial_x^4} \,dt.
 	\end{align*}
 	Then, by the Minkowski and Hardy-Littlewood sobolev inequality, we have the following $TT^*$ estimate
 	\begin{align*}
 	\left\| \int_{\mathbb{R}}D^{\alpha \left(1-\frac{2}{r}\right)}e^{i(t-t')\partial_x^4}F\left(t'\right)\,dt'
 	\right\|_{L_t^qL_x^r}&\lesssim  \left\|  \int_{\mathbb{R}} \frac{1}{\left|t-t' \right|^{\frac{1+\alpha}{4}\left(1-\frac{2}{r}\right)}} \left\|F(t') \right\|_{L_x^{r'}}\,dt'    \right\|_{L_t^q}\\
 	&\lesssim  \left\|  F\right\|_{L_t^{q'}L_x^{r'}}.
 	\end{align*}
 	
 	\noindent
 	This completes the proof of Lemma \ref{St}.
 \end{proof}
 
 \begin{remark}
 	In particular, we mainly use the following estiamates:
 	\begin{align}
 	\left\| D^{\frac{1}{2}}e^{it\partial_x^4}u_0 \right\|_{L_t^4L_x^{\infty}\left( \mathbb{R}\times\mathbb{R}\right)}&\lesssim \left\| u_0 \right\|_{L^2\left(\mathbb{R}\right)},\\
 	\left\| e^{it\partial_x^4}u_0 \right\|_{L_t^8L_x^\infty\left(\mathbb{R}\times \mathbb{R}\right)}&\lesssim  \left\| u_0 \right\|_{L^2\left(\mathbb{R}\right)},\\
 	\Vert e^{it\partial_x^4}u_0      \Vert_{L_t^{\infty}L_x^2\left(\mathbb{R}\times\mathbb{R} \right)}&\lesssim  \left\| u_0 \right\|_{L^2\left(\mathbb{R}\right)}.
 	\end{align}
 \end{remark}

 
 
 \subsection{Bilinear Strichartz estimate}
 In a low-high interaction, we need the following bilinear Strichartz estimates.
 \begin{lemma}[Bilinear Strichartz estimate]\label{LEM:bilin}
 	Let $N_1,N_2$ be dyadic numbers with $N_1\leq \frac{N_2}{8}$. Then, we have
 	\begin{align}\label{eqn:bilinear estimate}
 	\left\|e^{it\partial_x^4}\phi_{N_1} e^{it\partial_x^4}\phi_{N_2} \right\|_{L_{t,x}^2\left(\mathbb{R}\times \mathbb{R}\right)} \lesssim N_2^{-\frac{3}{2}}\left\|\phi_{N_1} \right\|_{L_x^2\left(\mathbb{R}\right)} \left\|\phi_{N_2} \right\|_{L_x^2\left(\mathbb{R}\right)}.
 	\end{align}
 \end{lemma}
 \begin{proof} 
 	Note that by duality, we have 
 	\begin{align*}
 	&\left\|e^{it\partial_x^4}\phi_{N_1} e^{it\partial_x^4}\phi_{N_2} \right\|_{L_{t,x}^2\left(\mathbb{R}\times \mathbb{R}\right)}\\
 	=&\left\|\int_{\mathbb{R}}e^{it\xi_1^4}\widehat{\phi_{N_1}}\left(\xi_1\right) e^{it\left(\xi-\xi_1\right)^4}\widehat{\phi_{N_2}}\left(\xi-\xi_1\right) \,d\xi_1 \right\|_{L_{t,\xi}^2\left(\mathbb{R}\times \mathbb{R}\right)}\\
 	=&\sup\limits_{\left\|\psi \right\|_{L_{t,\xi}^2}=1} \left|\int_{\mathbb{R}\times\mathbb{R}} \int_{\mathbb{R}}e^{it\xi_1^4}\widehat{\phi_{N_1}}\left(\xi_1\right) e^{it\left(\xi-\xi_1\right)^4}\widehat{\phi_{N_2}}\left(\xi-\xi_1\right)\psi\left(t,\xi
 	\right) \,d\xi_1 d\xi dt \right|.
 	\end{align*}
 	Hence, it suffcies to consider the integral 
 	
 	\begin{align*}
 	&\left|\int_{\mathbb{R}\times\mathbb{R}} \int_{\mathbb{R}}e^{it\xi_1^4}\widehat{\phi_{N_1}}\left(\xi_1\right) e^{it\left(\xi-\xi_1\right)^4}\widehat{\phi_{N_2}}\left(\xi-\xi_1\right)\psi\left(t,\xi
 	\right) \,d\xi_1 d\xi dt \right|\\
 	=& \left|\int_{\mathbb{R}\times\mathbb{R}} \int_{\mathbb{R}}e^{it\xi_1^4}\widehat{\phi_{N_1}}\left(\xi_1\right) e^{it \xi_2^4}\widehat{\phi_{N_2}}\left(\xi_2\right)\psi\left(t,\xi_1+\xi_2
 	\right) \,d\xi_1 d\xi_2 dt \right|\\
 	=& \left|\int_{\mathbb{R}} \int_{\mathbb{R}}\widehat{\phi_{N_1}}\left(\xi_1\right)\widehat{\phi_{N_2}}\left(\xi_2\right)\mathcal{F}_t\psi\left(\xi_1^4+\xi_2^4,\xi_1+\xi_2 \right) \,d\xi_1 d\xi_2  \right|.
 	\end{align*}
 	We consider the change of variable $\left(\xi_1,\xi_2\right) \to \left(\eta_1(\xi_1,\xi_2),\eta_2(\xi_1,\xi_2)\right)=\left(\xi_1+\xi_2,\xi_1^4+\xi_2^4 \right)$ with the Jacobian $\left| J \right|=4\left| \xi_1^3-\xi_2^3\right|\sim N_2^3$. Hence, 
 	\begin{align*}
 	\left| J \right|d\xi_1 d\xi_2=d\eta_1 d\eta_2 \quad or \quad d\xi_1 d\xi_2    = \left| J \right|^{-1} d\eta_1 d\eta_2.
 	\end{align*}
 	So, from H\"older's inequality, the Plancherel theorem and the change of variable, we have
 	\begin{align*}
 	&\left|\int_{\mathbb{R}} \int_{\mathbb{R}}\widehat{\phi_{N_1}}\left(\xi_1\right)\widehat{\phi_{N_2}}\left(\xi_2\right)\mathcal{F}_t\psi\left(\xi_1^4+\xi_2^4,\xi_1+\xi_2 \right) \,d\xi_1 d\xi_2  \right|\\
 	&\lesssim \left\| \mathcal{F}_t \psi \right\|_{L_{\eta_1,\eta_2}^2} \left|\int_{ \eta_1\in \mathbb{R}} \int_{\eta_2\in\mathbb{R}}\left|\widehat{\phi_{N_1}}\left(\eta_1,\eta_2 \right)\widehat{\phi_{N_2}}\left(\eta_1,\eta_2\right)\right|^{2}  \left| J\right|^{-2} \,d\eta_1 d\eta_2 \right|^{\frac{1}{2}}\\
 	&\lesssim   N_{2}^{-\frac{3}{2}}\left|\int_{ \xi_1\in \mathbb{R}} \int_{\xi_2\in\mathbb{R}}\left|\widehat{\phi_{N_1}}\left(\xi_1 \right)\widehat{\phi_{N_2}}\left(\xi_2\right)\right|^{2}  \,d\xi_1 d\xi_2 \right|^{\frac{1}{2}}\\
 	&\lesssim  N_2^{-\frac{3}{2}}\left\|\phi_{N_1} \right\|_{L_x^2\left(\mathbb{R}\right)} \left\|\phi_{N_2} \right\|_{L_x^2\left(\mathbb{R}\right)}.
 	\end{align*}
 	
 	\noindent
 	This completes the proof of Lemma \ref{LEM:bilin}. 
 \end{proof}

 \section{Local well-posedness of the fourth order NLS}\label{sec: local well-posedness}
 In this section, we prove local well-posedness of \eqref{eqn:fourth order NLS}. We use the contraction principle on the $X^{s,b}$-space. 
 
 \subsection{Duhamel formulation and $X^{s,b}$-space}
 In this subsection, we introduce the following Duhamel formulation and $X^{s,b}$-space.
 
 By expressing \eqref{eqn:fourth order NLS} in the Duhamel formulation, we have 
 \begin{align}\label{eqn:Duhamel}
 u(t)=e^{it\partial_x^4}u_0\mp\int_0^te^{i(t-t')}F(u)(t')\,dt',
 \end{align}
 where $F(u)=\abs{u}^2u=u\overline{u}u$. Let $\eta$ be a smooth cutoff function supported on $[-2,2],\eta=1$ on $[-1,1]$ and $\eta_{\delta}(t)=\eta(\frac{t}{\delta})$. If $u$ satisfies 
 \begin{align}\label{eqn:modified Duhamel}
 u(t)=\eta(t)e^{it\partial_x^4}u_0\mp\eta(t)\int_0^te^{i(t-t')\partial_x^4}\eta_{\delta}(t')F(\eta u )(t')  \,dt'
 \end{align}
 for some $\delta \ll 1$, then it also satisfies $\eqref{eqn:Duhamel}$ on $[-\delta,\delta]$. Hence, we consider $\eqref{eqn:modified Duhamel}$ in the following.
 
 Next, let us recall some standard notations and facts.
 We denote the $X^{s,b}_{\tau=-\xi^4}\left(\mathbb{\mathbb{R}\times\mathbb{R}}\right)$  by  $X^{s,b}$. The $X^{s,b}$ space is defined to be the closure of the Schwartz functions $\mathcal{S}(\mathbb{R}\times \mathbb{R})$ under the norm
 \begin{align*}
 \Vert u \Vert_{X^{s,b}_{\tau=-\xi^4}\left(\mathbb{R}\times\mathbb{R}\right)}:=\Vert \langle \xi \rangle^s \langle \tau+\xi^4 \rangle^b \widetilde{u}(\tau,\xi)\Vert_{L^2_{\tau,\xi}(\mathbb{R}\times \mathbb{R})}.
 \end{align*}
 Next, we state the standard facts related to $X^{s,b}$ space.
 \begin{lemma}
 	Let $b>\frac{1}{2}$ and $s\in\mathbb{R}$. Then for any $u\in X^{s,b}_{\tau=-\xi^4}\left(\mathbb{R}\times \mathbb{R}\right)$, we have 
 	\begin{align}
 	\Vert u \Vert_{C_tH_x^s\left(\mathbb{R}\times \mathbb{R}\right)}\lesssim \Vert u \Vert_{X^{s,b}_{\tau=-\xi^4}(\mathbb{R}\times\mathbb{R})}.
 	\end{align}
 	Furthermore, $X^{s,b}$ space enjoy the same Strichartz estimate that free solutions do.
 \end{lemma}
 For the proof of this lemma, see \cite{Taobook}.

 \begin{lemma}[Strichartz estimates and bilinear Strichartz estimates]
 	Let $(q,r,\alpha)$ be an admissible exponent. Then, for $b>\frac{1}{2}$ and $N_1\ll N_2$, we have
 	\begin{align}
 	\Vert             D^{\frac{\alpha}{2}\left(1-\frac{2}{r}\right)} u \Vert_{L_t^qL_x^r\left(\mathbb{R}\times\mathbb{R}\right)}&\lesssim \Vert u \Vert_{X^{0,b}\left(\mathbb{R}\times\mathbb{R}\right)},\label{eqn: X^{s,b} Strichartz}\\
 	\Vert P_{N_1}u_1 P_{N_2}u_2 \Vert_{L^2_{t,x}\left(\mathbb{R}\times\mathbb{R}\right) }&\lesssim N_2^{-\frac{3}{2}}\Vert P_{N_1}u_1 \Vert_{X^{0,b}\left(\mathbb{R}\times\mathbb{R}\right)}\Vert P_{N_2}u_2 \Vert_{X^{0,b}\left(\mathbb{R}\times\mathbb{R}\right)}.\label{eqn:bilinear X^{s,b}} 
 	\end{align}
 	In particular, we have
 	\begin{align}
 	\left\| D^{\frac{1}{2}}u \right\|_{L_t^4L_x^{\infty}\left( \mathbb{R}\times\mathbb{R}\right)}&\lesssim \left\| u \right\|_{X^{0,\frac{1}{2}+}\left(\mathbb{R}\times \mathbb{R} \right)}\label{eqn:Strichartz X^{s,b} 1/2},\\
 	\left\| u \right\|_{L_t^8L_x^\infty\left(\mathbb{R}\times \mathbb{R}\right)}&\lesssim \left\|u \right\|_{X^{0,\frac{1}{2}+}\left(\mathbb{R}\times \mathbb{R}\right)}\label{eqn: L_t^8L_x^infty X^{s,b 
 		}Strichartz},\\
 	\Vert u       \Vert_{L_t^{\infty}L_x^2\left(\mathbb{R}\times\mathbb{R} \right)}&\lesssim \Vert u \Vert_{X^{0,\frac{1}{2}+}\left(\mathbb{R}\times\mathbb{R}\right)}.\label{eqn:energy strichartz}
 	\end{align}
 \end{lemma} 
 \begin{proof}
 	It follows directly from the transference principle. For the proof, see \cite{Taobook}.
 \end{proof}	
 The following is a $X^{s,b}$ energy estimate of time cut-off solutions.  
 \begin{lemma}\label{non}
 	Let $b>\frac{1}{2}$, $s\in \mathbb{R}$ and let $u$ solve the inhomogeneous fourth order NLS eqution $i\partial_tu-\partial_x^4u=F$. Then we have
 	\begin{align}\label{eqn:X^{s,b} energy estimate}
 	\Vert \eta(t)u \Vert_{X^{s,b}}\lesssim \Vert u_0\Vert_{H^s} +\Vert  F \Vert_{X^{s,b-1}}.  
 	\end{align}	
 \end{lemma}        
 For the proof of this lemma, see for instance \cite{Taobook}.
 To gain a small time factor in a nonlinear term, we need the following lemma. 
 \begin{lemma}
 	For any $-\frac{1}{2}<b'<b<\frac{1}{2}$ and $s\in \mathbb{R}$, we have
 	\begin{align}\label{eqn: X^{s,b} time parametor}
 	\Vert \eta(t/\delta) u \Vert_{X^{s,b'}}\lesssim \delta^{b-b'}\Vert u \Vert_{X^{s,b}}.
 	\end{align}
 \end{lemma}
 For the proof of this lemma, see \cite{Taobook}.
 
 \subsection{Trilinear estimate}
 To prove Theorem \ref{thm: local well-posedness}, the last thing we need is the following trilinear estimate.
 \begin{prop}\label{prop:t1}
 	Let $-\frac{1}{2}\leq s $, $\frac{1}{2}<b\leq\frac{7}{8}$.
 	Then, for each time localized function $u_j$, we have
 	\begin{align}\label{eqn:trilinear estimate}
 	\| u_1 \overline{u_2}u_3\|_{X^{s,b-1}_{\tau=-\xi^4}}\leq \|u_1 \|_{X^{s,\frac{1}{2}+}_{\tau=-\xi^4}}\|u_2 \|_{X^{s,\frac{1}{2}+}_{\tau=-\xi^4}}\|u_3 \|_{X^{s,\frac{1}{2}+}_{\tau=-\xi^4}}.
 	\end{align}	
 \end{prop}
 Combining \eqref{eqn:X^{s,b} energy estimate},\eqref{eqn: X^{s,b} time parametor} and \eqref{eqn:trilinear estimate}, one can prove that the operator
 \begin{align}\label{Duha}
 \Phi(u)(t):=\eta(t)e^{it\partial_x^4}u_0\mp\eta(t)\int_0^te^{i(t-t')\partial_x^4}\eta_{\delta}(t')F(\eta u )(t')  \,dt'
 \end{align}
 is a contraction mapping on a ball of $X^{s,b}$ space
 \begin{align*}
 \mathcal{B}=\left\{u\in X^{s,b}: \Vert u \Vert_{X^{s,b}}\leq 2R   \right\}
 \end{align*}
 for $R>0$ and $\Vert u_0 \Vert_{H^s_x}\leq R $. Since our proof is via the contraction principle, we also obtain that the solution map is locally Lipschitz continuous. Therefore, it suffices to prove the trilinear estimate $\eqref{eqn:trilinear estimate}$. Before we give the proof of the trilinear estimate $\eqref{eqn:trilinear estimate}$, let us remark a example that the trilinear estimate fails in the $X^{s,b}$ space for $s< -1/2$. 
 
 \begin{remark}\label{rem: trilienar estimate fail remark.}
 	We present the example that for $s<-1/2$, the trilinear estimate fails 
 	\begin{align*}
 	\| u_1 \overline{u_2}u_3\|_{X^{s,b-1}_{\tau=-\xi^4}}\leq \|u_1 \|_{X^{s,\frac{1}{2}+}_{\tau=-\xi^4}}\|u_2 \|_{X^{s,\frac{1}{2}+}_{\tau=-\xi^4}}\|u_3 \|_{X^{s,\frac{1}{2}+}_{\tau=-\xi^4}}.	
 	\end{align*}
 	In particualr, the nonlinear interaction high $\times$ high $\times$ high $\to$ high is a sources that makes a trouble. We follow the argument presented in Kenig-Ponce-Vega \cite{KPV1996}.
 	For a fixed large $N$, set $A=\left\{(\tau,\xi)\in \mathbb{R}^2: \vert \tau-\xi^4 \vert\leq 1  , N\leq \xi \leq N+\frac{1}{N}  \right\}$ and $B=-A=\left\{(\tau,\xi)\in \mathbb{R}^2 : (-\tau,-\xi)\in A   \right\}.$ We define functions $u,v$ and $w$ by Fourier transform
 	\begin{align*} 
 	\widehat{u}(\tau,\xi)=1_A(\tau,\xi),\\
 	\widehat{v}(\tau,\xi)=1_B(\tau,\xi),\\
 	\intertext{and}
 	\widehat{w}(\tau,\xi)=1_A(\tau,\xi).
 	\end{align*}
 	Note that $\Vert \widehat{u} \Vert_{L^2_{\tau,\xi}}\approx N^{-\frac{1}{2}},\Vert \widehat{v} \Vert_{L^2_{\tau,\xi}}\approx N^{-\frac{1}{2}} $ and $\Vert \widehat{w} \Vert_{L^2_{\tau,\xi}}\approx N^{-\frac{1}{2}}$. Moreover by the  definiton of the convolution we have
 	\begin{align*}
 	\widehat{u}*\widehat{v}(\tau,\xi)\approx \frac{1}{N}1_R,
 	\end{align*}
 	where $R=\left\{(\tau,\xi): \vert \tau-4N^3\xi \vert\lesssim 1, \vert \xi \vert \leq \frac{1}{N}  \right\}$ is a low frequency region. Hence the high $\times$ high interaction gives us a low frequency localized function $\widehat{u}*\widehat{v}$.
 	We also have
 	\begin{align*}
 	\left(\widehat{u}*\widehat{v}\right)*\widehat{w}\approx \frac{1}{N^2}1_A.
 	\end{align*}
 	Note that on the sets $A,B,R$, $\langle \tau-\xi^4 \rangle\approx 1$ and hence we have 
 	\begin{align*}
 	&\Vert \langle \xi \rangle^s \langle \tau-\xi^4 \rangle^{b-1} \left( \widehat{u}*\widehat{v}*\widehat{w} \right)(\tau,\xi)  \Vert_{L^2_{\tau,\xi}}\approx N^s N^{-2} N^{-\frac{1}{2}},\\
 	&\Vert \langle \xi \rangle^s \langle \tau -\xi^4 \rangle^{b} \widehat{u}(\tau,\xi) \Vert_{L^2_{\tau,\xi}}\approx N^s N^{-\frac{1}{2}},\\
 	&\Vert \langle \xi \rangle^s \langle \tau -\xi^4 \rangle^{b} \widehat{v}(\tau,\xi) \Vert_{L^2_{\tau,\xi}}\approx N^s N^{-\frac{1}{2}},\\
 	&\Vert \langle \xi \rangle^s \langle \tau -\xi^4 \rangle^{b} \widehat{w}(\tau,\xi) \Vert_{L^2_{\tau,\xi}}\approx N^s N^{-\frac{1}{2}}.
 	\end{align*}
 	Hence, in order for the trilinear estimate to hold, we need to have 
 	\begin{align}\label{eqn:counter example trilinear estimate}
 	N^sN^{-2}N^{-\frac{1}{2}}\lesssim \left(N^{-\frac{1}{2}} N^s \right)^3.
 	\end{align}
 	Since $N$ is very large, $\eqref{eqn:counter example trilinear estimate}$ is possible only if $s\geq -1/2.$ That is, there is a counter example of the trilinear estimate when the regularity $s$ is less than $-\frac{1}{2}$.
 \end{remark}

 Now, we give the proof of Proposition \ref{prop:t1} . Before proceeding further, we simplify some of the notations. Let us suppress the smooth time cut-off function $\eta $ from $\eta u_j$ and simply denote them by $u_j$.
 
 \begin{proof}[Proof of Proposition \ref{prop:t1} ]	
 	We may asuume $-1/2\leq s <0$. By duality and Plancherel theorem, 
 	\begin{align*}
 	\| u_1 \overline{u_2}u_3\|_{X^{s,b-1}_{\tau=-\xi^4}}=&\sup\limits_{\| v\|_{X^{-s,1-b}_{\tau=-\xi^4}}=1}\left| \int_{\mathbb{R}\times\mathbb{R}}u_1 \overline{u_2} u_3 \overline{v}\,dxdt \right|\\
 	=&\sup\limits_{\| v\|_{X^{-s,1-b}_{\tau=-\xi^4}}=1}\left| \int_{\substack {\xi_1+\dots \xi_4=0\\ \tau_1+\dots \tau_4=0}} \widehat{u_1}\left(\tau_1,\xi_1 \right) \widehat{\overline{u_2}}\left(\tau_2,\xi_2 \right)\widehat{u_3}\left(\tau_3,\xi_3\right)
 	\widehat{\overline{v}}\left(\tau_4,\xi_4 \right)\right|
 	\end{align*}
 	Since $u_1,u_2$ and $u_3$ are functions localized at time $\vert t \vert \lesssim 1$, we may assume that $v$ is also the function localized at time $ \vert t \vert \lesssim 1$.	 
 	Now set
 	\begin{align*}
 	&f_1\left(\tau_1,\xi_1\right)=\left| \widehat{u_1}\left(\tau_1,\xi_1\right)\right| \left\langle \xi_1\right\rangle^s \left \langle \tau_1+\xi_1^4 \right\rangle^{\frac{1}{2}+},\\
 	&f_2\left(\tau_2,\xi_2\right)=\left| \widehat{\overline{u_2}}\left(\tau_2,\xi_2\right)\right| \left\langle \xi_2\right\rangle^s \left \langle \tau_2-\xi_2^4 \right\rangle^{\frac{1}{2}+},\\ 
 	&f_3\left(\tau_3,\xi_3\right)=\left| \widehat{u_3}\left(\tau_3,\xi_3\right)\right| \left\langle \xi_3\right\rangle^s \left \langle \tau_3+\xi_3^4 \right\rangle^{\frac{1}{2}+},\\
 	\intertext{and}
 	&f_4\left(\tau_4,\xi_4\right)=\left| \widehat{\overline{v}}\left(\tau_4,\xi_4\right)\right| \left\langle \xi_4\right\rangle^{-s} \left \langle \tau_4-\xi_4^4 \right\rangle^{1-b}.   
 	\end{align*}	
 	Hence, we need to show that for nonnegative $L^2$ functions $f_j$, we have  
 	\begin{equation}\label{eqn:multilinear estimate in trilinear}
 	\begin{split}
 	\int\limits_{\Gamma_4}\frac{\left\langle \xi_4 \right\rangle^{s}\left\langle \tau_4-\xi_4^4 \right\rangle^{b-1}}{\left\langle \xi_1 \right\rangle^s \left\langle \xi_2\right\rangle^s\left\langle \xi_3\right\rangle^{s}\left\langle \tau_1+\xi_1^4 \right\rangle^{\frac{1}{2}+}\left\langle \tau_2-\xi_2^4 \right\rangle^{\frac{1}{2}+}\left\langle \tau_3+\xi_3^4 \right\rangle^{\frac{1}{2}+}}&\prod\limits_{j=1}^{4}f_j\left(\tau_j,\xi_j\right) \\
 	\leq &\prod\limits_{j=1}^{4}\|f_j \|_{L^2_{\tau,\xi}}, 
 	\end{split}
 	\end{equation}
 	where $\Gamma_4=\left\{ \xi_1+\dots+\xi_4=0,\tau_1+\dots+\tau_4=0   \right\}$ is the hyperplane.
 	Note that the variables $\xi_1$ and $\xi_3$ have symmetry. 
 	We define $|\xi_{\max}|,|\xi_{\text{sub}}|,|\xi_{\text{thd}}|,|\xi_{\min}|$ to be the maximum, second largest, third largest and minimum of $|\xi_1|,|\xi_2|,|\xi_3|,|\xi_4|$ and define the multiplier
 	\begin{align*}
 	m\left(\xi_1,\dots,\xi_4\right)=\frac{\left\langle \xi_4\right\rangle^s}{\left\langle\xi_1 \right\rangle^s \left\langle \xi_2\right\rangle^s \left\langle \xi_3 \right\rangle^s}.    
 	\end{align*} 
 	In order to obtain the trilineaer estimate $\eqref{eqn:trilinear estimate}$, we need to split the domain of integration $\eqref{eqn:multilinear estimate in trilinear}$ in several cases:
 	
 	\smallskip
 	
 	\noindent \textbf{Case 1.}  $ \Omega_1=\left\{  |\xi_1|,\dots,|\xi_4|\lesssim 1 \right\}$.\\    
 	\textbf{Case 2.} $ \Omega_2=\left\{ |\xi_{\max}| \gg 1 \quad \text{and} \quad  |\xi_{\max}|\sim |\xi_{\text{sub}}|\gg |\xi_{\text{thd}}|,|\xi_{\min}| \right\} $ .\\
 	\textbf{Case 3.} $ \Omega_3=\left\{  |\xi_{\max}| \gg 1 \quad \text{and}\quad |\xi_{\max}| \sim |\xi_{\text{sub}}| \sim |\xi_{\text{thd}}| \gg |\xi_{\min}|\right\}$.\\
 	\textbf{Case 4.} $  \Omega_4=\left\{ |\xi_{\max}| \gg 1 \quad \text{and}\quad|\xi_{\max}| \sim |\xi_{\text{sub}}| \sim |\xi_{\text{thd}}| \sim |\xi_{\min}| \right\} $.\\
 	
 	\textbf{Case 1.} $ \Omega_1=\left\{  |\xi_1|,\dots,|\xi_4|\lesssim 1 \right\} $.
 	Here, $P_{\lesssim 1}u_i$ is defined by $ \widehat{P_{\lesssim 1}u_i } \left(\xi\right)  =1_{\vert \xi \vert \lesssim 1 }\widehat{u_i}\left(\xi\right)$. As we mentioned above, we may assume $v$ is localized at time $\vert t \vert\lesssim 1$. Therefore, by using H\"older's inequality, we have
 	\begin{align*}
 	&\left\vert \int_{\mathbb{R}\times\mathbb{R}} P_{\lesssim 1 }u_1 P_{\lesssim 1 }\overline{u_2} P_{\lesssim 1}u_3 P_{\lesssim 1}\overline{v}\,dxdt \right\vert\\ &\lesssim \Vert v \Vert_{L_t^{\frac{4}{3}}L_x^2 } \Vert P_{\lesssim 1} u_1 P_{\lesssim 1 }\overline{u_2} P_{\lesssim 1}u_3 \Vert_{L_t^4L_x^2}\\
 	&\lesssim \Vert v \Vert_{L_t^2L_x^2} \Vert P_{\lesssim 1 } u_1 \Vert_{L_t^8L_x^\infty} \Vert P_{\lesssim 1 }\overline{u_2} \Vert_{L_t^8L_x^\infty } \Vert P_{\lesssim 1} u_3 \Vert_{L_t^\infty L_x^2}\\
 	&\lesssim\Vert v \Vert_{X_{\tau=-\xi^4}^{-s,1-b}} \Vert P_{\lesssim 1 } u_1 \Vert_{L_t^8L_x^\infty} \Vert P_{\lesssim 1 }\overline{u_2} \Vert_{L_t^8L_x^\infty } \Vert P_{\lesssim 1} u_3 \Vert_{L_t^\infty L_x^2}.  
 	\end{align*} 	
 	Note that from the Strichartz estimates \eqref{eqn: L_t^8L_x^infty X^{s,b 
 		}Strichartz}, \eqref{eqn:energy strichartz} and using the low frequency range $\vert \xi \vert \lesssim 1$, we have 
 	\begin{align*}
 	&\left\| P_{\lesssim 1}u_1 \right\|_{L_t^8L_x^{\infty}}\lesssim \left\| P_{\lesssim 1}u_1 \right\|_{X^{0,\frac{1}{2}+}_{\tau=-\xi^4}}\lesssim  \Vert u_1 \Vert_{X^{s,\frac{1}{2}+}_{\tau=-\xi^4}},\\ 
 	&\left\| P_{\lesssim 1}\overline{u_2} \right\|_{L_t^8L_x^{\infty}}\lesssim \left\| P_{\lesssim 1}\overline{u_2} \right\|_{X^{0,\frac{1}{2}+}_{\tau=\xi^4}}\lesssim  \Vert u_2 \Vert_{X^{s,\frac{1}{2}+}_{\tau=-\xi^4}},\\
 	& \Vert P_{\lesssim 1 } u_3 \Vert_{L_t^\infty L_x^2}\lesssim \Vert P_{\lesssim 1}u_3 \Vert_{X^{0,\frac{1}{2}+}_{\tau=-\xi^4}}\lesssim \Vert u_3 \Vert_{X^{s,\frac{1}{2}+}_{\tau=-\xi^4}}.
 	\end{align*}
 	In this case, the trilinear esimtae holds for all negative $s<0$.

 	\textbf{Case 2.}
 	$ \Omega_2=\left\{ |\xi_{\max}| \geq 1 \quad \text{and} \quad  |\xi_{\max}|\sim |\xi_{\text{sub}}|\gg |\xi_{\text{thd}}|,|\xi_{\min}| \right\} $.  We split the set $\Omega_2$ into the following subsets:\\
 	
 	\noindent \textbf{Subcase 2.a} $ \Omega_{2.a}= \left\{ \vert \xi_{\max} \vert \gg 1 \; \text{and}\; |\xi_2|\sim |\xi_4|\gg |\xi_1|,|\xi_3| \right\}$ . \\
 	\textbf{Subcase 2.b} $ \Omega_{2.b}=\left\{\vert \xi_{\max} \vert \gg 1 \; \text{and} \; |\xi_2|\sim|\xi_3| \gg |\xi_1|,|\xi_4| \right\}$
 	
 	\hphantom{XXXXXXXXX}or $\left\{ \vert \xi_{\max} \vert\gg 1 \; \text{and} \;  |\xi_2|\sim  |\xi_1| \gg |\xi_3|,|\xi_4| \right\}$.\\
 	\textbf{Subcase 2.c}  $ \Omega_{2.c}= \left\{\vert \xi_{\max} \vert \gg 1 \;\text{and}\; |\xi_3|\sim  |\xi_4| \gg |\xi_1|, |\xi_2| \right\}$
 	
 	\hphantom{XXXXXXXXX}or $ \left\{\vert \xi_{\max} \vert \gg1 \; \text{and}\; |\xi_1|\sim  |\xi_4| \gg |\xi_3|, |\xi_2| \right\}$.\\ 
 	\textbf{Subcase 2.d} $\Omega_{2.d}=\left\{\vert \xi_{\max} \vert \gg1 \;\text{and}\; |\xi_3|\sim  |\xi_1| \gg |\xi_2|, |\xi_4| \right\}$.\\
 	
 	\textbf{Subcase 2.a} $ \Omega_{2.a}= \left\{ \vert \xi_{\max} \vert \gg 1 \; \text{and}\; |\xi_2|\sim |\xi_4|\gg |\xi_1|,|\xi_3| \right\}$.
 	Note that on $\Omega_{2.a}$ we have
 	\begin{align*}
 	m\left(\xi_1,\xi_2,\xi_3,\xi_4\right)\lesssim \left \langle  \xi_1 \right\rangle^{\frac{1}{2}} \left\langle  \xi_3 \right\rangle^{\frac{1}{2}}.
 	\end{align*}
 	On $\left\{ \vert \xi_1 \vert, \vert \xi_3 \vert \gg 1   \right\}$ there is no difference between $ \vert \xi_i \vert^{\frac{1}{2}} \; \text{and} \; \left\langle \xi_i \right\rangle^{\frac{1}{2}},\; i=1,3 $. This means that on the high frequency region there is no difference between homogeneous derivative and inhomogeneous derivative. We consider the region 
 	\begin{align*}
 	&\Omega_{2.a,high,high}=\Omega_{2.a}\cap \left\{\vert \xi_1 \vert,\vert \xi_3 \vert \gg 1  \right\},\\
 	&\Omega_{2.a,high,low}=\Omega_{2.a}\cap \left\{\vert \xi_1\vert \gg 1, \vert \xi_3 \vert \lesssim 1  \right\} , \\
 	& \Omega_{2.a,low,high}=\Omega_{2.a}\cap \left\{\vert \xi_1\vert \lesssim 1, \vert \xi_3 \vert \gg 1  \right\},\\
 	&\Omega_{2.a,low,low}=\Omega_{2.a}\cap \left\{\vert \xi_1\vert , \vert \xi_3 \vert \lesssim 1  \right\}.
 	\end{align*}
 	Then, on $\Omega_{2.a,high,high}$, from the Strichartz estimates \eqref{eqn:Strichartz X^{s,b} 1/2} and \eqref{eqn:energy strichartz}, we have
 	\begin{align*}
 	&\int\limits_{\Gamma_4\cap \Omega_{2.a,high,high}}\frac{\left\langle \xi_4 \right\rangle^{s}\left\langle \tau_4-\xi_4^4 \right\rangle^{b-1}}{\left\langle \xi_1 \right\rangle^s \left\langle \xi_2\right\rangle^s\left\langle \xi_3\right\rangle^{s}\left\langle \tau_1+\xi_1^4 \right\rangle^{\frac{1}{2}+}\left\langle \tau_2-\xi_2^4 \right\rangle^{\frac{1}{2}+}\left\langle \tau_3+\xi_3^4 \right\rangle^{\frac{1}{2}+}}\prod\limits_{j=1}^{4}f_j\left(\tau_j,\xi_j\right)\\
 	&\lesssim 
 	\left\|D^{\frac{1}{2}}\mathcal{F}^{-1}\left(\frac{f_1}{\left\langle \tau_1+\xi_1^4 \right\rangle^{\frac{1}{2}+}} \right) \right\|_{L_t^4L_x^{\infty}}
 	\left \| \mathcal{F}^{-1}\left(\frac{f_2}{\left\langle \tau_2-\xi_2^4 \right\rangle^{\frac{1}{2}+}} \right) \right\|_{L_t^{\infty}L_x^2}\\
 	&\hphantom{XXXXXXXXX}\times
 	\left \| D^{\frac{1}{2}}\mathcal{F}^{-1}\left(\frac{f_3}{\left\langle \tau_3+\xi_3^4 \right\rangle^{\frac{1}{2}+}} \right) \right\|_{L_t^4L_x^{\infty}}
 	\left\| f_4\right\|_{L_{\tau,\xi}^2}\\
 	&\lesssim  \left\| u_1\right\|_{X^{s,\frac{1}{2}+}}\left\| u_2\right\|_{X^{s,\frac{1}{2}+}}\left\| u_3\right\|_{X^{s,\frac{1}{2}+}}.
 	\end{align*}
 	On $\Omega_{2.a,high,low}=\Omega_{2.a}\cap \left\{\vert \xi_1\vert \gg 1, \vert \xi_3 \vert \lesssim 1  \right\}$, we consider the integral
 	\begin{align}\label{eqn:omega high low}
 	\left\vert    \int_{\mathbb{R}\times \mathbb{R}} P_{\gg 1}u_1 P_{\gg 1 }\overline{u_2} P_{\lesssim 1}u_3 P_{\gg 1 }\overline{v} \,dxdt    \right\vert,
 	\end{align}
 	where $ \widehat {P_{\gg 1 }u_i}=1_{\vert \xi \vert \gg 1 }\widehat{u_i}$ and $ \widehat {P_{\lesssim 1 }u_i}=1_{\vert \xi \vert \lesssim 1 }\widehat{u_i}$. Thefore, by using H\"older's inequality, we have  
 	\begin{align*}
 	\eqref{eqn:omega high low}\lesssim& \Vert v \Vert_{L_t^2L_x^2} \Vert P_{\gg 1}u_1 P_{\gg 1}\overline{u_2} P_{\lesssim 1}u_3 \Vert_{L_t^2L_x^2}\\
 	\lesssim& \Vert v \Vert_{X^{-s,1-b}} \Vert P_{\gg 1}u_1 P_{\gg 1}\overline{u_2} P_{\lesssim 1}u_3 \Vert_{L_t^2L_x^2}.
 	\end{align*}  
 	Note that 
 	\begin{align*}
 	&\left\vert    \mathcal{F}\left(P_{\gg1 }u_1  P_{\gg 1}\overline{u_2}  P_{\lesssim 1}u_3 \right)(\xi)    \right\vert\\   
 	&\lesssim \left \vert \int_{\xi_1+\xi_2+\xi_3=\xi}\widehat{P_{\gg 1}u_1 }(\xi_1)  \widehat{P_{\gg 1}\overline{u_2} }(\xi_2)  \widehat{P_{\lesssim 1}u_3 }(\xi_3) \right\vert\\
 	&\lesssim \left( \vert \xi\vert^{\frac{1}{2}}1_{ \vert \xi \vert \gg1}\langle \xi \rangle^{s}\vert \widehat{u_1} \vert \right) * \left(  \vert \xi\vert^{\frac{1}{2}}1_{\vert \xi \vert  \gg 1}\langle \xi \rangle^{s}\vert \widehat{\overline{u_2}} \vert\right) * \left( 1_{\vert \xi \vert \lesssim 1} \vert \widehat{u_3}\vert \right)(\xi).
 	\end{align*}
 	Therefore, from Strichartz estimates $\eqref{eqn:Strichartz X^{s,b} 1/2}, \eqref{eqn:energy strichartz}$ and using the low frequency range $\vert \xi \vert \lesssim 1$, we have 
 	\begin{align*}
 	&\Vert P_{\gg 1}u_1 P_{\gg 1}\overline{u_2} P_{\lesssim 1}u_3 \Vert_{L_t^2L_x^2}
 	\\ \lesssim &  \Vert D^{\frac{1}{2}} P_{\gg 1} \langle D \rangle^s \mathcal{F}^{-1}\vert \widehat{u_1} \vert  \vert_{L_t^4L_x^\infty} \Vert D^{\frac{1}{2}} P_{\gg 1} \langle D \rangle^s \mathcal{F}^{-1}\vert \widehat{\overline{u_2}} \vert  \vert_{L_t^4L_x^\infty}  \Vert P_{\lesssim 1} \mathcal{F}^{-1} \vert \widehat{u_3} \vert \Vert_{L_t^\infty L_x^2}\\
 	\lesssim & \left\| u_1\right\|_{X^{s,\frac{1}{2}+}}\left\| u_2\right\|_{X^{s,\frac{1}{2}+}}\left\| u_3\right\|_{X^{s,\frac{1}{2}+}}.
 	\end{align*}
 	For the case $\Omega_{2.a,low,high}$, it is essentially the same as $\Omega_{2.a,high,low}$ by symmetry.
 	
 	On $\Omega_{2.a,low,low}=\Omega_{2.a}\cap \left\{\vert \xi_1\vert , \vert \xi_3 \vert \lesssim 1  \right\}$, we consider the integral
 	\begin{align}\label{eqn:omega low low}
 	\left\vert    \int_{\mathbb{R}\times \mathbb{R}} P_{\lesssim 1}u_1 P_{\gg 1 }\overline{u_2} P_{\lesssim 1}u_3 P_{\gg 1 }\overline{v} \,dxdt    \right\vert.
 	\end{align} 
 	Recall that $v$ is time localized function. Hence, by using H\"older's inequality, we have
 	\begin{align*}
 	\eqref{eqn:omega low low} \lesssim& \Vert v \Vert_{L_t^{\frac{8}{5}} L_x^2  } \Vert P_{\lesssim 1} u_1 P_{\gg 1 }\overline{u_2} P_{\lesssim 1} u_3 \Vert_{L_t^{\frac{8}{3}} L_x^2 }\\
 	\lesssim& \Vert v \Vert_{L_t^2L_x^2} \Vert P_{\lesssim 1} u_1 P_{\gg 1 }\overline{u_2} P_{\lesssim 1} u_3 \Vert_{L_t^{\frac{8}{3}} L_x^2 }\\
 	\lesssim&\Vert v \Vert_{X^{-s,1-b}} \Vert P_{\lesssim 1} u_1 P_{\gg 1 }\overline{u_2} P_{\lesssim 1} u_3 \Vert_{L_t^{\frac{8}{3}} L_x^2 }.
 	\end{align*}	  
 	Note that 
 	\begin{align*}
 	\left\vert    \mathcal{F}\left(P_{\lesssim 1 }u_1  P_{\gg 1}\overline{u_2}  P_{\lesssim 1}u_3 \right) (\xi)   \right\vert   &\lesssim \left \vert \int_{\xi_1+\xi_2+\xi_3=\xi}\widehat{P_{\lesssim 1}u_1 }(\xi_1)  \widehat{P_{\gg 1}\overline{u_2} }(\xi_2)  \widehat{P_{\lesssim 1}u_3 }(\xi_3) \right\vert\\
 	&\lesssim \left( 1_{ \vert \xi \vert \lesssim 1} \vert \widehat{u_1} \vert \right) * \left(  \vert \xi\vert^{\frac{1}{2}}1_{\vert \xi \vert  \gg 1}\langle \xi \rangle^{s}\vert \widehat{\overline{u_2}} \vert\right) * \left( 1_{\vert \xi \vert \lesssim 1} \vert \widehat{u_3}\vert \right)(\xi).
 	\end{align*}	
 	Therefore, from Strichartz estimates $\eqref{eqn:Strichartz X^{s,b} 1/2}, \eqref{eqn:energy strichartz}$ and using the low frequency range $\vert \xi \vert \lesssim 1$, we have
 	\begin{align*}
 	&\Vert P_{\lesssim 1} u_1 P_{\gg 1 }\overline{u_2} P_{\lesssim 1} u_3 \Vert_{L_t^{\frac{8}{3}} L_x^2 }\\
 	\lesssim& \Vert  P_{\lesssim 1}  \mathcal{F}^{-1}\vert \widehat{u_1} \vert  \vert_{L_t^8L_x^\infty} \Vert D^{\frac{1}{2}} P_{\gg 1} \langle D \rangle^s \mathcal{F}^{-1}\vert \widehat{\overline{u_2}} \vert  \vert_{L_t^4L_x^\infty}  \Vert P_{\lesssim 1} \mathcal{F}^{-1} \vert \widehat{u_3} \vert \Vert_{L_t^\infty L_x^2}\\
 	\lesssim& \left\| u_1\right\|_{X^{s,\frac{1}{2}+}}\left\| u_2\right\|_{X^{s,\frac{1}{2}+}}\left\| u_3\right\|_{X^{s,\frac{1}{2}+}}.\\
 	\end{align*}	
 	
 	\textbf{Subcase 2.b} $ \Omega_{2.b}=\left\{\vert \xi_{\max} \vert \gg 1 \; \text{and} \; |\xi_2|\sim|\xi_3| \gg |\xi_1|,|\xi_4| \right\}$ or $\{ \vert \xi_{\max} \vert\gg 1 \; \text{and} \;  |\xi_2|\sim  |\xi_1| \gg |\xi_3|,|\xi_4| \}$.
 	By symmetry of $\xi_1$ and $\xi_3$, we may assume 	$ \Omega_{2.b}=\left\{\vert \xi_{\max} \vert \gg 1 \; \text{and} \; |\xi_2|\sim|\xi_3| \gg |\xi_1|,|\xi_4| \right\}$.  
 	We split the set $\Omega_{2.b}$ into the following subsets:
 	\begin{align*}
 	&\Omega_{2.b.i}=\left\{\vert \xi_{\max}\vert \gg 1 \;\text{and}\;  |\xi_2|\sim|\xi_3|\gg |\xi_4| \gtrsim |\xi_1| \right\}\\ 
 	&\Omega_{2.b.ii}=\left\{\vert \xi_{\max}\vert \gg 1 \;\text{and}\;  |\xi_2|\sim|\xi_3|\gg |\xi_1| \gg |\xi_4| \right\}.
 	\end{align*}
 	On $\Omega_{2.b.i}$ the multiplier is estimated by 
 	\begin{align*}
 	m\left(\xi_1,\dots,\xi_4\right)=&\frac{ \left\langle \xi_4\right\rangle^{s}  }{\left\langle \xi_1\right\rangle^{s} \left\langle \xi_2\right\rangle^{s} \left\langle \xi_3\right\rangle^{s} }\\
 	=& \left( \frac{\left\langle \xi_1\right\rangle }{\left\langle \xi_4\right\rangle } \right)^{-s} \frac{1}{\left\langle \xi_2 \right\rangle^s \left\langle \xi_3 \right\rangle^s }\\
 	\lesssim & \left\langle \xi_2 \right\rangle^{\frac{1}{2}} \left\langle \xi_3 \right\rangle^{\frac{1}{2}} \sim \vert \xi_2 \vert^{\frac{1}{2}} \vert \xi_3 \vert^{\frac{1}{2}}.
 	\end{align*}
 	Then, on $\Omega_{2.b.i}$, from Strichartz estimates \eqref{eqn:Strichartz X^{s,b} 1/2} and \eqref{eqn:energy strichartz}, we have
 	\begin{align*}
 	&\int\limits_{\Gamma_4\cap \Omega_{2.b.i}}\frac{\left\langle \xi_4 \right\rangle^{s}\left\langle \tau_4-\xi_4^4 \right\rangle^{b-1}}{\left\langle \xi_1 \right\rangle^s \left\langle \xi_2\right\rangle^s\left\langle \xi_3\right\rangle^{s}\left\langle \tau_1+\xi_1^4 \right\rangle^{\frac{1}{2}+}\left\langle \tau_2-\xi_2^4 \right\rangle^{\frac{1}{2}+}\left\langle \tau_3+\xi_3^4 \right\rangle^{\frac{1}{2}+}}\prod\limits_{j=1}^{4}f_j\left(\tau_j,\xi_j\right)\\
 	\lesssim& 
 	\left\| \mathcal{F}^{-1}\left(\frac{f_1}{\left\langle \tau_1+\xi_1^4 \right\rangle^{\frac{1}{2}+}} \right) \right\|_{L_t^\infty L_x^{2}}
 	\left \| D^{\frac{1}{2}}\mathcal{F}^{-1}\left(\frac{f_2}{\left\langle \tau_2-\xi_2^4 \right\rangle^{\frac{1}{2}+}} \right) \right\|_{L_t^{4}L_x^\infty}\\
 	&\hphantom{XXX} \times \left \| D^{\frac{1}{2}}\mathcal{F}^{-1}\left(\frac{f_3}{\left\langle \tau_3+\xi_3^4 \right\rangle^{\frac{1}{2}+}} \right) \right\|_{L_t^4L_x^{\infty}}
 	\left\| f_4\right\|_{L_{\tau,\xi}^2}\\
 	\lesssim & \left\| u_1\right\|_{X^{s,\frac{1}{2}+}}\left\| u_2\right\|_{X^{s,\frac{1}{2}+}}\left\| u_3\right\|_{X^{s,\frac{1}{2}+}}.
 	\end{align*}
 	
 	On $\Omega_{2.b.ii}$, we need to observe that the interaction is nonresonant. More precisely, either the output or at least one of the inputs should have lare modulation. Note that on the hyperplane $\left\{\xi_1+\dots+\xi_4=0, \tau_1+\dots+\tau_4=0 \right\}$, we have
 	\begin{equation}
 	\begin{split}
 	\label{eqn: resonance function} 
 	&\vert \tau_1+\xi_1^4 \vert +\vert \tau_2-\xi_2^4\vert+\vert \tau_3+\xi_3^4\vert+\vert\tau_4-\xi_4^4\vert\\
 	&\gtrsim \left\vert \left(\xi_1+\xi_2\right)\left(\xi_2+\xi_3\right)\left(\xi_1^2+\xi_2^2+\xi_3^2+\left(\xi_1+\xi_2+\xi_3 \right)^2 +2\left(\xi_1+\xi_3\right)^2\right)\right\vert.
 	\end{split}
 	\end{equation}
 	For the proof of the factorization, see \cite{OT2016}. Therefore on $\Omega_{2.b.ii}$ we have
 	\begin{align*}
 	\max\left(\left|\tau_1+\xi_1^4\right|,\left|\tau_2-\xi_2^4\right|,\left|\tau_3+\xi_3^4\right|,\left|\tau_4-\xi_4^4\right| \right)&\gtrsim \left |\xi_1+\xi_2 \right| \left| \xi_2+\xi_3 \right|   \abs{\xi_{\max}}^2\\
 	&\gtrsim \vert \xi_1 \vert \vert \xi_{\max} \vert^3.    
 	\end{align*}
 	First, we consider the case $\left| \tau_4-\xi_4^4 \right|=\text{ max}\left(\left|\tau_1+\xi_1^4\right|,\left|\tau_2-\xi_2^4\right|,\left|\tau_3+\xi_3^4\right|,\left|\tau_4-\xi_4^4\right|\\ \right)$. For $\frac{1}{2}<b\leq \frac{7}{8}$, the multiplier is estimated by
 	\begin{align*}
 	\frac{m\left(\xi_1,\dots,\xi_4
 		\right)  }{\left\langle\tau_4-\xi_4^4\right\rangle^{1-b}}&\lesssim  \frac{\left\langle \xi_4\right\rangle^{s}  }{\left\langle \xi_1\right\rangle^{s} \left\langle \xi_2\right\rangle^s \left\langle \xi_3\right\rangle^s \left\langle \xi_1 \right\rangle^{4(1-b)}   }\\
 	&\lesssim \left\langle \xi_1 \right\rangle^{\frac{1}{2}} \left\langle \xi_2 \right\rangle^{\frac{1}{2}}\left\langle \xi_3 \right\rangle^{\frac{1}{2}}\frac{1}{\left\langle \xi_1 \right\rangle^{4(1-b)} }\\
 	&\lesssim  \left\langle \xi_2 \right\rangle^{\frac{1}{2}} \left\langle \xi_3 \right\rangle^{\frac{1}{2}} \frac{\left\langle \xi_1 \right\rangle^{\frac{1}{2}} }{\left\langle \xi_1 \right\rangle^{4(1-b)} }\lesssim  \left\vert \xi_2 \right\vert^{\frac{1}{2}} \left\vert \xi_3 \right\vert^{\frac{1}{2}}.
 	\end{align*}
 	Hence, by using Strichartz estimates \eqref{eqn:Strichartz X^{s,b} 1/2} and \eqref{eqn:energy strichartz}, one can proceed as in case $\Omega_{2.b.i}$. 
 	
 	If $\left| \tau_1+\xi_1^4 \right|=\text{max}\left(\left|\tau_1+\xi_1^4\right|,\left|\tau_2-\xi_2^4\right|,\left|\tau_1+\xi_1^4\right|,\left|\tau_4-\xi_4^4\right|\\ \right)$, then we have  
 	\begin{align}\label{J1}
 	\left\langle \tau_4-\xi_4^4 \right\rangle^{1-b} \left\langle \tau_1+\xi_1^4 \right\rangle^{\frac{1}{2}+}\gtrsim \left\langle \tau_4-\xi_4^4 \right\rangle^{\frac{1}{2}+} \left\langle \tau_1+\xi_1^4 \right\rangle^{1-b}.
 	\end{align} 
 	Therefore, by using $\eqref{J1}$, we consider the integral
 	\begin{align}\label{eqn:aaaa}
 	&\int\limits_{\Gamma_4\cap \Omega_{2.b.ii}}\frac{\left\langle \xi_4 \right\rangle^{s}\left\langle \tau_1+\xi_1^4 \right\rangle^{b-1}}{\left\langle \xi_1 \right\rangle^s \left\langle \xi_2\right\rangle^s\left\langle \xi_3\right\rangle^{s}\left\langle \tau_2-\xi_2^4 \right\rangle^{\frac{1}{2}+}\left\langle \tau_3+\xi_3^4 \right\rangle^{\frac{1}{2}+}\left\langle \tau_4-\xi_4^4 \right\rangle^{\frac{1}{2}+}}\prod\limits_{j=1}^{4}f_j\left(\tau_j,\xi_j\right).
 	\end{align}  
 	For $\frac{1}{2}<b\leq \frac{7}{8}$, the multiplier is estimated by
 	\begin{align*}
 	\frac{m\left(\xi_1,\dots,\xi_4
 		\right)  }{\left\langle\tau_1+\xi_1^4\right\rangle^{1-b}}\lesssim & \frac{\left\langle \xi_4\right\rangle^{s}  }{\left\langle \xi_1\right\rangle^{s} \left\langle \xi_2\right\rangle^s \left\langle \xi_3\right\rangle^s \left\langle \xi_1 \right\rangle^{4(1-b)}   }\\
 	\lesssim& \left\langle \xi_1 \right\rangle^{\frac{1}{2}} \left\langle \xi_2 \right\rangle^{\frac{1}{2}}\left\langle \xi_3 \right\rangle^{\frac{1}{2}}\frac{1}{\left\langle \xi_1 \right\rangle^{4(1-b)} }\\
 	\lesssim & \left\langle \xi_2 \right\rangle^{\frac{1}{2}} \left\langle \xi_3 \right\rangle^{\frac{1}{2}} \frac{\left\langle \xi_1 \right\rangle^{\frac{1}{2}} }{\left\langle \xi_1 \right\rangle^{4(1-b)} }\lesssim  \left\vert \xi_2 \right\vert^{\frac{1}{2}} \left\vert \xi_3 \right\vert^{\frac{1}{2}}.
 	\end{align*}
 	Therefore, from Strichartz estimates \eqref{eqn:Strichartz X^{s,b} 1/2} and \eqref{eqn:energy strichartz}, we have
 	\begin{align*}
 	\eqref{eqn:aaaa}
 	&\lesssim 
 	\left\| f_1\right\|_{L_{\tau,\xi}^2}
 	\left \| D^{\frac{1}{2}}\mathcal{F}^{-1}\left(\frac{f_2}{\left\langle \tau_2-\xi_2^4 \right\rangle^{\frac{1}{2}+}} \right) \right\|_{L_t^{4}L_x^\infty}\\
 	&\hphantom{XXXXX}\times \left \| D^{\frac{1}{2}}\mathcal{F}^{-1}\left(\frac{f_3}{\left\langle \tau_3+\xi_3^4 \right\rangle^{\frac{1}{2}+}} \right) \right\|_{L_t^4L_x^{\infty}}
 	\left\| \mathcal{F}^{-1}\left(\frac{f_4}{\left\langle \tau_4+\xi_4^4 \right\rangle^{\frac{1}{2}+}} \right) \right\|_{L_t^\infty L_x^{2}}\\
 	&\lesssim  \left\| u_1\right\|_{X^{s,\frac{1}{2}+}}\left\| u_2\right\|_{X^{s,\frac{1}{2}+}}\left\| u_3\right\|_{X^{s,\frac{1}{2}+}}.
 	\end{align*}
 	If $\left| \tau_3+\xi_3^4 \right|=\text{max}\left(\left|\tau_1+\xi_1^4\right|,\left|\tau_2-\xi_2^4\right|,\left|\tau_1+\xi_1^4\right|,\left|\tau_4-\xi_4^4\right|\\ \right)$, then we have  
 	\begin{align*}
 	\left\langle \tau_4-\xi_4^4 \right\rangle^{1-b} \left\langle \tau_3+\xi_3^4 \right\rangle^{\frac{1}{2}+}\gtrsim \left\langle \tau_4-\xi_4^4 \right\rangle^{\frac{1}{2}+} \left\langle \tau_3+\xi_3^4 \right\rangle^{1-b}.
 	\end{align*}
 	Therefore, we consider the integral
 	\begin{align}\label{bbbbbb} 
 	&\int\limits_{\Gamma_4\cap \Omega_{2.b.ii}}\frac{\left\langle \xi_4 \right\rangle^{s}\left\langle \tau_3+\xi_3^4 \right\rangle^{b-1}}{\left\langle \xi_1 \right\rangle^s \left\langle \xi_2\right\rangle^s\left\langle \xi_3\right\rangle^{s}\left\langle \tau_1+\xi_1^4 \right\rangle^{\frac{1}{2}+}\left\langle \tau_2-\xi_2^4 \right\rangle^{\frac{1}{2}+}\left\langle \tau_4-\xi_4^4 \right\rangle^{\frac{1}{2}+}}\prod\limits_{j=1}^{4}f_j\left(\tau_j,\xi_j\right).	
 	\end{align}	
 	On $\Omega_{2.b.ii}=\left\{\vert \xi_{max}\vert \gg 1 \;\text{and}\;  |\xi_2|\sim|\xi_3|\gg |\xi_1| \gg |\xi_4| \right\}$, we may assume $\vert \xi_1 \vert $ is in the high frequency region $\vert \xi_1 \vert \gg 1$. If $\vert \xi_1 \vert$ is in the low frequency region, we can proceed as in the case $\Omega_{2.a,high,low}$.  
 	
 	For $\frac{1}{2}<b\leq \frac{5}{6}$, the multiplier is estimated by
 	\begin{align*}
 	\frac{m\left(\xi_1,\dots,\xi_4
 		\right)  }{\left\langle\tau_3+\xi_3^4\right\rangle^{1-b}}\lesssim & \frac{\left\langle \xi_4\right\rangle^{s}  }{\left\langle \xi_1\right\rangle^{s} \left\langle \xi_2\right\rangle^s \left\langle \xi_3\right\rangle^s \left\langle \xi_1 \right\rangle^{1-b} \left\langle \xi_3 \right\rangle^{3(1-b)}} \\
 	\lesssim& \left\langle \xi_1 \right\rangle^{\frac{1}{2}} \left\langle \xi_2 \right\rangle^{\frac{1}{2}}\left\langle \xi_3 \right\rangle^{\frac{1}{2}}\frac{1}{\left\langle \xi_1 \right\rangle^{1-b}\left\langle \xi_3 \right\rangle^{3(1-b)} }\\
 	\lesssim& \langle \xi_1 \rangle^{\frac{1}{2}} \left\langle \xi_2 \right\rangle^{\frac{1}{2}}.
 	\end{align*}
 	Therefore, from the Strichartz estimates \eqref{eqn:Strichartz X^{s,b} 1/2} and \eqref{eqn:energy strichartz}, we have
 	\begin{align*}
 	&\eqref{bbbbbb}\\
 	\lesssim& 
 	\left \| D^{\frac{1}{2}}\mathcal{F}^{-1}\left(\frac{f_1}{\left\langle \tau_1+\xi_1^4 \right\rangle^{\frac{1}{2}+}} \right) \right\|_{L_t^{4}L_x^\infty}
 	\left \| D^{\frac{1}{2}}\mathcal{F}^{-1}\left(\frac{f_2}{\left\langle \tau_2-\xi_2^4 \right\rangle^{\frac{1}{2}+}} \right) \right\|_{L_t^{4}L_x^\infty}\\
 	&\hphantom{XXX} \times \left \| \mathcal{F}^{-1}f_3\right\|_{L_t^2L_x^2}
 	\left\| \mathcal{F}^{-1}\left(\frac{f_4}{\left\langle \tau_4+\xi_4^4  \right\rangle^{\frac{1}{2}+}} \right) \right\|_{L_t^\infty L_x^{2}}\\
 	&\lesssim  \left\| u_1\right\|_{X^{s,\frac{1}{2}+}}\left\| u_2\right\|_{X^{s,\frac{1}{2}+}}\left\| u_3\right\|_{X^{s,\frac{1}{2}+}}.
 	\end{align*}
 	For the case $\left| \tau_2-\xi_2^4 \right|=\text{max}\left(\left|\tau_1+\xi_1^4\right|,\left|\tau_2-\xi_2^4\right|,\left|\tau_1+\xi_1^4\right|,\left|\tau_4-\xi_4^4\right|\\ \right)$, we can proceed as in the case $\left| \tau_3+\xi_3^4 \right|=\text{max}\left(\left|\tau_1+\xi_1^4\right|,\left|\tau_2-\xi_2^4\right|,\left|\tau_1+\xi_1^4\right|,\left|\tau_4-\xi_4^4\right|\\ \right)$.
 	
 	\textbf{Subcase 2.c}  $ \Omega_{2.c}= \left\{\vert \xi_{max} \vert \gg 1 \;\text{and}\; |\xi_3|\sim  |\xi_4| \gg |\xi_1|, |\xi_2| \right\}$ or $ \{\vert \xi_{max} \vert \gg1 \; \text{and}\; |\xi_1|\sim  |\xi_4| \gg |\xi_3|, |\xi_2| \}$. By symmetry we may assume $ \Omega_{2.c}= \{\vert \xi_{max} \vert \gg 1 \;\text{and}\; |\xi_3|\sim  |\xi_4| \gg |\xi_1|, |\xi_2| \}$. Note that 
 	\begin{align*}
 	m\left(\xi_1,\xi_2,\xi_3,\xi_4\right)\lesssim \langle \xi_1 \rangle^{\frac{1}{2}} \langle \xi_2 \rangle^{\frac{1}{2}}.
 	\end{align*} 
 	This subcase can be handled as in Subcase 2.a.
 	
 	\textbf{Subcase 2.d} $\Omega_{2.d}=\left\{\vert \xi_{\max} \vert \gg1 \;\text{and}\; |\xi_3|\sim  |\xi_1| \gg |\xi_2|, |\xi_4| \right\}$. As in the case Subcase 2.b, we decompose the set $\Omega_{2.d}$ into the set $\left\{ \vert \xi_3 \vert \sim \vert \xi_1 \vert \gg \vert \xi_4 \vert \gtrsim \vert \xi_2 \vert \right\}$ and $\left\{\vert \xi_3 \vert \sim \vert \xi_1 \vert \gg \vert \xi_2 \vert \gg \vert \xi_4 \vert \right\}$. The remaining part can be proceed as in Subcase 2.b.
 	
 	\textbf{Case3.} $ \Omega_3=\left\{  |\xi_{\max}| \gg 1 \quad \text{and}\quad |\xi_{\max}| \sim |\xi_{\text{sub}}| \sim |\xi_{\text{thd}}| \gg |\xi_{\min}|\right\}$.
 	Note that $|\xi_{\max}| \sim |\xi_{\text{sub}}| \sim |\xi_{\text{thd}}| \gg |\xi_{\min}|$ and $\xi_1+\dots+\xi_4=0$ imply that
 	\begin{align*}	
 	&\vert \xi_1\vert=\vert\xi_{\min}\vert \Rightarrow \vert \xi_1+\xi_2 \vert \vert \xi_2+\xi_3 \vert \vert \xi_2 \vert^2\sim \vert \xi_2 \vert \vert\xi_4 \vert \vert \xi_2 \vert^2 \gtrsim \vert \xi_{\max}\vert^4, \\
 	&\vert \xi_2\vert=\vert\xi_{\min}\vert \Rightarrow \vert \xi_1+\xi_2 \vert \vert \xi_2+\xi_3 \vert \vert \xi_2 \vert^2\sim \vert \xi_1 \vert \vert\xi_3 \vert \vert \xi_2 \vert^2 \gtrsim \vert \xi_{\max}\vert^4, \\	
 	&\vert \xi_3\vert=\vert\xi_{\min}\vert \Rightarrow \vert \xi_1+\xi_2 \vert \vert \xi_2+\xi_3 \vert \vert \xi_2 \vert^2\sim \vert \xi_4 \vert \vert\xi_2 \vert \vert \xi_2 \vert^2 \gtrsim \vert \xi_{\max}\vert^4,\\ 	
 	&\vert \xi_4\vert=\vert\xi_{\min}\vert \Rightarrow \vert \xi_1+\xi_2 \vert \vert \xi_2+\xi_3 \vert \vert \xi_2 \vert^2\sim \vert \xi_3 \vert \vert\xi_1 \vert \vert \xi_2 \vert^2 \gtrsim \vert \xi_{\max}\vert^4. 
 	\end{align*}
 	Therefore, we have
 	\begin{align*}
 	\max\left( |\tau_1+\xi_1^4|,\dots,|\tau_4-\xi_4^4|    \right)\gtrsim& \abs{\xi_{\max}}^4.    
 	\end{align*}
 	We may assume $\vert \tau_4-\xi_4^4 \vert =\max\left( |\tau_1+\xi_1^4|,\dots,|\tau_4-\xi_4^4|    \right)$
 	and hence $\langle \tau_4 -\xi_4^4 \rangle \gtrsim \langle \xi_{\max} \rangle^4 $. The remaining case can be handled as in Subcase 2.b. Hence, we have 
 	\begin{align*}
 	& \int\limits_{\Gamma_4 \cap \Omega_3}\frac{\left\langle \xi_4 \right\rangle^{s}\left\langle \tau_4-\xi_4^4 \right\rangle^{b-1}}{\left\langle \xi_1 \right\rangle^s \left\langle \xi_2\right\rangle^s\left\langle \xi_3\right\rangle^{s}\left\langle \tau_1+\xi_1^4 \right\rangle^{\frac{1}{2}+}\left\langle \tau_2-\xi_2^4 \right\rangle^{\frac{1}{2}+}\left\langle \tau_3+\xi_3^4 \right\rangle^{\frac{1}{2}+}}\prod\limits_{j=1}^{4}f_j\left(\tau_j,\xi_j\right)\\
 	&\lesssim  \int\limits_{\Gamma_4 \cap \Omega_3} \frac{\left\langle \xi_{\max} \right\rangle^{\frac{1}{2}} \left\langle \xi_{\text{sub}} \right\rangle^{\frac{1}{2}}  \left\langle \xi_{\text{thd}} \right\rangle^{\frac{1}{2}}  }{ \left\langle \tau_4-\xi_4^4 \right \rangle^{1-b} \left\langle \tau_1+\xi_1^4 \right\rangle^{\frac{1}{2}+}\left\langle \tau_2-\xi_2^4 \right\rangle^{\frac{1}{2}+}\left\langle \tau_3+\xi_3^4 \right\rangle^{\frac{1}{2}+}}\prod\limits_{j=1}^4 f_j\left(\tau_j,\xi_j\right)\\
 	&\lesssim  \int\limits_{\Gamma_4 \cap \Omega_3} \frac{  \left\langle \xi_{\max} \right\rangle^{\frac{1}{2}}  }{ \left\langle \xi_{\max} \right \rangle^{4(1-b)}} \frac{\vert \xi_{\text{sub}} \vert^{\frac{1}{2}} \vert \xi_{\text{thd}} \vert^{\frac{1}{2}}}{\left\langle \tau_1+\xi_1^4 \right\rangle^{\frac{1}{2}+}\left\langle \tau_2-\xi_2^4 \right\rangle^{\frac{1}{2}+}\left\langle \tau_3+\xi_3^4 \right\rangle^{\frac{1}{2}+}}   \prod\limits_{j=1}^4 f_j\left(\tau_j,\xi_j\right).
 	\end{align*}
 	Therefore, from the Strichartz estimates \eqref{eqn:Strichartz X^{s,b} 1/2} and \eqref{eqn:energy strichartz}, we obtain the desired result.
 	
 	\textbf{Case 4.} $  \Omega_4=\left\{ |\xi_{\max}| \gg 1 \quad \text{and}\quad|\xi_{\max}| \sim |\xi_{\text{sub}}| \sim |\xi_{\text{thd}}| \sim |\xi_{\min}| \right\} $. Note that the multiplier is estimated by
 	\begin{align*}
 	m\left(\xi_1,\xi_2,\xi_3,\xi_4\right)\lesssim \langle \xi_1 \rangle^{\frac{1}{2}} \langle \xi_2 \rangle^{\frac{1}{2}}\lesssim \vert \xi_1 \vert^{\frac{1}{2}} \vert \xi_2 \vert^{\frac{1}{2}}.
 	\end{align*}
 	Therefore, by using Strichartz estimates \eqref{eqn:Strichartz X^{s,b} 1/2} and \eqref{eqn:energy strichartz}, we obtain the desired result.
 \end{proof}

 
 \section{Global well-posedness on $H^{-\frac{1}{2}}(\mathbb{R})$ and correction term strategy }\label{sec: Global well-posedness and correction term strategy}
 
 We introduce an even, smooth, and monotone multiplier $m$ taking values in $[0,1]$, and
 \begin{align}
 \label{Im}
 m(\xi):=
 \begin{cases}
 1,   & |\xi|<N \\
 \left(\frac{|\xi|}{N}\right)^s,  & |\xi|>2N.
 \end{cases}
 \end{align}
 Here, $N$ is a large parameter to be determined later. We define an operator $I$ by 
 \begin{align*}
 \widehat{Iu}\left(\xi\right):=m\left(\xi\right) \widehat{u}\left( \xi \right).
 \end{align*}
 Note that we have the estimate
 \begin{align*}
 \| u\|_{H^s}\lesssim \| Iu \|_{L^2}\lesssim N^{-s} \| u\|_{H^s}.
 \end{align*}
 
 \noindent
 Thus, the operator $I$ acts as the identity for low frequencies while it maps $H^s$ solutions to $L^2$. The main part is to prove that the increment of the modified energy $\|Iu(t) \|_{L_x^2}$ decays with respect to a large parameter $N$.

 Now we introduce useful multilinear notations. If $n\geq2$ is an even integer, we define a multiplier of order $n$ to be any function $M_n\left(\xi_1\,\dots,\xi_n \right)$ on the hyperplane 
 \begin{align*}
 \Gamma_n:=\left\{\left(\xi_1,\dots,\xi_n\right)\in \mathbb{R}^n: \xi_1+\dots+\xi_n=0 \right\},     
 \end{align*}
 which we endow with the standard measure
 \begin{align}
 \int_{\Gamma_n}h\left(\xi_1,\dots,\xi_n\right)=\int_{\mathbb{R}^{n-1}}h\left(\xi_1,\dots,\xi_{n-1},-\xi_1-\dots -\xi_{n-1}\right)\,d\xi_1\dots d\xi_{n-1}.    
 \end{align}
 We define a multilinear form
 \begin{align}
 \Lambda_n\left(M_n; u_1,\dots,u_n\right):=\int\limits_{\xi_1+\dots+\xi_n=0}M_n\left(\xi_1,\dots,\xi_n \right)\prod\limits_{j=1}^n\widehat{u}_j\left(\xi_j\right).  
 \end{align}
 Usually we apply $\Lambda_n$ on $n$ same functions which are all $u$. We will use the following notation:
 \begin{align*}
 \Lambda \left(M_n\right):&=\Lambda\left(M_n;u,\bar{u},u,\bar{u},\dots,u,\bar{u}\right)\\
 =&\int\limits_{\xi_1+\dots+\xi_n=0}M_n\left(\xi_1,\dots,\xi_n \right)\widehat{u}\left(\xi_1\right)\widehat{\bar{u}}\left(\xi_2\right)\widehat{u}\left(\xi_3\right)\widehat{\bar{u}}\left(\xi_4\right)\dots \widehat{u}\left(\xi_{n-1}\right)\widehat{\bar{u}}\left(\xi_n\right).    
 \end{align*}
 Note that $\Lambda_n\left(M_n\right)$ is invariant under permutations of the even $\xi_j$ indices, or of the odd $\xi_j$ indices.
 We shall often write $\xi_{ij}$ for $\xi_i+\xi_j$,$\xi_{ijk}$ for $\xi_i+\xi_j+\xi_k$. We also write $\xi_{i-j}$ for $\xi_i-\xi_j,\xi_{ij-klm}$ for $\xi_{ij}-\xi_{klm}$.
 
 Define the modified energy $E_I^2\left(t\right)$
 \begin{align}\label{eqn: modified energy E_I^2}
 E_I^2\left(t\right)=\left\| Iu(t)\right\|_{L_x^2}^2.    
 \end{align}
 Note that from the Plancherel's theorem, we obtain
 \begin{align*}
 E_I^2\left(t\right)=&\left\|Iu(t) \right\|_{L_x^2}^2\\
 =&\int_{\xi_1+\xi_2=0}m\left(\xi_1\right)m\left(\xi_2\right)\widehat{u}\left(\xi_1\right)\widehat{\bar{u}}\left(\xi_2\right)\\
 =&\Lambda_2\left( m\left(\xi_1\right) m\left(\xi_2\right);u \right).
 \end{align*}
 By differentiating this, we have
 \begin{align*}
 \frac{d}{dt}E_I^2\left(t\right)=&\int_{\xi_1+\xi_2=0}i\left(\xi_1^4-\xi_2^4\right)m\left(\xi_1\right)m\left(\xi_2\right)\widehat{u}\left(\xi_1\right)\widehat{\bar{u}}\left(\xi_2\right)\\
 +&\int_{\xi_1+\xi_2+\xi_3+\xi_4=0}i\left(m^2\left(\xi_1\right)-m^2\left(\xi_4\right) \right)\widehat{u}\left(\xi_1\right)\widehat{\bar{u}}\left(\xi_2\right)\widehat{u}\left(\xi_3\right)\widehat{\bar{u}}\left(\xi_4\right).
 \end{align*}
 Since $\xi_1+\xi_2=0$, the first term vanishes. By symmetrizing the second term, we have
 \begin{align*}
 \frac{d}{dt}E_I^2\left(t\right)=&\frac{1}{2}\int_{\xi_1+\xi_2+\xi_3+\xi_4=0}i\left(m^2\left(\xi_1\right)-m^2\left(\xi_2\right)+m^2\left(\xi_3\right)-m^2\left(\xi_4\right) \right)\widehat{u}\left(\xi_1\right)\widehat{\bar{u}}\left(\xi_2\right)\widehat{u}\left(\xi_3\right)\widehat{\bar{u}}\left(\xi_4\right)\\
 =&\frac{i}{2}\Lambda_4 \left(m^2\left(\xi_1\right)-m^2\left(\xi_2\right)+m^2\left(\xi_3\right)-m^2\left(\xi_4\right));u \right)
 .
 \end{align*}
 We set
 \begin{align*}
 M_4\left(\xi_1,\dots,\xi_4\right)=m^2\left(\xi_1\right)-m^2\left(\xi_2\right)+m^2\left(\xi_3\right)-m^2\left(\xi_4\right).    
 \end{align*}
 Hence, we obtain 
 \begin{align*}
 \frac{d}{dt}E_I^2\left(t\right)=\frac{i}{2}\Lambda_4\left(M_4\right).   
 \end{align*}
 Define a new modified energy
 \begin{align}\label{eqn : modified energy E_I^4}
 E_I^4\left(t\right)=E_I^2\left(t\right)+\Lambda_4\left(\sigma_4\right),    
 \end{align}
 where the function $\sigma_4$ is symmetric under the even $\xi_j$ indices, or of the odd $\xi_j$ indices. The $\sigma_4$ will be determined later.  
 Observe that
 \begin{align*}
 \frac{d}{dt}\Lambda_4\left(\sigma_4\right)=&\int\limits_{\xi_1+\xi_2+\xi_3+\xi_4=0}\sigma_4\left(\xi_1,\dots,\xi_4\right)i\left(\xi_1^4-\xi_2^4+\xi_3^4-\xi_4^4 \right)\widehat{u}\left(\xi_1\right)\widehat{\bar{u}}\left(\xi_2\right)\widehat{u}\left(\xi_3\right)\widehat{\bar{u}}\left(\xi_4\right) \\
 +&\int\limits_{\xi_1+\xi_2+\xi_3+\xi_4+\xi_5+\xi_6=0}2i\left(\sigma_4\left(\xi_1,\xi_2,\xi_3,\xi_4+\xi_5+\xi_6 \right)- \sigma_4\left(\xi_1,\xi_2,\xi_3+\xi_4+\xi_5,\xi_6 \right) \right)\\
 &\hspace{60mm}\times\widehat{u}\left(\xi_1\right)\widehat{\bar{u}}\left(\xi_2\right)\widehat{u}\left(\xi_3\right)\widehat{\bar{u}}\left(\xi_4\right)\widehat{u}\left(\xi_5\right)\widehat{\bar{u}}\left(\xi_6\right).
 \end{align*}
 By differentiating the new modified energy $E_{I}^4\left(t\right)$, we obtain
 \begin{equation}
 \begin{split}
 \label{cancel}
 &\frac{d}{dt}E_I^4\left(t\right)\\
 & =\Lambda_4\left( M_4 \right)+\Lambda_4 \left( \sigma_4 i\left(\xi_1^4-\xi_2^4+\xi_3^4-\xi_4^4\right)\right)+4\Re\Lambda_6\left( \sigma_4\left(\xi_1,\xi_2,\xi_3,\xi_4+\xi_5+\xi_6 \right) \right).    
 \end{split}
 \end{equation}
 Define $\alpha_n,\sigma_4$ by
 \begin{align}
 \alpha_n:=&i\left(\xi_1^4-\xi_2^4+\dots+(-1)^{n+1}\xi_n^4\right)\label{eqn: n-resonance function}\\
 \intertext{and}
 \sigma_4:=&-\frac{M_4}{i\alpha_4}\label{eqn: cancel correction term}
 \end{align}
 such that the two first terms in (\ref{cancel}) are canceled.
 We define
 \begin{align*}
 M_6\left(\xi_1,\dots,\xi_6\right)=i\sigma_4\left(\xi_1,\xi_2,\xi_3,\xi_4+\xi_5+\xi_6 \right).
 \end{align*}
 Then, we have
 \begin{align*}
 \frac{d}{dt}E_{I}^4\left(t\right)=4\Re\Lambda_6\left(M_6\right).     
 \end{align*}
 We will prove two main properties. One is to show that modified energy $E_I^4(t)$ is almost conserved. The other is to prove that it is close to $E_I^2(t)$. To control the increment  of $E_I^4(t)$, it suffices to estimate its derivative
 \begin{align*}
 \frac{d}{dt}E_{I}^4\left(t\right)=4\Re\Lambda_6\left(M_6\right).     
 \end{align*}
 
 \noindent
 Hence, in the following subsections, we consider the multiplier $M_6$ and multilinear form $\Lambda_6(M_6)$.

 \subsection{Multiplier estimates} 
 In this subsection, we present multiplier estimates.
 If $m$ is of the form (\ref{Im}), then $m^2$ satisfies
 
 \begin{equation}
 \begin{split}
 \begin{cases}
 m^2\left(\xi\right)\sim m^2\left(\xi' \right), \quad |\xi|\sim |\xi'|, \\
 \left(m^2\right)'\left(\xi\right)=O\left(\frac{m^2\left(\xi\right)}{|\xi|} \right),\\
 \left(m^2\right)^{''}\left(\xi\right)=O\left(\frac{m^2\left(\xi\right)}{|\xi|^2} \right).
 \end{cases}
 \end{split}    
 \end{equation} 
 We will use two mean value formulas: if $|\eta|,|\lambda|\ll |\xi|$, then
 \begin{align}\label{eqn: mean value theorem}
 \left| a(\xi+\eta)-a(\xi) \right|\lesssim |\eta|\sup\limits_{|
 	\xi'|\sim |\xi|}\left|a'\left(\xi'\right) \right|,    
 \end{align}
 and
 \begin{align}\label{eqn: double mean value theorem}
 \left| a\left(\xi+\eta+\lambda \right)-a\left(\xi+\eta\right)-a\left(\xi+\lambda\right)+a\left(\xi\right)  \right|\lesssim |\eta||\lambda|\sup\limits_{|\xi'|\sim |\xi|} \left| a^{''}\left(\xi'\right) \right|.    
 \end{align}
 \begin{lemma}\label{lem: multiplier estimate sigma_4}
 	Let $m$ be the multiplier of (\ref{Im}). On the area $|\xi_i|\sim N_i$, we have
 	\begin{align*}
 	\left|\sigma_4\left(\xi_1,\dots,\xi_4 \right) \right|=\left|\frac{M_4}{\alpha_4} \right|\lesssim \frac{m^2\left(min\left(N_i\right)\right)}{\left(N+N_1\right)\left(N+N_2\right)\left(N+N_3\right)\left(N+N_4\right)}.    
 	\end{align*}
 \end{lemma} 
 \begin{proof}
 	In the following, we use the notation 	
 	$|\xi_j+\xi_k|\sim N_{jk}$, where $N_{j,k}$ are dyadic numbers.
 	If $|\xi_{\max}|\ll N$, then $m^2\left(\xi_1 \right)-m^2\left(\xi_2 \right)+m^2\left(\xi_3 \right)-m^2\left(\xi_4\right)=0. $
 	Hence, we may assume that $|\xi_{\max}|\sim |\xi_{sub}|\gtrsim N$.
 	Recall that
 	\begin{align*}
 	\frac{M_4}{\alpha_4}=\frac{m^2\left(\xi_1\right)-m^2\left(\xi_2\right)+m^2\left(\xi_3\right)-m^2\left(\xi_4\right)}{i\left(\xi_1^4-\xi_2^4+\xi_3^4-\xi_4^4 \right)}.     
 	\end{align*}
 	Note that the resonance function on the hyperplane
 	$\left\{\xi_1+\dots+\xi_4=0\right\}$ is given by
 	\begin{align*}
 	\xi_1^4-\xi_2^4+\xi_3^4-\xi_4^4=\left(\xi_1+\xi_2\right)\left(\xi_2+\xi_3\right)\left(\xi_1^2+\xi_2^2+\xi_3^2+\left(\xi_1+\xi_2+\xi_3 \right)^2 +2\left(\xi_1+\xi_3\right)^2\right).
 	\end{align*}	
 	Therefore, on the hyperplane $\left\{\xi_1+\dots+\xi_4=0\right\}$ we consider
 	\begin{align*}
 	\frac{m^2\left(\xi_1\right)-m^2\left(\xi_2\right)+m^2\left(\xi_3\right)-m^2\left(\xi_4\right)}{\left(\xi_1+\xi_2\right)\left(\xi_1+\xi_4\right)\left(\xi_1^2+\xi_2^2+\xi_3^2+\left(\xi_1+\xi_2+\xi_3 \right)^2 +2\left(\xi_1+\xi_3\right)^2\right)}.
 	\end{align*} 	
 	We may assume that $N_{max}=N_1$. By symmetry, we can also assume $\vert \xi_{12} \vert \leq \vert \xi_{14} \vert$. We consider the two cases.
 	
 	\smallskip
 	
 	\noindent \textbf{Case 1.} $\vert \xi_{14} \vert \ll N_1.$\\
 	\noindent \textbf{Case 2.} $N_1 \lesssim \vert \xi_{14} \vert $.\\	
 	
 	\textbf{Case 1.} $\vert \xi_{14} \vert \ll N_1.$
 	In this case it suffices to show that
 	\begin{align*}
 	\vert m^2(\xi_1)-m^2(\xi_2)+m^2(\xi_3)-m^2(\xi_4) \vert\lesssim \vert \xi_{12} \vert \vert \xi_{14} \vert \frac{m^2(N_1)}{N_1^2}.
 	\end{align*}	 
 	By using the double mean value theorem $\eqref{eqn: double mean value theorem}$, we have
 	\begin{align*}
 	&\vert m^2(\xi_1)-m^2(\xi_2)+m^2(\xi_3)-m^2(\xi_4) \vert\\
 	&=\vert m^2(\xi_1-\xi_{12}-\xi_{14})-m^2(\xi_1-\xi_{12})-m^2(\xi_1-\xi_{14})+m^2(\xi_1))\\
 	&\lesssim\vert \xi_{12} \vert \vert \xi_{14} \vert \frac{m^2(N_1)}{N_1^2}.
 	\end{align*}	 
 	Since  $N_1 \gtrsim N $, we have $N_1+N\sim N_1.$ So we have the desired result.\\
 	
 	\textbf{Case 2.} $N_1 \lesssim \vert \xi_{14} \vert $.
 	Since $N_1 \lesssim \vert \xi_{14} \vert$, it suffices to consider
 	\begin{align*}
 	\frac{\vert m^2(\xi_1)-m^2(\xi_2) \vert }{\vert \xi_1 +\xi_2 \vert}\quad \text{and} \quad \frac{\vert m^2(\xi_3)-m^2(\xi_4) \vert }{\vert \xi_3 +\xi_4 \vert}.
 	\end{align*}	
 	If $N_{1}\lesssim \vert \xi_{12} \vert$, then we obtain the result directly. Hence, we may assume $\vert \xi_{12} \vert \ll N_1$. Then, by the mean value theorem $\eqref{eqn: mean value theorem}$, we have
 	\begin{align*}
 	\vert m^2(\xi_1)-m^2(\xi_2) \vert=&\vert m^2(\xi_1)-m^2(\xi_1-\xi_{12}) \vert\\
 	\lesssim & \vert \xi_{12} \vert\frac{m^2(N_1)}{N_1}. 
 	\end{align*}	
 	By symmetry, we now assume $N_3\geq N_4$. As we did in the previous step, if  $N_3 \lesssim \vert \xi_{12} \vert=\vert \xi_{34} \vert $, then we obtain the desired result. If $\vert \xi_{34} \vert \ll N_3 $, then by the mean value theorem $\eqref{eqn: mean value theorem}$ we have
 	\begin{align*}
 	\vert  m^2(\xi_3) -m^2(\xi_4) \vert=&\vert m^2(\xi_3)-m^2(\xi_3-\xi_{34})\vert\\
 	\lesssim& \vert \xi_{34} \vert \frac{m^2(N_3)}{N_3}. 
 	\end{align*}
 	Note that if $N_3 \ll N$, then we have $m^2(\xi_3)-m^2(\xi_4)=1-1=0 $. Therefore, it is enough to consider the case $N_3\gtrsim N$ and hence we have $N_3+N\sim N_3$.
 \end{proof}

 By applying the above lemma, the esimate of $M_6$ follows directly from the estimate $M_4$.
 \begin{prop}
 	If m is of the form (\ref{Im}),then
 	
 	\begin{align*}
 	\left|M_6\left(\xi_1,\dots,\xi_6\right) \right|\lesssim& 
 	\frac{m^2\left(min\left(N_i,N_{jkl}\right)\right)}{\left(N+N_1\right)\left(N+N_2\right)\left(N+N_3\right)\left(N+N_{456}\right)},
 	\end{align*}
 	where $N_i,N_{jkl}$ are dyadic numbers and $N_i\sim |\xi|, N_{jkl}\sim |\xi_k+\xi_k+\xi_l|.$
 \end{prop}
 
 \subsection{Almost conservation law}
 In this subsection, we present the following almost conservation law.
 \begin{lemma}[Almost conservation law.]\label{lem:almost conservation law}
 	Let $0<\delta \leq 1$. Then, for $m$ (given by (\ref{Im})) with $s=-\frac{1}{2}$, we have
 	\begin{align*}
 	\left|  E_{I}^4(\delta)-E_{I}^4(0) \right|&=\left|\int_0^\delta 4\Re \Lambda_6 \left(M_6\right)\,dt \right|\\
 	&\lesssim  N^{-3}\left\|I (\eta u) \right\|_{X^{0,\frac{1}{2}+}}^6,
 	\end{align*}
 	where $\eta$ is a smooth time cutoff function which satisfies $\eta =1 $ on $[0,\delta]$.
 \end{lemma}
 Before proving Lemma \ref{lem:almost conservation law}, we simplify the notation. Let us suppress the smooth time cut-off function $\eta $ from $\eta u$ and simply denote them by $u$.
 \begin{proof}
 	We may assume the functions $\widehat{u}_j$ are nonnegative. We project $\widehat{u}_j$ onto a dyadic piece $|\xi_j|\sim N_j$ where $N_j=2^{k_j} $ is a dyadic number for $k_j\in \left\{0,1,\dots \right\}$. Note that we are not decomposing the frequencies $\abs{\xi}\leq 1 $ here. 
 	By symmetry, it suffices to show that
 	
 	\begin{align*}
 	&\left| \int_{0}^{\delta} \Lambda_6\left(   \frac{m^2\left(\min\left(N_i,N_{jkl}\right)\right)}{\left(N+N_1\right)\left(N+N_2\right)\left(N+N_3\right)\left(N+N_{456}\right) m\left(N_1\right)\cdots m\left(N_6\right)} \right)\,dt   \right|\\
 	\lesssim& \, N^{-3} \left\| u\right\|_{X^{0,\frac{1}{2}+}}^6
 	\end{align*}
 	By symmetry, we may assume $N_1\geq N_2 \geq N_3 $ and $N_4\geq N_5 \geq N_6$. Observe that $M_4$ vanishes if $|\xi_i|\ll N$ for $i=1,2,3,4$. Therefore, we assume at least one and hence two of the $N_i \gtrsim N$.
 	Note that 
 	\begin{align*}
 	&\left| \int_{0}^{\delta} \Lambda_6\left(   \frac{m^2\left(\min\left(N_i,N_{jkl}\right)\right)}{\left(N+N_1\right)\left(N+N_2\right)\left(N+N_3\right)\left(N+N_{456}\right) m\left(N_1\right)\cdots m\left(N_6\right)}\right)  \,dt \right|\\
 	\lesssim & \left| \int_{0}^{\delta} \Lambda_6\left( \prod\limits_{j=1}^3  \frac{1}{\left(N+N_j\right)m\left(N_j\right)} \frac{1}{m\left(N_4\right)m\left(N_5\right)m\left(N_6\right) } \right)\,dt   \right|.
 	\end{align*}
 	Note that with $s=-\frac{1}{2}$, $\frac{1}{\left(N+N_i\right)m\left(N_i\right)}\lesssim N^{-\frac{1}{2}}N_i^{-\frac{1}{2}}$. We decompose the cases as follows:
 	
 	\smallskip
 	\noindent \textbf{Case 1}. $N\gg N_4 \geq N_5\geq N_6$.\\    
 	\textbf{Case 2.} $N_4\gtrsim N \gg N_5\geq N_6$.\\
 	\textbf{Case 3.} $N_4\geq N_5 \gtrsim N \gg N_6$.\\
 	\textbf{Case 4.} $N_4\geq N_5 \geq N_6 \gtrsim N$.\\

 	\textbf{Case 1.} $N \gg N_4 \geq N_5 \geq N_6$.
 	Since the two largest frequncies should be comparable and one of $N_1, N_2, N_3$ and  $N_{456 }$ must be $N_i\gtrsim N$, we have 
 	$N_1\sim N_2 \gtrsim N \gg N_4\geq N_5 \geq N_6 $.
 	In this case, we have $m\left(N_4\right)=m\left(N_5\right)=m\left(N_6\right)=1$. Hence, we consider the integral
 	\begin{align*}
 	\sum\limits_{N_1,\dots,N_6 \geq 1} N^{-\frac{3}{2}} \left|  \int_0^\delta \Lambda_6 \left( N_{1}^{-\frac{1}{2}} N_{2}^{-\frac{1}{2}}  N_{3}^{-\frac{1}{2}} \right)\,dt  \right|.
 	\end{align*}
 	Note that $\frac{N_1}{N^{\frac{1}{2}}N_4^{\frac{1}{2}}} \frac{N_2}{N^{\frac{1}{2}}N_5^{\frac{1}{2}}} \frac{N_2}{N^{\frac{1}{2}}N_6^{\frac{1}{2}}} \gtrsim 1 $. By multiplying this term, it suffices to bound
 	\begin{align*}
 	&\sum\limits_{N_1,\dots,N_6 \geq 1} N^{-3} \left|  \int_0^\delta \Lambda_6 \left( N_{1}^{-\frac{1}{2}} N_{2}^{-\frac{1}{2}}  N_{3}^{-\frac{1}{2}} N_1N_2N_2 N_{4}^{-\frac{1}{2}}N_{5}^{-\frac{1}{2}}  N_{6}^{-\frac{1}{2}}     \right)\,dt  \right|\\
 	&\lesssim N^{-3} \left\| u \right\|_{X^{0,\frac{1}{2}+}}^6.
 	\end{align*}
 	By the bilinear Strichartz estimate $\eqref{eqn:bilinear X^{s,b}}$ and Bernstein \eqref{Bs}, we have
 	\begin{align*}
 	&\sum\limits_{N_1,\dots,N_6 \geq 1} N^{-3} \left|  \int_0^\delta \Lambda_6 \left( N_{1}^{-\frac{1}{2}} N_{2}^{-\frac{1}{2}}  N_{3}^{-\frac{1}{2}} N_1N_2N_2 N_{4}^{-\frac{1}{2}}N_{5}^{-\frac{1}{2}}  N_{6}^{-\frac{1}{2}}     \right)\,dt  \right|\\
 	\lesssim & \sum\limits_{N_1,\dots,N_6 \geq 1} N^{-3} \left( N_{1}^{-\frac{1}{2}} N_{2}^{-\frac{1}{2}}  N_{3}^{-\frac{1}{2}} N_1N_2N_2 N_{4}^{-\frac{1}{2}}N_{5}^{-\frac{1}{2}}  N_{6}^{-\frac{1}{2}}     \right)\\
 	&\hspace{40mm}\times \Vert u_{N_1} u_{N_4} \Vert_{L_{t,x}^2} \Vert u_{N_2}u_{N_5} \Vert_{L_{t,x}^2} \Vert u_{N_3} \Vert_{L_{t,x}^{\infty}} \Vert u_{N_6} \Vert_{L_{t,x}^{\infty}}\\
 	\lesssim&  \sum\limits_{N_1,\dots,N_6 \geq 1}  N^{-3} \left( N_{1}^{-\frac{1}{2}} N_{2}^{-\frac{1}{2}}   N_1N_2N_2 N_{4}^{-\frac{1}{2}}N_{5}^{-\frac{1}{2}}       \right)N_{1}^{-\frac{3}{2}}N_2^{-\frac{3}{2}} \\
 	&\hspace{20mm}\times \Vert u_{N_1} \Vert_{X^{0,\frac{1}{2}+}} \Vert u_{N_4} \Vert_{X^{0,\frac{1}{2}+}}  \Vert u_{N_2} \Vert_{X^{0,\frac{1}{2}+}}  \Vert u_{N_5} \Vert_{X^{0,\frac{1}{2}+}} \Vert u_{N_3} \Vert_{L_t^{\infty}L_x^2} \Vert u_{N_6} \Vert_{L_{t}^{\infty}L_x^2}\\
 	\lesssim& \sum\limits_{N_1,\dots,N_6 \geq 1} N^{-3} N_{1}^{-1}N_{4}^{-\frac{1}{2}}N_{5}^{-\frac{1}{2}}  \Vert u \Vert_{X^{0,\frac{1}{2}+}}^6\\
 	\lesssim& \; N^{-3} \left\| u \right\|_{X^{0,\frac{1}{2}+}}^6.
 	\end{align*}


 	\textbf{Case 2.} $N_4\gtrsim N \gg N_5\geq N_6 $. In this case, we have $m(N_5)=m(N_6)=1$.
 	Since the two largest frequencies must be comparable, we obtain $N_1\gtrsim N_4$ and hence $N_1 \gtrsim N \gg N_5 \geq N_6 $.
 	We consider the integral
 	\begin{align*}
 	\sum\limits_{N_1,\dots,N_6 \geq 1} N^{-2} \left|  \int_0^\delta \Lambda_6 \left( N_{1}^{-\frac{1}{2}} N_{2}^{-\frac{1}{2}}  N_{3}^{-\frac{1}{2}} N_4^{\frac{1}{2}}\right)\,dt  \right|.
 	\end{align*}
 	Note that $\frac{N_1}{N^{\frac{1}{2}}N_5^{\frac{1}{2}}} \frac{N_4}{N^{\frac{1}{2}}N_6^{\frac{1}{2}}} \gtrsim 1 $. 
 	Then, by multiplying this term, it suffices to bound
 	\begin{align*}
 	&\sum\limits_{N_1,\dots,N_6 \geq 1} N^{-3} \left|  \int_0^\delta \Lambda_6 \left( N_{1}^{-\frac{1}{2}} N_{2}^{-\frac{1}{2}}  N_{3}^{-\frac{1}{2}} N_1 N_4  N_{4}^{\frac{1}{2}}N_{5}^{-\frac{1}{2}}  N_{6}^{-\frac{1}{2}}     \right)\,dt  \right|\\
 	&\lesssim N^{-3} \left\| u \right\|_{X^{0,\frac{1}{2}+}}^6.
 	\end{align*}
 	From the bilinear Strichartz estimate $\eqref{eqn:bilinear X^{s,b}}$ and Bernstein \eqref{Bs}, we have
 	\begin{align*}
 	&\sum\limits_{N_1,\dots,N_6 \geq 1} N^{-3} \left|  \int_0^\delta \Lambda_6 \left( N_{1}^{-\frac{1}{2}} N_{2}^{-\frac{1}{2}}  N_{3}^{-\frac{1}{2}} N_1 N_4 N_{4}^{\frac{1}{2}}N_{5}^{-\frac{1}{2}}  N_{6}^{-\frac{1}{2}}     \right)\,dt  \right|\\
 	\lesssim & \sum\limits_{N_1,\dots,N_6 \geq 1} N^{-3} \left( N_{1}^{-\frac{1}{2}} N_{2}^{-\frac{1}{2}}  N_{3}^{-\frac{1}{2}} N_1N_4 N_{4}^{\frac{1}{2}}N_{5}^{-\frac{1}{2}}  N_{6}^{-\frac{1}{2}}     \right)\\
 	&\hspace{40mm}\times \Vert u_{N_1} u_{N_5} \Vert_{L_{t,x}^2} \Vert u_{N_4}u_{N_6} \Vert_{L_{t,x}^2} \Vert u_{N_2} \Vert_{L_{t,x}^{\infty}} \Vert u_{N_3} \Vert_{L_{t,x}^{\infty}}\\
 	\lesssim&  \sum\limits_{N_1,\dots,N_6 \geq 1}  N^{-3} \left( N_{1}^{-\frac{1}{2}}    N_1N_4 N_4^{\frac{1}{2}} N_{5}^{-\frac{1}{2}} N_{6}^{-\frac{1}{2}}       \right)N_{1}^{-\frac{3}{2}}N_4^{-\frac{3}{2}} \\
 	&\hspace{20mm} \times \Vert u_{N_1} \Vert_{X^{0,\frac{1}{2}+}} \Vert u_{N_5} \Vert_{X^{0,\frac{1}{2}+}}  \Vert u_{N_4} \Vert_{X^{0,\frac{1}{2}+}}  \Vert u_{N_6} \Vert_{X^{0,\frac{1}{2}+}} \Vert u_{N_2} \Vert_{L_t^{\infty}L_x^2} \Vert u_{N_3} \Vert_{L_{t}^{\infty}L_x^2}\\
 	\lesssim& \sum\limits_{N_1,\dots,N_6 \geq 1} N^{-3} N_{1}^{-1} N_{5}^{-\frac{1}{2}}N_{6}^{-\frac{1}{2}} \left\| u \right\|_{X^{0,\frac{1}{2}+}}^6 \\
 	\lesssim& \; N^{-3} \left\| u \right\|_{X^{0,\frac{1}{2}+}}^6.
 	\end{align*}
 	
 	\textbf{Case 3.} $N_4 \geq N_5 \gtrsim N \gg N_6$. In this case we have $m(N_6)=1$.
 	We consider a integral
 	\begin{align*}
 	&\sum\limits_{N_1,\dots,N_6 \geq 1} N^{-\frac{5}{2}} \left|  \int_0^\delta \Lambda_6 \left( N_{1}^{-\frac{1}{2}} N_{2}^{-\frac{1}{2}}  N_{3}^{-\frac{1}{2}} N_{4}^{\frac{1}{2}}N_{5}^{\frac{1}{2}}     \right)\,dt  \right|. 
 	\end{align*}
 	Note that $\frac{N_5}{N^{\frac{1}{2}}N_6^{\frac{1}{2}} }\gtrsim 1$. Therefore, by multiplying this term, it suffices to bound
 	\begin{align*}
 	&\sum\limits_{N_1,\dots,N_6\geq 1}N^{-3}\left| \int_0^{\delta} \Lambda_6\left(N_1^{-\frac{1}{2}}N_2^{\frac{1}{2}} N_3^{-\frac{1}{2}} N_4^{\frac{1}{2}}N_5^{\frac{1}{2}} N_5 N_6^{-\frac{1}{2}}   \right)\,dt \right|\\
 	\lesssim& N^{-3} \Vert u \Vert_{X^{0,\frac{1}{2}+}}^6.
 	\end{align*}
 	
 	\noindent \textbf{Subcase 3.a.} $N_1\gg N_4$.
 	Since the two largest frequencies must be comparable, we have $N_1\sim N_2 \gg N_4\geq N_5 \gg N_6$.
 	By using the bilinear Strichartz $\eqref{eqn:bilinear X^{s,b}}$ and Bernstein’s inequality \eqref{Bs}, we have
 	\begin{align*}
 	&\sum\limits_{N_1,\dots,N_6\geq 1}N^{-3} \int_0^\delta \Lambda_6\left( N_1^{-\frac{1}{2}} N_2^{-\frac{1}{2}}N_3^{-\frac{1}{2}}N_4^{\frac{1}{2}}N_5^{\frac{1}{2}}N_5 N_6^{-\frac{1}{2}} \right)\,dt\\
 	&\lesssim\sum\limits_{N_1,\dots,N_6\geq 1 }N^{-3}\left( N_1^{-\frac{1}{2}}N_2^{-\frac{1}{2}}N_3^{-\frac{1}{2}}N_4^{\frac{1}{2}}N_5^{\frac{1}{2}}N_5 N_6^{-\frac{1}{2}} \right)\\
 	&\hspace{10mm} \times \Vert u_{N_1} u_{N_4} \Vert_{L_{t,x}^2} \Vert u_{N_2} u_{N_5} \Vert_{L_{t,x}^2} \Vert u_{N_3} \Vert_{L_{t,x}^{\infty}} \Vert u_{N_6} \Vert_{L_{t,x}^{\infty}}\\
 	&\lesssim\sum\limits_{N_1,\dots N_6\geq 1 }N^{-3}\left(N_1^{-\frac{1}{2}}N_2^{-\frac{1}{2}}N_4^{\frac{1}{2}}N_5^{\frac{1}{2}}N_5  \right)N_1^{-\frac{3}{2}}N_2^{-\frac{3}{2}}\\
 	&\hspace{10mm} \times \Vert u_1 \Vert_{X^{0,\frac{1}{2}+}}\Vert u_4 \Vert_{X^{0,\frac{1}{2}+}}\Vert u_2 \Vert_{X^{0,\frac{1}{2}+}}\Vert u_5 \Vert_{X^{0,\frac{1}{2}+}} \Vert u_{N_3} \Vert_{L_{t}^{\infty}L_x^2}\Vert u_{N_6} \Vert_{L_t^{\infty}L_x^2}\\
 	&\lesssim \sum\limits_{N_1,\dots,N_6\geq 1}N^{-3}N_1^{-2}N_2^{-2}N_4^{\frac{1}{2}}N_5^{\frac{1}{2}}N_5\Vert u \Vert_{X^{0,\frac{1}{2}+}}^6\\
 	&\lesssim N^{-3}\Vert u_j \Vert_{X^{0,\frac{1}{2}+}}^6.
 	\end{align*} 
 	
 	\noindent \textbf{Subcase 3.b.} $N_4\gg N_1$.
 	Since the two largest freqeuncies must be comparable, we have $N_4\sim N_5 \gg N_1\geq N_2\geq N_3$.
 	By using the bilinear Stricharatz \eqref{eqn:bilinear X^{s,b}} and Bernstein’s inequality \eqref{Bs}, we have
 	\begin{align*}
 	&\sum\limits_{N_1,\dots,N_6\geq 1}N^{-3} \int_0^\delta \Lambda_6\left( N_1^{-\frac{1}{2}} N_2^{-\frac{1}{2}}N_3^{-\frac{1}{2}}N_4^{\frac{1}{2}}N_5^{\frac{1}{2}}N_5 N_6^{-\frac{1}{2}} \right)\,dt\\
 	&\lesssim\sum\limits_{N_1,\dots,N_6\geq 1 }N^{-3}\left( N_1^{-\frac{1}{2}}N_2^{-\frac{1}{2}}N_3^{-\frac{1}{2}}N_4^{\frac{1}{2}}N_5^{\frac{1}{2}}N_5 N_6^{-\frac{1}{2}} \right)\\
 	&\hspace{10mm} \times \Vert u_{N_1} u_{N_4} \Vert_{L_{t,x}^2} \Vert u_{N_2} u_{N_5} \Vert_{L_{t,x}^2} \Vert u_{N_3} \Vert_{L_{t,x}^{\infty}} \Vert u_{N_6} \Vert_{L_{t,x}^{\infty}}\\
 	&\lesssim\sum\limits_{N_1,\dots N_6\geq 1 }N^{-3}\left(N_1^{-\frac{1}{2}}N_2^{-\frac{1}{2}}N_4^{\frac{1}{2}}N_5^{\frac{1}{2}}N_5  \right)N_4^{-\frac{3}{2}}N_5^{-\frac{3}{2}}\\
 	&\hspace{10mm} \times \Vert u_1 \Vert_{X^{0,\frac{1}{2}+}}\Vert u_4 \Vert_{X^{0,\frac{1}{2}+}}\Vert u_2 \Vert_{X^{0,\frac{1}{2}+}}\Vert u_5 \Vert_{X^{0,\frac{1}{2}+}} \Vert u_{N_3} \Vert_{L_{t}^{\infty}L_x^2}\Vert u_{N_6} \Vert_{L_t^{\infty}L_x^2}\\
 	&\lesssim \sum\limits_{N_1,\dots,N_6\geq 1}N^{-3}N_4^{-1}N_5^{-1}N_5N_1^{-\frac{1}{2}}N_2^{-\frac{1}{2}} \Vert u \Vert_{X^{0,\frac{1}{2}+}}^6\\
 	&\lesssim N^{-3}\Vert u \Vert_{X^{0,\frac{1}{2}+}}^6.
 	\end{align*} 
 	
 	\noindent \textbf{Subcase 3.c.} $N_4\sim N_1$.
 	Note that $N_1\sim N_4 \geq N_5 \gg N_6$. From the bilinear Strichartz estimate \eqref{eqn:bilinear X^{s,b}}, Strichartz estimate $\eqref{eqn: X^{s,b} Strichartz}$ and Bernstein inequality \eqref{Bs}, we have 
 	\begin{align*}
 	&\sum\limits_{N_1,\dots,N_6\geq 1}N^{-3} \int_0^\delta \Lambda_6\left( N_1^{-\frac{1}{2}} N_2^{-\frac{1}{2}}N_3^{-\frac{1}{2}}N_4^{\frac{1}{2}}N_5^{\frac{1}{2}}N_5 N_6^{-\frac{1}{2}} \right)\,dt\\
 	&\lesssim\sum\limits_{N_1,\dots,N_6\geq 1 }N^{-3}\left( N_1^{-\frac{1}{2}}N_2^{-\frac{1}{2}}N_3^{-\frac{1}{2}}N_4^{\frac{1}{2}}N_5^{\frac{1}{2}}N_5 N_6^{-\frac{1}{2}} \right)\\
 	&\hspace{10mm} \times \Vert u_{N_1}u_{N_6} \Vert_{L_{t,x}^2}\Vert u_{N_2} \Vert_{L_{t}^{\infty}L_x^2} \Vert u_{N_3}\Vert_{L_{t,x}^{\infty}} \Vert u_{N_4} \Vert_{L_t^4L_x^{\infty}} \Vert u_{N_5} \Vert_{L_t^4L_x^{\infty}}\\
 	&\lesssim \sum\limits_{N_1,\dots,N_6\geq 1}N^{-3}  \left(N_1^{-\frac{1}{2}} N_2^{-\frac{1}{2}} N_5 N_6^{-\frac{1}{2}}   \right)N_1^{-\frac{3}{2}}  \Vert u \Vert_{X^{0,\frac{1}{2}+}}^6\\
 	&\lesssim N^{-3}\Vert u \Vert_{X^{0,\frac{1}{2}+}}^6.
 	\end{align*}

 	\textbf{Case 4.} $N_4\geq N_5 \geq N_6 \gtrsim N$.\\
 	
 	\noindent \textbf{Subcase 4.a.} $N_1 \gg N_4 $.	
 	Since the two largest frequencies must be comparable, we have $N_1\sim N_2 \gg N_4\geq N_5 \geq N_6 $.
 	In this case, we consider the integral
 	\begin{align*}
 	\sum\limits_{N_1,\dots ,N_6 \geq 1 }N^{-3}\int_0^\delta \Lambda_6\left( N_1^{-\frac{1}{2}}  N_2^{-\frac{1}{2}}  N_3^{-\frac{1}{2}}  N_4^{\frac{1}{2}}  N_5^{\frac{1}{2}}  N_6^{\frac{1}{2}} \right)\,dt.
 	\end{align*}
 	By the bilinear Strichartz estimate $\eqref{eqn:bilinear X^{s,b}}$ and Bernstein inequality \eqref{Bs}, we have
 	\begin{align*}
 	&\sum\limits_{N_1,\dots ,N_6 \geq 1}N^{-3}\int_0^\delta \Lambda_6\left( N_1^{-\frac{1}{2}}  N_2^{-\frac{1}{2}}  N_3^{-\frac{1}{2}}  N_4^{\frac{1}{2}}  N_5^{\frac{1}{2}}  N_6^{\frac{1}{2}}\right)\,dt\\
 	&\lesssim \sum\limits_{N_1,\dots,N_6 \geq 1} N^{-3} \left(N_1^{-\frac{1}{2}} N_2^{-\frac{1}{2}} N_3^{-\frac{1}{2}} N_4^{\frac{1}{2}} N_5^{\frac{1}{2}} N_6^{\frac{1}{2}} \right)\\
 	& \hspace{20mm} \times \Vert u_{N_1} u_{N_4} \Vert_{L_{t,x}^2} \Vert u_{N_2} u_{N_5} \Vert_{L_{t,x}^2} \Vert u_{N_3} \Vert_{L_{t,x}^{\infty}} \Vert u_{N_6} \Vert_{L_{t,x}^{\infty}}\\
 	& \lesssim  \sum\limits_{N_1,\dots,N_6 \geq 1} N^{-3} \left(N_1^{-\frac{1}{2}} N_2^{-\frac{1}{2}} N_3^{-\frac{1}{2}} N_4^{\frac{1}{2}} N_5^{\frac{1}{2}} N_6^{\frac{1}{2}} \right) N_1^{-\frac{3}{2}} N_2^{-\frac{3}{2}} N_{3}^{\frac{1}{2}}N_{6}^{\frac{1}{2}}\\
 	& \hspace{20mm} \times \Vert u_{N_1} \Vert_{X^{0,\frac{1}{2}+}}	\Vert u_{N_4} \Vert_{X^{0,\frac{1}{2}+}} \Vert u_{N_2} \Vert_{X^{0,\frac{1}{2}+}} \Vert u_{N_5} \Vert_{X^{0,\frac{1}{2}+}} \Vert u_{N_3} \Vert_{L_{t}^{\infty}L_x^2} \Vert u_{N_6} \Vert_{L_t^{\infty} L_x^2 }\\
 	& \lesssim \sum\limits_{N_1,\dots,N_6 \geq 1} N^{-3} N_1^{-2}N_2^{-2}N_4^{\frac{1}{2}}N_5^{\frac{1}{2}}N_6^{1}\Vert u \Vert_{X^{0,\frac{1}{2}+}}^6\\
 	& \lesssim N^{-3}  \Vert u \Vert_{X^{0,\frac{1}{2}+}}^6.
 	\end{align*}

 	\noindent \textbf{Subcase 4.b} $N_4 \gg N_1 $.
 	Since the two largest frequencies must be comparable, we have $N_4\sim N_5 \gg N_1\geq N_2 \geq N_3 $. In this case, it suffices to consider the integral 
 	\begin{align*}
 	\sum\limits_{N_1, \dots, N_6 \geq 1}N^{-3}\int_0^\delta \Lambda_6\left( N_1^{-\frac{1}{2}}  N_2^{-\frac{1}{2}}  N_3^{-\frac{1}{2}}  N_4^{\frac{1}{2}}  N_5^{\frac{1}{2}}  N_6^{\frac{1}{2}} \right)\,dt.
 	\end{align*}
 	By the bilinear Strichartz estimate $\eqref{eqn:bilinear X^{s,b}}$ and Bernstein’s inequality \eqref{Bs}, we have
 	\begin{align*}
 	&\sum\limits_{N_1,\dots , N_6\geq 1}N^{-3}\int_0^\delta \Lambda_6\left( N_1^{-\frac{1}{2}}  N_2^{-\frac{1}{2}}  N_3^{-\frac{1}{2}}  N_4^{\frac{1}{2}}  N_5^{\frac{1}{2}}  N_6^{\frac{1}{2}} \right)\,dt\\
 	&\lesssim \sum\limits_{N_1,\dots,N_6 \geq 1} N^{-3} \left(N_1^{-\frac{1}{2}} N_2^{-\frac{1}{2}} N_3^{-\frac{1}{2}} N_4^{\frac{1}{2}} N_5^{\frac{1}{2}} N_6^{\frac{1}{2}} \right)\\
 	& \hspace{20mm} \times \Vert u_{N_1} u_{N_4} \Vert_{L_{t,x}^2} \Vert u_{N_2} u_{N_5} \Vert_{L_{t,x}^2} \Vert u_{N_3} \Vert_{L_{t,x}^{\infty}} \Vert u_{N_6} \Vert_{L_{t,x}^{\infty}}\\
 	& \lesssim  \sum\limits_{N_1,\dots,N_6 \geq 1} N^{-3} \left(N_1^{-\frac{1}{2}} N_2^{-\frac{1}{2}} N_3^{-\frac{1}{2}} N_4^{\frac{1}{2}} N_5^{\frac{1}{2}} N_6^{\frac{1}{2}} \right) N_4^{-\frac{3}{2}} N_5^{-\frac{3}{2}} N_{3}^{\frac{1}{2}}N_{6}^{\frac{1}{2}}\\
 	& \hspace{20mm} \times \Vert u_{N_1} \Vert_{X^{0,\frac{1}{2}+}}	\Vert u_{N_4} \Vert_{X^{0,\frac{1}{2}+}} \Vert u_{N_2} \Vert_{X^{0,\frac{1}{2}+}} \Vert u_{N_5} \Vert_{X^{0,\frac{1}{2}+}} \Vert u_{N_3} \Vert_{L_{t}^{\infty}L_x^2} \Vert u_{N_6} \Vert_{L_t^{\infty} L_x^2 }\\
 	& \lesssim \sum\limits_{N_1,\dots,N_6 \geq 1} N^{-3} N_4^{-3}N_1^{-\frac{1}{2}}N_2^{-\frac{1}{2}}N_4^{\frac{1}{2}}N_5^{\frac{1}{2}}N_6^{1}\Vert u \Vert_{X^{0,\frac{1}{2}+}}^6\\
 	& \lesssim N^{-3} \Vert u \Vert_{X^{0,\frac{1}{2}+}}^6.
 	\end{align*}
 	
 	\noindent \textbf{Subcase 4.c} $N_4 \sim  N_1 $. We consider the integral 
 	\begin{align*}
 	N^{-3}\int_0^\delta \Lambda_6\left( N_1^{-\frac{1}{2}}  N_2^{-\frac{1}{2}}  N_3^{-\frac{1}{2}}  N_4^{\frac{1}{2}}  N_5^{\frac{1}{2}}  N_6^{\frac{1}{2}} \right)\,dt.
 	\end{align*}
 	\noindent
 	From the Strichartz estimate \eqref{eqn: X^{s,b} Strichartz}, we have
 	\begin{align*}
 	&N^{-3}\int_0^\delta \Lambda_6\left( N_1^{-\frac{1}{2}}  N_2^{-\frac{1}{2}}  N_3^{-\frac{1}{2}}  N_4^{\frac{1}{2}}  N_5^{\frac{1}{2}}  N_6^{\frac{1}{2}} \right)\,dt\\
 	& \lesssim  \sum\limits_{N_1,\dots,N_6 \geq 1} N^{-3} N_1^{-\frac{1}{2}} N_2^{-\frac{1}{2}} N_3^{-\frac{1}{2}} N_4^{\frac{1}{2}} N_5^{\frac{1}{2}} N_{6}^{\frac{1}{2}}\\
 	&\hspace{10mm} \times \Vert u_{N_1} \Vert_{L_t^4L_x^{\infty}} \Vert u_{N_2} \Vert_{L_t^{\infty}L_x^2} \Vert u_{N_3} \Vert_{L_t^\infty L_x^2}  \Vert u_{N_4} \Vert_{L_t^4L_x^{\infty}}  \Vert u_{N_5} \Vert_{L_t^4L_x^{\infty}}   \Vert u_{N_6} \Vert_{L_t^4L_x^{\infty}}\\
 	&\lesssim \sum\limits_{N_1,\dots,N_6 \geq 1} N^{-3} N_1^{-1}N_2^{-\frac{1}{2}}N_3^{-\frac{1}{2}}  \Vert u \Vert_{X^{0,\frac{1}{2}+}}^6\\
 	&\lesssim N^{-3}   \Vert u \Vert_{X^{0,\frac{1}{2}+}}^6.
 	\end{align*}
 	
 	\noindent
 	This completes the proof of Lemma \ref{lem:almost conservation law}. 
 	
 \end{proof}

 \subsection{Correction term estimate}
 The following  lemma shows that $E_I^4(t)$ is very close to $E_I^2(t)$.
 \begin{lemma}\label{LEM:cor}
 	Let $I$ be defined with the multiplier $m$ of the form (\ref{Im}) and $s=-\frac{1}{2}$. Then, 
 	\begin{align*}
 	\left|E_I^4(t)-E_I^2(t) \right|\lesssim \left\| Iu(t)\right\|_{L_x^2}^4.
 	\end{align*}
 \end{lemma}
 
 \begin{proof}
 	Note that $E_I^4(t)=E_I^2(t)+\Lambda_4\left(\sigma_4\right)$ and 
 	\begin{align*}
 	\left|\sigma_4\left(\xi_1,\dots,\xi_4\right)  \right|\lesssim \frac{m^2\left(\min\left(N_i\right)\right)}{\left(N+N_1\right) \left(N+N_2\right) \left(N+N_3\right) \left(N+N_4\right)}
 	\end{align*}
 	with $\left|\xi_i\right| \sim N_i $.
 	We restrict our attention to the contribution arising from $\left|\xi_i\right| \sim N_i $. It suffices to bound
 	\begin{align*}
 	\left| \Lambda_4 \left(\sigma_4 \right)\right|\lesssim \prod\limits_{j=1}^4 \left\|Iu_j(t)  \right\|_{L_x^2}.
 	\end{align*} 
 	We may assume the functions $\widehat{u}_j$ are nonnegative. We project $\widehat{u}_j$ to a dyadic piece $|\xi_j|\sim N_j$ where $N_j=2^{k_j} $ is a dyadic number for $k_i\in \left\{0,1,\dots \right\}$. Note that we are not decomposing the frequencies $\abs{\xi}\leq 1 $ here.
 	We may assume $N_1\geq N_2 \geq N_3 \geq N_4$ with $N_1\sim N_2 \gtrsim N $.
 	Therefore, for $s=-\frac{1}{2}$, it suffices to bound
 	\begin{equation}
 	\begin{split}
 	\label{eqn:E_I^4 estimate}
 	&\sum\limits_{N_1,\dots N_4 \geq 1 }\Lambda_4 \left(\frac{1}{\left(N+N_1\right) \left(N+N_2\right) \left(N+N_3\right) \left(N+N_4\right)}  \frac{1}{m\left(N_1\right) m\left(N_2\right) m\left(N_3\right) m\left(N_4\right)} \right)\\
 	&\lesssim \prod\limits_{j=1}^4 \Vert u_j \Vert_{L_x^2}.
 	\end{split}
 	\end{equation}
 	The multiplier appearing in $\eqref{eqn:E_I^4 estimate}$ is 
 	\begin{align*}
 	\frac{N^{4s}}{N_{1}^{1+s} N_{2}^{1+s} N_{3}^{1+s} N_{4}^{1+s}}\lesssim N_1^{-\frac{1}{2}}N_2^{-\frac{1}{2}} N_3^{-\frac{1}{2}} N_4^{-\frac{1}{2}}.
 	\end{align*}
 	Then, by Bernstein inequality \eqref{Bs}, we have 
 	\begin{align*}
 	\left|\Lambda_4\left(\sigma_4\right) \right|&\lesssim  \sum\limits_{N_1\geq N_2\geq N_3\geq N_4\geq 1}N_1^{-\frac{1}{2}}N_2^{-\frac{1}{2}} N_3^{-\frac{1}{2}} N_4^{-\frac{1}{2}}   \left\| u_{N_1}\right\|_{L_x^2} \left\| u_{N_2}\right\|_{L_x^2} \left\| u_{N_3}\right\|_{L_x^\infty} \left\| u_{N_4}\right\|_{L_x^\infty}\\
 	&\lesssim  \sum\limits_{N_1\geq N_2\geq N_3\geq N_4\geq 1}N_1^{-\frac{1}{2}}N_2^{-\frac{1}{2}} \left\|u_1 \right\|_{L_x^2} \left\|u_2 \right\|_{L_x^2} \left\|u_3 \right\|_{L_x^2} \left\|u_4 \right\|_{L_x^2}\\
 	&\lesssim  \prod\limits_{j=1}^4 \left\|u_j \right\|_{L_x^2}.
 	\end{align*}
 	
 	\noindent
 	This completes the proof of Lemma \ref{LEM:cor}.
 \end{proof}
 
 \subsection{Well-posedness of $I$-system}
 Before we construct a global solution, we show the following modified local well-posedness.
 \begin{prop}
 	Consider the initial value problem
 	\begin{align*}
 	\begin{cases}
 	i\partial_t Iu+I\partial_x^4u=I(|u|^2u) \\ Iu(0,x)=Ig(x)\in L^2\left(\mathbb{R}\right).
 	\end{cases}
 	\end{align*}
 	Then, the initial value problem is locally well-posed (in $L^2\left(\mathbb{R}\right)$) on an interval $[-\delta,\delta]$ with $\delta\sim \| Ig\|_{L^2}^{-\alpha}$ for some $\alpha>0$ and satisfies the bound
 	\begin{align*}
 	\left\| I (\eta u)  \right\|_{X^{0,\frac{1}{2}+}}\lesssim \left\| Ig  \right\|_{L_x^2},    
 	\end{align*}
 	
 	\noindent	
 	where $\eta$ is a smooth time cutoff function which satisfies $\eta =1 $ on $[-\delta,\delta]$.
 \end{prop}

 \begin{proof}
 	Let us suppress the smooth time cut-off function $\eta $ from $\eta u$ and simply denote them by $u$.
 	It suffices to show the following trilinear estimate
 	\begin{align*}
 	\Vert I(\abs{u}^2u) \Vert_{X^{0,-\frac{1}{2}+}}\lesssim \Vert Iu \Vert_{X^{0,\frac{1}{2}+}}^3.
 	\end{align*}	
 	Note that 
 	\begin{align*}
 	\Vert I(\abs{u}^2u) \Vert_{X^{0,-\frac{1}{2}+}}=&\sup\limits_{\| v\|_{X^{0,1-b}}=1}\left| \int_{\mathbb{R}\times\mathbb{R}}I(u\bar{u}u) v\,dxdt \right|\\
 	=&\sup\limits_{\| v\|_{X^{0,\frac{1}{2}- }}=1}\left| \int_{\substack {\xi_1+\dots \xi_4=0\\ \tau_1+\dots \tau_4=0}} m(\xi_4) \widehat{u_1}\left(\tau_1,\xi_1 \right) \widehat{\bar{u_2}}\left(\tau_2,\xi_2 \right)\widehat{u_3}\left(\tau_3,\xi_3\right)
 	\widehat{v}\left(\tau_4,\xi_4 \right)\right|.
 	\end{align*}
 	Now set
 	\begin{align*}
 	&f_1\left(\tau_1,\xi_1\right)=\left| \widehat{u_1}\left(\tau_1,\xi_1\right)\right|  \left \langle \tau_1+\xi_1^4 \right\rangle^{\frac{1}{2}+}m(\xi_1),\\
 	&f_2\left(\tau_2,\xi_2\right)=\left| \widehat{\overline{u_2}}\left(\tau_2,\xi_2\right)\right|  \left \langle \tau_2-\xi_2^4 \right\rangle^{\frac{1}{2}+}m(\xi_2),\\ 
 	&f_3\left(\tau_3,\xi_3\right)=\left| \widehat{u_1}\left(\tau_3,\xi_3\right)\right| \left \langle \tau_3+\xi_3^4 \right\rangle^{\frac{1}{2}+}m(\xi_3),\\
 	&f_4\left(\tau_4,\xi_4\right)=\left| \widehat{v}\left(\tau_4,\xi_4\right)\right|  \left \langle \tau_4-\xi_4^4 \right\rangle^{\frac{1}{2}-}.   
 	\end{align*}
 	Therefore, it suffices to show that for nonnegative $L^2$ functions $f_j$  
 	\begin{align*}
 	\int\limits_{\substack{\xi_1+\dots+\xi_4=0 \\ \tau_1+\dots+\tau_4=0}}\frac{m(\xi_4)\left\langle \tau_4-\xi_4^4 \right\rangle^{-\frac{1}{2}+ }}{m(\xi_1)m(\xi_2)m(\xi_3) \left\langle \tau_1+\xi_1^4 \right\rangle^{\frac{1}{2}+}\left\langle \tau_2-\xi_2^4 \right\rangle^{\frac{1}{2}+}\left\langle \tau_3+\xi_3^4 \right\rangle^{\frac{1}{2}+}}\prod\limits_{j=1}^{4}f_j\left(\tau_j,\xi_j\right)\\
 	\lesssim \prod\limits_{j=1}^{4}\|f_j \|_{L^2_{\tau,\xi}}.
 	\end{align*}
 	Note that $m(\xi)\langle \xi \rangle^{\frac{1}{2}}$ is increasing in $\xi$ and $m(\xi)\langle \xi \rangle^{\frac{1}{2}}\gtrsim 1 $ for all $ \xi \in \mathbb{R}$. By symmetry, we may assume that $\abs{\xi_1}\geq \abs{\xi_2}\geq \abs{\xi_3}$. Also, $\xi_1+\dots +\xi_4=0 $ implies $\abs{\xi_{\max}}\sim \abs{\xi_{sub}} $ and $\abs{\xi_1}\gtrsim \abs{\xi_4}$.
 	Note that
 	\begin{align*}
 	\frac{m(\xi_4) \langle \xi_4 \rangle^{\frac{1}{2}} }{m(\xi_1)\langle \xi_1 \rangle^{\frac{1}{2}}  m(\xi_2)\langle \xi_2 \rangle^{\frac{1}{2}}  m(\xi_3)\langle \xi_3 \rangle^{\frac{1}{2}}  } &\lesssim \frac{1}{m(\xi_2)\langle \xi_2 \rangle^{\frac{1}{2}} m(\xi_3)\langle \xi_3 \rangle^{\frac{1}{2}} }
 	\\ &\lesssim \; 1.
 	\end{align*}
 	Therefore, it is enough to show that
 	\begin{align*}
 	\int\limits_{\substack{\xi_1+\dots+\xi_4=0 \\ \tau_1+\dots+\tau_4=0}} \frac{\left\langle \xi_1 \right\rangle^{\frac{1}{2}} \left\langle \xi_2 \right\rangle^{\frac{1}{2}}  \left\langle \xi_3 \right\rangle^{\frac{1}{2}} \left\langle \tau_4-\xi_4^4 \right \rangle^{-\frac{1}{2}+ }  }{\langle \xi_4 \rangle^{\frac{1}{2}} \left\langle \tau_1+\xi_1^4 \right\rangle^{\frac{1}{2}+}\left\langle \tau_2-\xi_2^4 \right\rangle^{\frac{1}{2}+}\left\langle \tau_3+\xi_3^4 \right\rangle^{\frac{1}{2}+}}\prod\limits_{j=1}^4 f_j\left(\tau_j,\xi_j\right)\\
 	\lesssim \prod\limits_{j=1}^{4}\|f_j \|_{L^2_{\tau,\xi}}.
 	\end{align*}
 	Hence, the trilinear estimate follows directly by just repeating the proof of Proposition \ref{prop:t1}.
 \end{proof}

 \subsection{Proof of global well-posedness for \eqref{eqn:fourth order NLS}}
 In this subsection, we prove global well-posedness of \eqref{eqn:fourth order NLS}. For any given $u_0\in H^{s}, s\geq -\frac{1}{2}$ and $T>0$, our goal is to construct a solution on $[0,T]$. Suppose that $u$ is a solution to \eqref{eqn:fourth order NLS} with initial data $u_0$. Then, for any $\lambda>0$, $u_{\lambda}(x,t)=\lambda^{-2}u(\frac{x}{\lambda},\frac{t}{\lambda^4})$
 is also a solution to \eqref{eqn:fourth order NLS} with initial data $u_{0,\lambda}=\lambda^{-2}u_0(\frac{x}{\lambda})$. Recall that 
 \begin{align*}
 \| u\|_{H^s}\lesssim \| Iu \|_{L^2}\lesssim N^{-s} \| u\|_{H^s}.   
 \end{align*}
 Straightforward calculation shows 
 \begin{align*}
 \Vert Iu_{0,\lambda} \Vert_{L^2} \lesssim \lambda^{-\frac{3}{2}-s}N^{-s}\Vert u_0 \Vert_{H^s}.
 \end{align*} 
 The parameter $N$ will be chosen later but we take $\lambda $ now such that
 \begin{align*}
 \lambda^{-\frac{3}{2}-s}N^{-s}\left\| u_0\right\|_{H^s}=\varepsilon_0<1 \quad \Rightarrow \quad \lambda\sim N^{-\frac{2s}{3+2s}} .    
 \end{align*} 
 For simplicity of notations, we still denote $u_\lambda$ by $u, u_{0,\lambda} $ by $u_0$ and assume $\Vert Iu_0 \Vert_{L^2}\leq \varepsilon_0 $. The goal is to construct solutions on $[0,\lambda^4T]$. We already have the local solution on $[0,1]$ and need to extend the solution on $[0,\lambda^4T]$. It suffices to control the modified energy $E_I^2(t)=\Vert Iu \Vert_{L^2}^2$. 
 
 First, we control $E_I^2(t)$ for $t \in [0,1]$. By a standard bootstraping argument, we may assume $E_I^2(t)<4\varepsilon_0^2$ for $t\in [0,1]$. By Lemma \ref{LEM:cor}, we know that modified energies $E_I^2(t)$ and $E_I^4(t)$ are very close, i.e
 \begin{align*}
 E_I^4(0)=E_I^2(0)+O(\varepsilon_0^4),
 \end{align*}
 and
 \begin{align*}
 E_I^4(t)=E_I^2(t)+O(\varepsilon_0^4).
 \end{align*}  
 for $t\in [0,1]$.
 
 Also, it follows from Lemma \ref{lem:almost conservation law} that for all $t\in [0,1]$, we have
 \begin{align*}
 E_I^4(t)\leq E_I^4(0)+C\varepsilon_0^6N^{-3}.
 \end{align*}
 Then, our rescaled solution satisfies
 \begin{align*}
 \Vert u(1) \Vert_{H^s}^2\leq \Vert Iu(1) \Vert_{L^2}^2=E_I^4(1)+O(\varepsilon_0^4)\leq& E_I^4(0)+C\varepsilon_0^6N^{-3}+O(\varepsilon_0^4)\\
 \leq& \varepsilon_0^2+C\varepsilon_0^6N^{-3}+O(\varepsilon_0^4)\\
 <&4\varepsilon_0^2. 
 \end{align*}
 
 Now it suffices to do an iteration. We now consider the initial value problem for \eqref{eqn:fourth order NLS} with initial data $u(1)$. By the above bound $\Vert Iu(1) \Vert_{L^2}^2<4\varepsilon_0^2$, the local solution extend on $[0,2]$. By iterating this procedure $M$ steps, we obtain 
 \begin{align*}
 E_I^4(t)\leq E_I^4(0)+MC\varepsilon_0^6N^{-3}.
 \end{align*}
 for all $t \in [0,M+1]$. As long as $MN^{-3}\sim 1 $, we have the bound
 \begin{align*}
 \Vert u(M) \Vert_{H^s}^2\leq \Vert Iu(M) \Vert_{L^2}^2=&E_I^4(t)+O(\varepsilon_0^4)\\
 \leq& \varepsilon_0^2+O(\varepsilon_0^4)+CM\varepsilon_0^6N^{-3}<4\varepsilon_0^2, 
 \end{align*}
 and the lifetime span of the local result remains uniformly of size 1.
 We choose $M \sim N^{3}$. This process extends the local solution to the time interval $[0,N^{3}]$. Take $N$ such that 
 
 \begin{align*}
 N^{3}\sim \lambda^4T \sim N^{-\frac{8s}{3+2s}}T \quad \text{or} \quad N^{3+\frac{8s}{3+2s}}\sim T,    
 \end{align*}
 which is certainly done for $s\geq-\frac{1}{2}$. Therefore, the solution $u$ is extended on $[0,\lambda^4T]$ for arbitrary fixed time $T$. This completes the proof of global well-posedness for \eqref{eqn:fourth order NLS} in $H^s(\mathbb{R}), s\geq -\frac{1}{2} $.
 
 In the end of this section, we prove some properties of the global solution. By rescaling of our global solution, we get 
 \begin{align*}
 \sup\limits_{[0,T]}\Vert u(t) \Vert_{H^s}\sim \lambda^{\frac{3}{2}+s}\sup\limits_{t\in [0,\lambda^4 T]}\Vert u_{\lambda}(t) \Vert_{H_x^s} \leq \lambda^{\frac{3}{2}+s}\sup\limits_{t\in[0,\lambda^4T]} \Vert Iu_{\lambda}(t) \Vert_{L^2},
 \end{align*}
 and
 \begin{align*}
 \Vert Iu_{0,\lambda} \Vert_{L^2}\lesssim N^{-s} \Vert u_{0,\lambda} \Vert_{H^s} \sim N^{-s}\lambda^{-\frac{3}{2}-s}\Vert u_0 \Vert_{H^s}.
 \end{align*}
 From the above local well-posedness iteration argument, we have
 \begin{align*}
 \sup\limits_{t\in [0,\lambda^4T]}\Vert Iu_{\lambda}(t) \Vert_{L^2} \lesssim \Vert Iu_{0,\lambda} \Vert_{L^2},
 \end{align*}
 and hence
 \begin{align*}
 \sup\limits_{t\in [0,T]} \Vert u(t) \Vert_{H^s} \lesssim N^{-s}\Vert u_0 \Vert_{H^s}.
 \end{align*}
 Recall that the parameter $\lambda $ was chosen such that $\Vert Iu_{0,\lambda} \Vert_{L^2}\sim \varepsilon_0$. Thus we get $\lambda\sim N^{-\frac{2s}{3+2s}}$. Also the parameter $N$ is chosen such that $N^{3}\sim \lambda^4 T$ or $ N^{\frac{14s+9}{3+2s}}\sim T$. This shows that the selection of $N$ is polynomial in $T$. So this gives a polynomial growth bound on $\Vert u(t) \Vert_{H^s}$
 \begin{align*}
 \Vert u(t) \Vert_{H^{-\frac{1}{2}}}\lesssim t^{\frac{1}{2}}\Vert u_0 \Vert_{H^{-\frac{1}{2}}}.
 \end{align*}
 


 \section{Mild ill-posedness}\label{sec:ill-posedness}
 In this section we give the proof of Theorem $\ref{thm: ill-posedness}$. We follow the argument
 introduced by Christ-Colliander-Tao \cite{CCT2003}.  In \cite{CCT2003}, Christ-Colliander-Tao established two solutions to \eqref{eqn: cubic NLS,focusing defocusing}  breaking the uniform continuity of the flow map for $s<0$. The method presented in Christ-Colliander-Tao \cite{CCT2003} can be applied to both focusing and defocusing cases. They used the Galilean and scale invariances to construct the counter example. For the mKdV case, instead of using Galilean invariance, they used the approximate solution to the mKdV equation by exploiting the solution of \eqref{eqn: cubic NLS,focusing defocusing}.
 Likewise, we approximate \eqref{eqn:fourth order NLS} by \eqref{eqn: cubic NLS,focusing defocusing} to prove mild ill-posedness.

 \subsection{Mild ill-posedness result for the cubic NLS}
 First, we state the ill-posedness result for \eqref{eqn: cubic NLS,focusing defocusing} in \cite{CCT2003}.
 \begin{theorem}[\cite{CCT2003}]\label{thm: ill-posedness defocsuing NLS}
 	The solution map of the cubic NLS  in $H^s$ for $s<0$ fails to be uniformly continuous. More precisely, there exists $\varepsilon_0>0$ such that for any $\delta>0,T>0$ and $\varepsilon<\varepsilon_0$, there are two solutions $u_1,u_2$ to $\eqref{eqn: defocusing cubic NLS}$ such that
 	\begin{align}
 	\Vert u_1(0)\Vert_{H^s},\Vert u_2(0) \Vert_{H^s}&\lesssim \varepsilon, \label{ean: ill posedness small data}\\
 	\Vert u_1(0)-u_2(0)\Vert_{H^s}&\lesssim \delta,\label{eqn: ill posedness small difference} \\
 	\sup\limits_{0\leq t <T}\Vert u_1(t)-u_2(t) \Vert_{H^s}&\gtrsim \varepsilon. \label{eqn: ill-posedness difference lower bound}
 	\end{align}		 
 \end{theorem} 
 This implies the solution map is not uniformly continuous from the ball 
 \\$\left\{u_0 \in H_x^s : \Vert u_0 \Vert_{H_x^s} \lesssim \varepsilon \right\}$ to $C_t^\infty\left([0,T];H_x^s \right)$.
 

 \begin{remark}
 	In the case of focusing cubic NLS, there is another way to show mild ill-posedness. Kenig-Ponce-Vega \cite{KPV2001} demonstrated mild ill-posedness of the focusing cubic NLS in the sense that the solution map fails to be locally uniformly continuous for $s<0$. They used the Galilean invariance on the soliton solutions to obtain mild ill-posedness result. Recall that if $u(t,x)$ is a solution of the cubic NLS with initial data $u_0$, then $u_N(t,x)=e^{iNx}e^{-itN^2}u(t,x-2tN)$ 
 	is another solution with initial data $u_N(0,x)=e^{iNx}u_0(x)$. Observe that for $s<0$, we have
 	\begin{align}\label{eqn:e^{iNx} to zero}
 	\lim\limits_{N\to \infty}\Vert u_N(0,x) \Vert_{H^s\left(\mathbb{R}\right)}=0.
 	\end{align}
 	By applying the Galilean symmetry to the soliton soution $u$, we consider the solutions $u_{N_1},u_{N_2}$. By using $\eqref{eqn:e^{iNx} to zero}$,  we can make $\Vert u_{N_1}(0,\cdot)-u_{N_2}(0,\cdot)\Vert_{H_x^s}$ sufficiently small by choosing $N_1$ and $N_2$ large enough. Notice that $u_{N_1},u_{N_2}$ move with different speeds. Therefore, the difference $\Vert u_{N_1}(t)-u_{N_2}(t)\Vert_{H_x^s}$ is bounded below by some fixed constant.
 \end{remark} 
 
 \subsection{Approximate solution}
 In the following, we only consider defocusing \eqref{eqn:fourth order NLS}. 
 The same argument can be applied in the focusing case.
 First, we find the approximate solution to defocusing \eqref{eqn:fourth order NLS} by using the solution of the cubic NLS
 \begin{align}\label{eqn: defocusing cubic NLS}
 i\partial_tu=\partial_x^2u-\vert u \vert^2u.
 \end{align}
 Assume that $u(s,y)$ solves \eqref{eqn: defocusing cubic NLS}. We use the following change of variable
 \begin{align*}
 \left(s,y\right):=\left(t,\frac{x+4N^3t}{\sqrt{6}N} \right).
 \end{align*}
 Then, we define the approximate solution
 \begin{align*}
 U_{ap}(t,x):=e^{iN^4t}e^{iNx}u(s,y),
 \end{align*}
 where $N \gg 1$ will be chosen later.
 
 We want to show that $U_{ap}$ is an approximate solution to the defocusing \eqref{eqn:fourth order NLS}. A straightforward calculation shows that
 \begin{align*}
 i\partial_tU_{ap}(t,x)=&ie^{iNx}\left(iN^4\right)e^{iN^4t}u\left(t,\frac{x}{\sqrt{6}N}+\frac{4N^2}{\sqrt{6}}t\right)+ie^{iNx}e^{iN^4t}\partial_su\left(t, \frac{x}{\sqrt{6}N}+\frac{4N^2}{\sqrt{6}}t\right)\\
 +&ie^{iNx}e^{iN^4t}\partial_yu\left(t, \frac{x}{\sqrt{6}N}+\frac{4N^2}{\sqrt{6}}t  \right)\frac{4N^2}{\sqrt{6}},\\
 \partial_x^4U_{ap}(t,x)=&\left(iN\right)^4e^{iNx}e^{iN^4t}u(s,y)+e^{iNx}e^{iN^4t}\partial_y^4u(s,y)\left(\frac{1}{\sqrt{6}N}\right)^4\\
 +& {4 \choose 1}\left(iN\right)e^{iNx}e^{iN^4t}\partial_y^3u(s,y)\left(\frac{1}{\sqrt{6}N} \right)^3 + {4\choose 2}\left( iN\right)^2e^{iNx}e^{iN^4t}\partial_y^2u(s,y)\left(\frac{1}{\sqrt{6}N} \right)^2\\
 +&{4\choose3}\left(iN\right)^3e^{iNx}e^{iN^4t}\partial_yu(s,y)\frac{1}{\sqrt{6}N}
 \end{align*}
 and hence we have                
 \begin{align*}
 \left(i\partial_t-\partial_x^4 \right)U_{ap}(t,x)=&e^{iNx}e^{iN^4t}\left(i\partial_su(s,y)-\partial_y^2u(s,y)  \right)\\
 +&\frac{1}{36}N^{-4}e^{iNx}e^{iN^4t}\partial_y^4u(s,y)+\frac{4i}{6^{\frac{3}{2}}}N^{-2}e^{iNx}e^{iN^4t}\partial_y^3u(s,y).
 \end{align*}
 Moreover, we note that 
 \begin{align*}  
 \vert U_{ap}(t,x) \vert^2 U_{ap}(t,x)=&\vert u (s,y) \vert^2 e^{iNx}e^{iN^4t}u(s,y).
 \end{align*}
 Since $u$ is a solution of $\eqref{eqn: defocusing cubic NLS}$, we obtain 
 \begin{align*}
 &\left(i\partial_t-\partial_x^4 \right)U_{ap}(t,x)-\vert U_{ap}(t,x) \vert^2 U_{ap}(t,x)\\
 &=e^{iNx}e^{iN^4t}\left(i\partial_su(s,y)-\partial_y^2u(s,y)  +\vert u(s,y) \vert^2 u(s,y) \right)\\
 &+\frac{1}{36}N^{-4}e^{iNx}e^{iN^4t}\partial_y^4u(s,y)+\frac{4i}{6^{\frac{3}{2}}}N^{-2}e^{iNx}e^{iN^4t}\partial_y^3u(s,y)\\
 &=\frac{1}{36}N^{-4}e^{iNx}e^{iN^4t}\partial_y^4u(s,y)+\frac{4i}{6^{\frac{3}{2}}}N^{-2}e^{iNx}e^{iN^4t}\partial_y^3u(s,y)\\
 &=E, 
 \end{align*}
 where the error term $E$ is a linear combination of the expressions
 \begin{align*}
 E_1:=N^{-4}e^{iNx}e^{iN^4t}\partial_y^4u(s,y),\\
 E_2:=N^{-2}e^{iNx}e^{iN^4t}\partial_y^3u(s,y).
 \end{align*}
 
 \subsection{Error estimate}
 In this subsection, we prove the following error estimates for $E_1$ and $E_2$.
 \begin{lemma}\label{lem : ill-posedness error estimate}
 	For each $j=1,2$, let $e_j$ be the solution to the initial value problem
 	\begin{align*}
 	\left(\partial_t-\partial_x^4\right)e_j=E_j, \quad e_j(0)=0.
 	\end{align*}	
 	Let $\eta(t)$ be a smooth time cut-off function taking value 1 near the origin and compactly supported. Then, we have
 	\begin{align*}
 	\Vert \eta(t) e_j \Vert_{X^{-\frac{1}{2},b}}\lesssim \varepsilon N^{-2}.
 	\end{align*}	
 \end{lemma}
 For the proof of the above lemma, we need the following lemma.
 \begin{lemma}[\cite{CCT2003}]\label{lem: ill-posedness modulation}
 	Let $\sigma \in \mathbb{R}^{+}$ and $u\in H^{\sigma}\left(\mathbb{R}\right)$. For any $M >1,\tau\in \mathbb{R}^{+},x_0\in \mathbb{R}$ and $A>0$ let
 	\begin{align*}
 	v(x)=Ae^{iMx}u(\left(x-x_0\right)/\tau).
 	\end{align*}	
 	\textup{(i)} Suppose $s\geq 0$. Then, there exists a constant $C_1<\infty $,depending only on $s$, such that whenver $M\tau \geq 1 $,
 	\begin{align*}
 	\Vert v \Vert_{H^s}\leq C_1 \vert A \vert \tau^{\frac{1}{2}}M^s\Vert u \Vert_{H^s}
 	\end{align*}
 	for all $u,A,x_0$.
 	
 	\noindent \textup{(ii)} Suppose that $s<0$ and that $\sigma \geq \abs{s}$. Then there exists a constant $C_1<\infty$ depending only on $s$ and on $\sigma$, such that whenver $1\leq \tau \cdot M^{1+\left(s/\sigma \right)}$,
 	\begin{align*}
 	\Vert v \Vert_{H^s}\leq C_1 \abs{A}\tau^{\frac{1}{2}}M^s\Vert u \Vert_{H^{\sigma}}
 	\end{align*}
 	for all $u,A,x_0$.
 	
 	\noindent\textup{(iii)} There exists $c_1>0$ such that for each $u$, there exists $C_u<\infty $ such that 
 	\begin{align*}
 	\Vert v \Vert_{H^s}\geq c_1 \abs{A}\tau^{\frac{1}{2}}M^s\Vert u \Vert_{L^2}
 	\end{align*}
 	whenever $\tau \cdot M \geq C_u$.	
 \end{lemma}
 \begin{proof}
 	Observe that
 	\begin{align*}
 	A^{-2}\tau^{-1}M^{-2s}\Vert v \Vert_{H^s}^2=&c\tau^{-1}M^{-2s}\int_{\mathbb{R}} \left(1+\abs{\xi}^2\right)^s \tau^2 \vert \widehat{u}\left(\tau\left(\xi-M\right)\right)^2\,d\xi\\
 	=&c \int_{\mathbb{R}} \left(\frac{\tau^2+\abs{M\tau+\eta}^2 }{\tau^2M^2} \right)^s \abs{\widehat{u}(\eta)}^2 \,d\eta\\
 	\lesssim& \int_{\abs{\eta}\leq\tau M/2} \vert \widehat{u}(\eta) \vert^2 +\int_{ \tau M/2  \leq \abs{\eta} \leq  2\tau M }M^{-2s} \vert \widehat{u}(\eta) \vert^2\\
 	&\hspace{20mm}+\int_{\abs{\eta}\geq 2\tau M }\frac{\abs{\eta}^{2s}}{\left(\tau M \right)^{2s}}\vert \widehat{u}(\eta) \vert^2\\
 	=&I+II+III.
 	\end{align*}
 	
 	Term $I$ is $\lesssim \Vert u \Vert_{L^2}^2$. If $s\geq 0 $, then $M^{-2s}\leq 1 $, so $II \lesssim \Vert u \Vert_{L^2}^2$ and $III\lesssim \Vert u \Vert_{H^s}^2$, because $\tau M \geq 1 $.
 	
 	If  $s<0$, then $III \lesssim \Vert u \Vert_{L^2}^2$, since $\abs{\eta}/\tau M \gtrsim 1 $. Moreover, $II \lesssim M^{-2s}\left(\tau M \right)^{-2\sigma}\Vert u \Vert_{H^\sigma}^2$, which is $\lesssim \Vert u \Vert_{H^\sigma}^2$ under the further hypothesis $1\leq \tau \cdot M^{1+\left(s/\sigma \right)}$.	
 	
 	To obtain $(III)$, it suffices to consider term $I$ : for any $u$, $\int_{\abs{\eta} \leq \tau M/2 }\vert \widehat{u}(\eta) \vert^2 $ approches c$\Vert u \Vert_{L^2}^2$ as $\tau M \to \infty.$
 \end{proof}
 Now we are ready to prove the error estimates.
 \begin{proof}
 	For $\frac{1}{2}<b<1$, from Lemma \ref{non} and $e_j(0)=0$, we have 
 	\begin{align*}
 	\Vert \eta(t) e_j \Vert_{X^{-\frac{1}{2},b}}\lesssim& \Vert \eta(t) E_j \Vert_{X^{-\frac{1}{2},b-1}}\\
 	=&\Vert \langle \tau+\xi^4 \rangle^{b-1} \langle \xi \rangle^{-\frac{1}{2}} \widetilde{\eta(t)E_j} \Vert_{L_{\tau,\xi}^2}\\
 	\leq& \Vert \langle \xi \rangle^{-\frac{1}{2}} \widetilde{\eta(t)E_j}  \Vert_{L_{\tau,\xi}^2}\\
 	=&\Vert \eta(t) \langle \xi \rangle^{-\frac{1}{2}}\mathcal{F}_xE_j(t,\xi) \Vert_{L_{t,\xi}^2}\\
 	\lesssim& \Vert \langle \xi \rangle^{-\frac{1}{2}}\mathcal{F}_xE_j(t,\xi) \Vert_{L_t^\infty L_{\xi}^2\left([0,1]\times\mathbb{R}\right)}.
 	\end{align*}
 	Hence, it suffices to show that 
 	\begin{align*}
 	\sup\limits_{0\leq t \leq 1 } \Vert E_j(t) \Vert_{H_x^{-\frac{1}{2}}}\lesssim \varepsilon N^{-2}.
 	\end{align*}
 	
 	The above error estimate then follows by Lemma $\ref{lem: ill-posedness modulation}$, the fact that $H^s$ is a Banach algebra for $s>\frac{1}{2}$, and $\eqref{eqn: ill-posedness H_x^5 bound}$. More precisely, we apply Lemma $\ref{lem: ill-posedness modulation}$ with $M=N,\tau=N$ and $A=N^{-4}$ or $N^{-2}$.
 	\begin{align*}
 	&\Vert E_1(t) \Vert_{H_x^{-\frac{1}{2}}}\lesssim N^{-4} N^{\frac{1}{2}}N^{-\frac{1}{2}}\Vert u(t) \Vert_{H_x^K},\\
 	&\Vert E_2(t) \Vert_{H_x^{-\frac{1}{2}}}\lesssim N^{-2} N^{\frac{1}{2}}N^{-\frac{1}{2}}\Vert u(t) \Vert_{H_x^K}
 	\end{align*}
 	for all $t\in \mathbb{R}$.
 \end{proof} 
 
 \subsection{Perturbation lemma}
 In this subsection,
 we prove the following perturbation lemma. 
 \begin{lemma}\label{lem : ill-posedness perturbation lem}
 	Let $u$ be a smooth solution to $\eqref{eqn:fourth order NLS}$ and $v$ be a Schwartz solution to the approximate fourth order NLS equation
 	\begin{align*}
 	i\partial_tv-\partial_x^4v-\abs{v}^2u=E
 	\end{align*}	
 	for some error function $E$. Let $e$ be the solution to the inhomogeneous problem
 	\begin{align*}
 	i\partial_t e -\partial_x^4e=E, \quad e(0)=0.
 	\end{align*}	
 	Suppose that 
 	\begin{align*}
 	\Vert u(0) \Vert_{H_x^{-\frac{1}{2}}}, \Vert v(0) \Vert_{H_x^{-\frac{1}{2}}}\lesssim \varepsilon, \quad \text{and} \quad \Vert \eta(t)e \Vert_{X^{-\frac{1}{2},b}}\lesssim \varepsilon.
 	\end{align*}
 	Then we have
 	\begin{align*}
 	\Vert \eta(t)(u-v)\Vert_{X^{-\frac{1}{2},b}}\lesssim \Vert u(0)-v(0) \Vert_{H^{-\frac{1}{2}}}+\Vert \eta(t) e \Vert_{X^{-\frac{1}{2},b}}.  
 	\end{align*}
 	In particular, we have 
 	\begin{align*}
 	\sup\limits_{0\leq t \leq 1 } \Vert u(t)-v(t) \Vert_{H^{-\frac{1}{2}}}\lesssim \Vert u(0)-v(0) \Vert_{H^{-\frac{1}{2}}}+\Vert \eta(t) e \Vert_{X^{-\frac{1}{2},b}}.
 	\end{align*}
 \end{lemma}
 \begin{proof}
 	The proof is a standard perturbation argument. See \cite[Lemma 5.1]{CCT2003}. We give only a sketch. We write the Duhamel formula for $v$ with a time cut-off function $\eta(t)$
 	\begin{align*}
 	\eta(t)v(t)=\eta(t)e^{it\partial_x^4}u(0)-\eta(t)e(t)+\eta(t)\int_0^te^{i(t-t')\partial_x^4}\vert v \vert^2v(t')\,dt'.
 	\end{align*}
 	We use $\eqref{eqn:X^{s,b} energy estimate}$,$\eqref{eqn:trilinear estimate}$ and a continuity argument assuming that $\varepsilon$ is very small. Then, we obtain
 	\begin{align*}
 	\Vert \eta(t) \Vert_{X^{-\frac{1}{2},b }}\lesssim \varepsilon.
 	\end{align*}
 	We repeat the same argument on the difference $w=u-v$ to get the desired result.
 \end{proof}
 
 Let $U$ be the global Schwartz solution to the fourth order NLS equation $\eqref{eqn:fourth order NLS}$ with initial datum $U(0,\cdot)=U_{ap}(0,\cdot)$. By applying the above two Lemma $\ref{lem : ill-posedness error estimate}$ and $\ref{lem : ill-posedness perturbation lem}$, we obtain 
 \begin{align*}
 \sup\limits_{k\leq t \leq k+1} \Vert U(t)-U_{ap}(t)\Vert_{H^{-\frac{1}{2}}} \lesssim& \Vert U(k)-U_{ap}(k)\Vert_{H^{-\frac{1}{2}}}+\varepsilon N^{-2}\\
 \lesssim & \sup\limits_{k-1\leq t \leq k} \Vert U(t)-U_{ap}(t) \Vert_{H^{-\frac{1}{2}}}+\varepsilon N^{-2}.
 \end{align*}
 Therefore, by applying the induction on $k$ for $k\lesssim \log N$ and using $U(0)=U_{ap}(0)$, we conclude that for any $\eta>0$
 \begin{align}
 \sup\limits_{0\leq t \lesssim \text{log}N } \Vert U(t)-U_{ap}(t) \Vert_{H_x^{-\frac{1}{2}}}\lesssim \varepsilon N^{-2+\eta},
 \end{align}
 uniformly for all $N\gg 1 $. 
 
 \subsection{Proof of Theorem $\ref{thm: ill-posedness}$}
 In this subsection, we prove Theorem \ref{thm: ill-posedness}. We follow the argument in \cite{CCT2003}.
 Before we start proving the theorem, let us recall the following.  
 In \cite[(3.16),(3.19),(3.20)]{CCT2003}, Christ-Colliander-Tao constructed the global solution $u^{\langle aw \rangle}$ to \eqref{eqn: cubic NLS,focusing defocusing} for all $a\in [1/2,2]$, where $w(x)=\varepsilon \exp(-x^2)$ for some parameter $0<\varepsilon \ll 1$, such that 
 \begin{align}
 \sup\limits_{0\leq t <\infty}\Vert u^{\langle aw \rangle}(t) \Vert_{H_x^K}&\lesssim \varepsilon,\label{eqn: ill-posedness H_x^5 bound} \\
 \Vert u^{\langle aw \rangle}(0)-u^{\langle a'w \rangle}(0) \Vert_{H_x^K}&\lesssim \varepsilon \vert a-a' \vert,\label{eqn: ill-posedness H_x^5 difference} \\
 \limsup\limits_{t \to +\infty} \Vert u^{\langle aw \rangle}(t)-u^{\langle a'w \rangle}(t) \Vert_{L_x^2}&\gtrsim 1. \label{eqn: ill-posedness limsup L^2}
 \end{align}
 for sufficiently large integer $K\geq 5$.
 The details of solutions $u^{\langle aw \rangle}$ are presented in \cite[Section 3]{CCT2003}.
 The Galilean invariances and scale symmetry can be used to transform such solutions $u^{\langle aw \rangle}$ into solutions $u_1,u_2$ in Theorem \ref{thm: ill-posedness defocsuing NLS}. They also used these solutions $u^{\langle aw \rangle}$ to establish approximate solution of the mKdV equation to show mild ill-posedness of the mKdV equation. In the proof of Theorem \ref{thm: ill-posedness}, we also use these solutions $u^{\langle aw \rangle}$.

 \begin{proof}[Proof of Theorem $\ref{thm: ill-posedness}$]

 	Let $0<\delta \ll \varepsilon \ll 1$ and $T>0$ be given. Note that we have two global solutions $u_1,u_2$ of defocusing NLS $\eqref{eqn: defocusing cubic NLS}$ satisfying $\eqref{eqn: ill-posedness H_x^5 bound}$, $\eqref{eqn: ill-posedness H_x^5 difference}$ and $\eqref{eqn: ill-posedness limsup L^2}$.
 	
 	Define $U_{ap,1}$ and $U_{ap,2}$ by
 	\begin{align*}
 	&U_{ap,1}(t,x):=e^{iN^4t}e^{iNx}u_{1}(s,y),\\
 	&U_{ap,2}(t,x):=e^{iN^4t}e^{iNx}u_{2}(s,y)
 	\end{align*}
 	and let $U_1,U_2$ be global Schwartz solutions with initial data $U_{ap,1}(0,\cdot), U_{ap,2}(0,\cdot)$, respectively. Let $\lambda \gg 1 $ be a large parameter to be chosen later. Let $U_{j}^{\lambda}$ denote the function
 	\begin{align*}
 	U_{j}^{\lambda}(t,x):=\lambda^2U_{j}\left(\lambda^4t,\lambda x \right).
 	\end{align*}
 	Since $U_j$ is a global solution to fourth order NLS, so is $U_j^\lambda$. Similarly, we define 
 	\begin{align*}
 	U_{ap,j}^{\lambda}(t,x):=\lambda^2 U_{ap,j}(\lambda^4t,\lambda x )
 	\end{align*}
 	for $j=1,2$. Note that 
 	\begin{align*}
 	U_{j}^\lambda(0,x)=U_{ap,j}^{\lambda}(0,x)=\lambda^2 U_{ap,j}(0,\lambda x )=\lambda^2 e^{iN\lambda x}u_j(0,\frac{\lambda x}{ \sqrt{6}N} ).
 	\end{align*}
 	
 	\noindent
 	From Lemma \ref{lem: ill-posedness modulation}, we have
 	\begin{align*}
 	\Vert U_j^{\lambda}(0) \Vert_{H^s} \lesssim \lambda^{s+\frac{3}{2}}N^{s+\frac{1}{2}}\Vert u_j(0)\Vert_{H_x^K}
 	\end{align*}
 	provided that its hypothesis $1\leq \tau M^{1+(s/K)}$ is satisfied with $A=\lambda^2, M=N\lambda$, and $\tau=N/\lambda$. We define $\lambda$ by 
 	\begin{align}
 	\label{lambda}
 	\lambda:= N^{-\frac{s+\frac{1}{2}}{s+\frac{3}{2}}}.
 	\end{align}
 	The condition $1\leq \tau M^{1+(s/K)}$ then becomes $1\leq N^{\frac{2s+3}{s+\frac{3}{2}}}\cdot \left(N^{\frac{1}{s+3/2 }} \right)^{s/K}$. Hence, for any $s>-3/2$ this condition is satisfied with sufficiently large $N,K$.

 	Then, from $\eqref{eqn: ill-posedness H_x^5 bound}$, we have $\Vert U_j^\lambda(0) \Vert_{H_x^s}\lesssim \varepsilon $ for $j=1,2$. Similarly by using $\eqref{eqn: ill-posedness H_x^5 difference}$ instead of $\eqref{eqn: ill-posedness H_x^5 bound}$, we have $\Vert U_1^{\lambda}(0)-U_2^{\lambda}(0) \Vert _{H_x^s}\lesssim \delta $.
 	
 	Now we show $\sup\limits_{0\leq t \leq T } \Vert U_1^\lambda(t)-U_2^\lambda(t) \Vert_{H_x^s}\gtrsim \varepsilon $. Recall that we have
 	\begin{align}\label{eqn: ill-posedness U(t)-U_{ap}(t)}
 	\sup\limits_{0\leq t \lesssim \text{log}N } \Vert U_j(t)-U_{ap,j}(t) \Vert_{H_x^{-\frac{1}{2}}}\lesssim \varepsilon N^{-2+\eta}
 	\end{align}
 	A scaling calcuation shows
 	\begin{align*}
 	\Vert U_j^\lambda(t)-U_{ap,j}^\lambda(t) \Vert_{H_x^{s}}\lesssim& \lambda^{\text{max}(0,s)+\frac{3}{2}} \Vert  U_j(\lambda^4 t)-U_{ap,j}(\lambda^4 t)  \Vert_{H_x^s}\\
 	\lesssim& \lambda^{\text{max}(0,s)+\frac{3}{2}} \Vert  U_j(\lambda^4 t)-U_{ap,j}(\lambda^4 t)  \Vert_{H_x^{-\frac{1}{2}}}\\
 	\lesssim& \lambda^{\text{max}(0,s)+\frac{3}{2}}\varepsilon N^{-2+\eta}
 	\end{align*}
 	for $0< \lambda^4 t\lesssim \text{log}N$. By using $(\ref{lambda})$ and $s>-\frac{15}{14}$, we have
 	\begin{align}\label{eqn: ill-posedness triangle inequality 1}
 	\Vert U_j^\lambda(t)-U_{ap,j}^\lambda(t) \Vert_{H_x^{s}}\ll \varepsilon
 	\end{align}
 	for sufficiently small $\eta>0$ and large $N\gg 1$. From $\eqref{eqn: ill-posedness limsup L^2}$, we can find $t_0$ such that 
 	\begin{align}\label{eqn: ill-posedness limsup L^2 lower bound u_1, u_2}
 	\Vert u_1(t_0)-u_2(t_0)  \Vert_{L_x^2}\gtrsim \varepsilon.
 	\end{align}
 	Fix this $t_0.$ We choose $N$ so large that $t_0 \ll \text{log}N$. Using Lemma \ref{lem: ill-posedness modulation} and $\eqref{eqn: ill-posedness limsup L^2 lower bound u_1, u_2}$ as before, we have
 	\begin{align}\label{eqn : ill-posedness triangle inequality 2}
 	\Vert U_{ap,1}^{\lambda}(t_0/\lambda^4)-U_{ap,2}^\lambda(t_0/\lambda^4) \Vert_{H_x^s} \gtrsim \lambda^{s+\frac{3}{2}}N^{s+\frac{1}{2}}\Vert u_1(t_0)-u_2(t_0)  \Vert_{L_x^2}\gtrsim \varepsilon.
 	\end{align}
 	
 	\noindent
 	Hence, from \eqref{eqn: ill-posedness triangle inequality 1}, \eqref{eqn : ill-posedness triangle inequality 2}, and triangle inequality, we have 
 	\begin{align*}
 	\Vert U_{1}^{\lambda}(t_0/\lambda^4)-U_{2}^\lambda(t_0/\lambda^4) \Vert_{H_x^s}\gtrsim \varepsilon.
 	\end{align*}
 	By choosing $N$(and hence $\lambda$) large enough that $t_0/\lambda^4<T$, we obtain the desired result.
 	
 \end{proof}

\end{document}